\documentclass[a4paper]{amsart}

\usepackage{amssymb}
\usepackage[arrow]{xy}
\usepackage{pb-diagram,pb-xy}
\usepackage{graphicx}

\theoremstyle{plain}
\newtheorem{theorem}{Theorem}[section]
\newtheorem{proposition}[theorem]{Proposition}
\newtheorem{corollary}[theorem]{Corollary}
\newtheorem{lemma}[theorem]{Lemma}

\newtheorem{assertion}{Assertion}
\theoremstyle{definition}
\newtheorem{definition}[theorem]{Definition}

\newtheorem{remark}[theorem]{Remark}

\def\Z{\mathbb{Z}}
\def\Q{\mathbb{Q}}
\def\R{\mathbb{R}}
\def\C{\mathbb{C}}

\def\p{\mathfrak{p}}
\def\F{\mathcal{F}}
\def\G{\mathcal{G}}
\def\BF{\mathcal{BF}}

\def\K{\mathcal{K}}

\def\M{\mathcal{M}}

\def\Ker{\operatorname{Ker}}
\def\Coker{\operatorname{Coker}}

\def\Hom{\operatorname{Hom}}
\def\sign{\operatorname{sign}}
\def\rank{\operatorname{rank}}
\def\dis{\operatorname{dis}}
\def\inte{\operatorname{int}}
\def\Arf{\mathrm{Arf}}
\def\id{\mathrm{id}}

\def\to{\mathchoice{\longrightarrow}{\rightarrow}{\rightarrow}{\rightarrow}}
\makeatletter
\newcommand{\shortxra}[2][]{\ext@arrow 0359\rightarrowfill@{#1}{#2}}
\def\longrightarrowfill@{\arrowfill@\relbar\relbar\longrightarrow}
\newcommand{\longxra}[2][]{\ext@arrow 0359\longrightarrowfill@{#1}{#2}}
\renewcommand{\xrightarrow}[2][]{\mathchoice{\longxra[#1]{#2}}%
  {\shortxra[#1]{#2}}{\shortxra[#1]{#2}}{\shortxra[#1]{#2}}}
\makeatother

\usepackage[flushmargin]{footmisc}
\makeatletter
\def\@addressinfootnote{%
    \begingroup
    \def\author##1{##1:}%
    \def\\{, }%
    \def\address##1##2{
      {##2}%
    }%
    \def\email##1##2{%
      \@ifnotempty{##2}{, e-mail: \ignorespaces{\ttfamily##2}}}%
    \def\curraddr##1##2{}%
    \def\urladdr##1##2{}%
    \addresses
    \endgroup
    \nolinebreak\vrule width 0mm depth 2ex
}
\def\@adminfootnotes{%
  \let\@makefnmark\relax  \let\@thefnmark\relax
  \ifx\@empty\authors\else \@footnotetext{\@addressinfootnote}\fi
  \ifx\@empty\@date\else \@footnotetext{\@setdate}\fi
  \ifx\@empty\@subjclass\else \@footnotetext{\@setsubjclass}\fi
  \ifx\@empty\@keywords\else \@footnotetext{\@setkeywords}\fi
  \ifx\@empty\thankses\else \@footnotetext{%
    \def\par{\let\par\@par}\@setthanks}%
  \fi
}
\renewcommand{\author}[2][]{%
  \ifx\@empty\authors
    \gdef\authors{#2}%
  \else
    \g@addto@macro\authors{\and#2}%
  \fi
  \g@addto@macro\addresses{\author{#2}}%
  \@ifnotempty{#1}{%
    \ifx\@empty\shortauthors
      \gdef\shortauthors{#1}%
    \else
      \g@addto@macro\shortauthors{\and#1}%
    \fi
  }%
}
\def\enddoc@text{
}
\makeatother

\makeatletter
\def\Nopagebreak{\@nobreaktrue\nopagebreak}
\makeatother

\begin{document}

\title%
[Hirzebruch-type defects from iterated $p$-covers]%
{Link concordance, homology cobordism, and Hirzebruch-type
  defects from iterated $p$-covers}

\author{Jae Choon Cha}

\address{Department of Mathematics and Pohang Mathematics Institute, Pohang University of Science and Technology\\
  Pohang, Gyungbuk 790--784, Republic of Korea}

\email{jccha@postech.ac.kr}

\def\subjclassname{\textup{2000} Mathematics Subject Classification}
\expandafter\let\csname subjclassname@1991\endcsname=\subjclassname
\expandafter\let\csname subjclassname@2000\endcsname=\subjclassname
\subjclass{%
57M25, 
57M27, 
57N70. 
}

\keywords{Link concordance, Homology cobordism, Iterated $p$-cover,
  Intersection form defect, Bing double}

\begin{abstract}
  We obtain new invariants of topological link concordance and
  homology cobordism of 3-manifolds from Hirzebruch-type intersection
  form defects of towers of iterated $p$-covers.  Our invariants can
  extract geometric information from an arbitrary depth of the derived
  series of the fundamental group, and can detect torsion which is
  invisible via signature invariants.  Applications illustrating these
  features include the following: (1) There are infinitely many
  homology equivalent rational 3-spheres which are indistinguishable
  via multisignatures, $\eta$-invariants, and $L^2$-signatures but
  have distinct homology cobordism types.  (2) There is an infinite
  family of 2-torsion (amphichiral) knots with non-slice iterated Bing
  doubles; as a special case, we give the first proof of the conjecture
  that the Bing double of the figure eight knot is not slice.  (3)
  There exist infinitely many torsion elements at any depth of the
  Cochran-Orr-Teichner filtration of link concordance.
\end{abstract}

\maketitle

%


\section{Introduction and main results}

In this paper we define invariants of 3-manifolds and links in $S^3$
from towers of iterated $p$-covers, where $p$ is prime, and employ the
invariants to study homology cobordism and link concordance.
Essentially our invariants are defects of the (Witt classes of)
twisted intersection forms of topological 4-manifolds bounded by the
covers.  The invariant has two remarkable merits: (i)~it can extract
geometric information from an arbitrary depth of the derived series of
the fundamental group, and (ii)~it can detect ``torsion'', which is
invisible via signature invariants.  Also, in many interesting cases,
the invariant can be computed by combinatorial algorithms as
illustrated in our applications.

The above properties of our invariants may be discussed from the
viewpoint of a geometric technique producing new 3-manifolds and
links, which is often referred to as ``tying a knot'', ``satellite
construction'', or ``infection''.
Figure~\ref{figure:infection-example} illustrates infection on a link
$L\subset S^3$ by the figure eight knot, which gives us the
twice-iterated Bing double: given an unknotted circle $\alpha\subset
S^3$ disjoint from $L$, by tying the figure eight knot along a 2-disk
bounded by $\alpha$, we obtain a new link.  Or alternatively, the new
link is obtained from $L$ by filling in the exterior of $\alpha$ with
the figure eight knot exterior.  Infection on a 3-manifold is
defined in a similar way.  (For a precise definition, refer to
Section~\ref{subsection:knot-infection}.)

\begin{figure}[ht]
  \begin{center}
    \includegraphics[scale=.9]{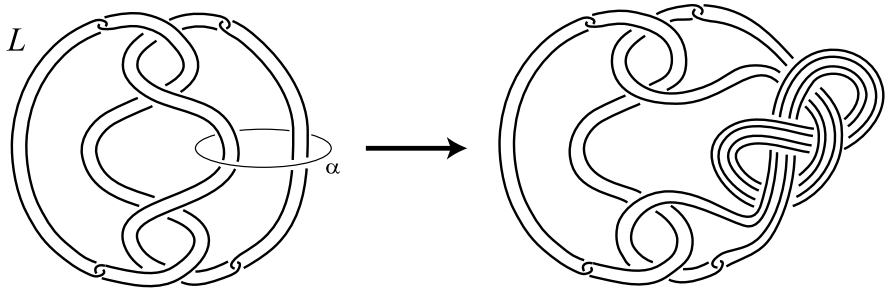}
    \caption{Infection on a link $L$ by the figure eight knot}
    \label{figure:infection-example}
  \end{center}
\end{figure}

Roughly speaking, if $\alpha$ lies in the $n$th derived subgroup of
the fundamental group, then by infection along $\alpha$, further
complication from the group of the infection knot appears in the $n$th
derived subgroup.  Based on this property, in many papers infection
has been used as a primary source of examples revealing the structure
of knot and link concordance.  In particular, in their landmark
paper~\cite{Cochran-Orr-Teichner:1999-1} and subsequent works
\cite{Cochran-Orr-Teichner:2002-1,Cochran-Teichner:2003-1}, Cochran,
Orr, and Teichner gave a new framework of a systematic study of
concordance of knots and links in terms of a filtration related to the
derived series, and used infection as a realization tool to illustrate
the rich contents of their theory.  (For definitions of concordance
and the Cochran-Orr-Teichner filtration, see
Section~\ref{section:witt-class-defects-of-links}
and~\ref{section:obstruction-to-solvability}.)  There are several
recent related results using infection, including works of Cochran,
Taehee Kim, Harvey, Leidy, and the author \cite{Cochran:2002-1,
  Cochran-Kim:2004-1, Kim:2004-1, Kim:2005-1, Kim:2006-1,
  Harvey:2006-1, Cochran-Harvey:2006-01, Cochran-Harvey-Leidy:2006-1,
  Cha:2006-1}.  Such reams of results leads us to study how to detect
the effect of infection along a curve contained in a higher term of
the derived series of the fundamental group, up to concordance and
homology cobordism.

Our result provides a new method to detect the effect of infection.
Our method is effective even when the infection knot $K$ is
\emph{torsion}, that is, of finite order in the knot concordance
group.  So far, the von Neumann-Cheeger-Gromov $L^2$-signature
invariants have been used as the only available tool to detect
infection when the infection curve $\alpha$ is in a higher term of the
derived series.  Roughly speaking, $L^2$-signatures of the infected
manifold or link associated to certain solvable coefficient systems
reflect $L^2$-signatures of the infection knot~$K$ associated to much
simpler coefficient systems (e.g., abelian or metabelian ones).  All
recent works mentioned above \cite{Cochran-Orr-Teichner:1999-1,
  Cochran-Orr-Teichner:2002-1, Cochran-Teichner:2003-1,
  Cochran:2002-1, Cochran-Kim:2004-1, Kim:2004-1, Kim:2005-1,
  Kim:2006-1, Harvey:2006-1, Cochran-Harvey:2006-01,
  Cochran-Harvey-Leidy:2006-1, Cha:2006-1} depend on results of this
type.
%
When the infection knot $K$ is torsion, however, those
$L^2$-signatures have failed to detect anything.  Partly because of
this limitation, many questions on torsion still remain unsolved.  Our
invariants can detect torsion in many interesting cases, as
illustrated by the following applications:

\begin{enumerate}
\item There are homology equivalent rational homology 3-spheres which
  have vanishing previously known signature invariants but are not
  homology cobordant.
\item There are infinitely many 2-torsion (amphichiral) knots whose
  iterated Bing doubles are not slice; as a special case, we give the
  first proof of the conjecture that the Bing double of the figure
  eight knot is not slice.
\item There exist infinitely many 2-torsion elements in an arbitrarily
  deep level of the Cochran-Orr-Teichner filtration of link
  concordance.
\end{enumerate}
Before discussing more about the applications, we start with an
overview of our invariants.  In this paper all manifolds are oriented
topological manifolds, and submanifolds are assumed to be locally
flat.

\subsection{Intersection form defects and homology cobordism}

Our invariant is basically an intersection form analogue of
Hirzebruch-type signature defects of odd dimensional closed manifolds.
Let $\Gamma$ be a group with $H_4(\Gamma)=0$ and fix a homomorphism of
the integral group ring $\Z\Gamma$ into a (skew-)field $\K$ with
involution.  Suppose $M$ is a closed 3-manifold endowed with a group
homomorphism $\phi\colon \pi_1(M) \to \Gamma$ such that $(M,\phi)$ is
null-bordant over $\Gamma$, i.e., $(M,\phi)=0$ in the bordism group
$\Omega^{top}_3(B\Gamma)$.  Roughly speaking, we define an invariant
$\lambda(M,\phi)$ to be the difference of the Witt class of
``$\K$-coefficient intersection form'' of a topological 4-manifold
bounded by $M$ over $\phi$ and the Witt class of its ordinary
intersection form.  The value of $\lambda(M,\phi)$ lives in the Witt
group $L^0(\K)$ of nonsingular hermitian forms over~$\K$.  We prove
that $\lambda(M,\phi)$ is well-defined along the lines of a standard
bordism approach, appealing to an Atiyah-type result
(Lemma~\ref{lemma:atiyah-type-result}) on the symmetric signatures of
Mishchenko-Ranicki.  We remark that $\lambda(M,\phi)$ can also be
defined, as an element of $S^{-1}\Z\otimes_\Z L^0(\K)$, under an
weaker assumption that $(M,\phi)$ is $S$-torsion in
$\Omega^{top}_4(B\Gamma)$ for a multiplicatively closed subset $S$
of~$\Z$.  To simplify statements, in this section we will always
consider the case that $S=\{1\}$.  For the general case, see
Section~\ref{section:Witt-class-defect}.

In order to extract homology cobordism invariants, we consider towers
of abelian $p$-covers.  Fix a prime~$p$ and let $M_0$ be~$M$.
Inductively choosing surjections $\phi_i \colon \pi_1(M_i) \to
\Gamma_i$ with $\Gamma_i$ an abelian $p$-group (that is, $\Gamma_i$ is
abelian and $|\Gamma_i|=p^a$), a tower
\[
M_n \to M_{n-1} \to \cdots \to M_0=M
\]
of connected $p$-covers of $M$ is constructed.  We call it a
\emph{$p$-tower} determined by $\{\phi_i\}$.  For
$\phi_n\colon\pi_1(M_n) \to \Gamma_n$ with $\Gamma_n = \Z_{d}$ for
some $d=p^a$, we consider the invariant $\lambda(M_n,\phi_n)$.  Here,
in order to define the invariant, $\Z_d$ is endowed with the canonical
map $\Z[\Z_d] \to \Q(\zeta_{d})$ sending $1\in \Z_{d}$ to
$\zeta_{d}=\exp(2\pi\sqrt{-1}/d)$.  The following result says that
there is a bijection between $p$-towers of homology cobordant
manifolds, and corresponding intersection form defect invariants have
the same value:

\begin{theorem}
  \label{theorem:intro-homology-cobordism-invariance}
  Suppose $M$ and $M'$ are homology cobordant, $\Gamma_0$, $\Gamma_1$,
  \ldots, $\Gamma_{n-1}$ are abelian $p$-groups, and $\Gamma_n$ is a
  cyclic $p$-group.  Then for any $p$-tower
  \[
  M_n \to M_{n-1} \to \cdots \to M_0=M
  \]
  determined by $\{\phi_i\colon \pi_1(M_i) \to \Gamma_i\}$, there is a
  unique corresponding $p$-tower
  \[
  M'_n \to M'_{n-1} \to \cdots \to M'_0=M'
  \]
  which satisfies the following:
  \begin{enumerate}
  \item For each $i=0,1\ldots,n$, there is a bijection
    \[
    f_i\colon \Hom(\pi_1(M_i),\Gamma_i) \xrightarrow{\cong}
    \Hom(\pi_1(M'_i),\Gamma_i).
    \]
  \item $M'_{i+1}$ is determined by~$f_i(\phi_i)$ for
    $i=0,1,\ldots,n-1$.
  \item For any $\phi_n\colon \pi_1(M_n) \to \Gamma_n$, $\lambda(M_n,
    \phi_n)$ is well-defined if and only if so is $\lambda(M'_n,
    f_n(\phi_n))$, and in this case, $\lambda(M_n, \phi_n) =
    \lambda(M'_n, f_n(\phi_n))$ in $L^0(\Q(\zeta_d))$.
  \end{enumerate}
\end{theorem}  

Indeed Theorem~\ref{theorem:intro-homology-cobordism-invariance} holds
when $M$ and $M'$ are $\Z_p$-homology cobordant.  See
Section~\ref{section:homology-cobordism-invariants} for more details.

We remark that for any (possibly nonabelian) $p$-cover $\tilde M$ of
$M$, there is some $p$-tower which has $\tilde M$ as its top
cover~$M_n$.  Covers from which one can extract invariants using
Theorem~\ref{theorem:intro-homology-cobordism-invariance} are not
limited to $p$-covers of~$M$; in general, for a $p$-tower considered
in Theorem~\ref{theorem:intro-homology-cobordism-invariance}, $M_n$ is
not necessarily a regular cover of~$M$.  We also remark that
in~\cite{Cha-Ko:1999-1} metabelian towers were used to define
signature invariants of links.

In many cases (e.g., for manifolds obtained by surgery along a link,
see Corollary~\ref{corollary:size-of-character-group}) there are
infinitely many highly nontrivial $p$-towers so that
Theorem~\ref{theorem:intro-homology-cobordism-invariance} is not
vacuous.  To prove results of this type we reduce the problem to the
case of another space that is easier to investigate, by appealing to
the following proposition.  For a group $G$, let $\widehat G$ be the
``algebraic closure of $G$ with respect to $\Z_{(p)}$-coefficients''
in the sense of~\cite{Cha:2004-1}.

\begin{proposition}
  \label{proposition:t-equivalence-intro}
  If $X$ and $Y$ are (of the homotopy type of) finite CW-complexes and
  $f\colon X \to Y$ is a map inducing an isomorphism
  $\widehat{\pi_1(X)} \to \widehat{\pi_1(Y)}$, then $f$ induces a 1-1
  correspondence between $p$-towers of $X$ and $Y$ via pullback.
\end{proposition}

For a more precise statement, the reader is referred to
Definition~\ref{definition:t-equivalence} and
Proposition~\ref{proposition:algebraic-closure-and-t-equivalences}.
We note that a map which is 2-connected on the integral homology
satisfies (the hypothesis of)
Proposition~\ref{proposition:t-equivalence-intro}, due
to~\cite{Cha:2004-1}.

When we have a map of a simpler space $X$ (e.g., a 1-complex) into a
3-manifold $M$ satisfying
Proposition~\ref{proposition:t-equivalence-intro}, we can obtain
$p$-towers of $M$ from those of $X$ which are easier to construct.
Furthermore, for a 3-manifold obtained by infection, we can compute
algorithmically the intersection form defect invariants in terms of
some combinatorial data from the complex $X$ and the infection knot.
This algorithmic method applies to several interesting cases,
including our applications that will be discussed below (e.g., see
Section \ref{section:homology-cobordism-for-b1=0} and
\ref{section:iterated-bing-double}).

To investigate the structure of the Witt group $L^0(\Q(\zeta_d))$
where our invariant lives, in this paper we use the signature and
discriminant of a Witt class.  In particular, the discriminant
\[
\dis\lambda \in \frac{\Q(\zeta_d^{\mathstrut}+\zeta_d^{-1})^\times}{
  \{z\cdot \bar z \mid z \in \Q(\zeta_d)^\times \}},
\]
which is defined for $\lambda\in L^0(\Q(\zeta_d))$ as the determinant
with a sign correction, plays an essential role in our results on
torsion examples.  We employ tools from algebraic number theory to
detect nontrivial values of the discriminant; for $x \in
\Q(\zeta_d^{\vphantom{1}}+\zeta_d^{-1})^\times$ and a prime (ideal)
$\p$ of (the integer ring of)
$\Q(\zeta_d^{\vphantom{1}}+\zeta_d^{-1})$, the ``norm
residue symbol''
\[
(x, D)_\p \in \{1, -1\}, \text{\ where\ }
D=(\zeta_d^{\mathstrut}+\zeta_d^{-1})^2 - 4,
\]
is defined.  For a rapid summary of algebraic properties on the norm
residue symbol for non-experts, see
Section~\ref{subsection:norm-residue-symbol}.  (See also Chapter 3 of
the author's monograph~\cite{Cha:2003-1}.)  By the Hasse principle and
local Artin reciprocity of algebraic number theory, if $x=z\cdot \bar
z$ for some $z\in \Q(\zeta_d)$ then $(x,D)_\p$ is trivial for all~$\p$
(and the converse is true if we consider the archimedian valuations as
well as non-archimedian ones associated to primes).  Since $(-,D)_\p$
is multiplicative, it follows that $(-,D)_\p$ induces a well-defined
homomorphism
\[
(-,D)_\p\colon
\frac{\Q(\zeta_d^{\vphantom{1}}+\zeta_d^{-1})^\times}{\{z\cdot \bar z
  \mid z \in \Q(\zeta_d)^\times \}} \to \{1,-1\}.
\]
In the applications, we produce torsion examples realizing nontrivial
values of the norm residue symbols.

\subsection{Homology cobordism of rational homology 3-spheres}
\label{subsection:homology-cobordism-of-rational-3-sphere-intro}

The work of Cappell and Shaneson~\cite{Cappell-Shaneson:1974-1}
provides a surgery theoretic strategy for the study of homology
cobordism, which gives classication results in higher dimensions.  For
3-manifolds, it is known that Cappell-Shaneson theory does not
classify homology cobordism classes, however, it still gives a
framework to understand some useful invariants as homology surgery
obstructions.  For example, Wall's multisignature, or equivalently
Atiyah-Singer's $\alpha$-invariant or Casson-Gordon's invariant, is
defined for a given $M$ and a character $\chi\colon \pi_1(M) \to
\Z_d$, as in~\cite{Wall-Ranicki:1999-1}.  Due to Gilmer and
Livingston~\cite{Gilmer-Livingston:1983-1}, multisignatures for prime
power order characters are invariant under (topological) homology
cobordism.  Related works include the study of (topological and
smooth) homology cobordism of (homology) lens spaces, due to
Gilmer-Livingston~\cite{Gilmer-Livingston:1983-1},
Ruberman~\cite{Ruberman:1984-1,Ruberman:1988-1},
Cappell-Ruberman~\cite{Cappell-Ruberman:1988-1}, and
Fintushel-Stern~\cite{Fintushel-Stern:1987-1}.

Levine studied the Atiyah-Patodi-Singer $\eta$-invariants of homology
cobordant manifolds, focusing on applications to link
concordance~\cite{Levine:1994-1}.  His result for the most general
situation can be described in terms of the algebraic closure of
groups; for the algebraic closure $\widehat G$ of a group $G$ with
respect to $\Z$-coefficients in the sense of~\cite{Cha:2004-1}, there
is a natural map $p_G\colon G \to \widehat G$.  Denote by $R_k(G)$ the
variety of $k$-dimensional unitary representations of~$G$.  Then, for
$M$ with $G=\pi_1(M)$ and for $\theta\in R_k(\widehat{G})$, there is
defined the Atiyah-Patodi-Singer $\eta$-invariant
$\tilde\eta(M,\theta\circ p_{G})\in \R$.  Roughly speaking, Levine
showed the following: the function $\tilde\eta(M,-\circ p_{G})\colon
R_k(\widehat G) \to \R$ depends only on the homology cobordism class
of $M$, except for representations $\theta$ in some proper subvariety
of $R_k(G)$ called ``special'' in~\cite{Levine:1994-1}.  We note that
for homology cobordisms obtained from link concordances, it is more
natural to consider the algebraic closure in the sense of
Levine~\cite{Levine:1989-1} instead of~\cite{Cha:2004-1}, as
originally done in~\cite{Levine:1994-1}.

Recently, Harvey has proved the homology cobordism invariance of
certain $L^2$-signatures~\cite{Harvey:2006-1}.  For a group $G$, let
$G^{(n)}_H$ be the ``torsion-free-derived-series'' due to Harvey (for
the definition, refer to~\cite{Harvey:2006-1,Cochran-Harvey:2004-1}).
For $M$ with $G=\pi_1(M)$, let $\rho_n(M)$ be the von Neumann
$L^2$-signature defect of a 4-manifold bounded by $M$ over $G \to
G/G^{(n)}_H$, or equivalently, the Cheeger-Gromov invariant of $M$
associated to $G \to G/G^{(n)}_H$.  In~\cite{Harvey:2006-1}, it was
shown that if $M$ and $M'$ are homology cobordant, then $\rho_n(M) =
\rho_n(M')$.

As an application of our result, we illustrate that our invariant
distinguishes ``exotic homology cobordism types'' of some rational
homology spheres for which the known signature invariants always
vanish.

\begin{theorem}
  \label{theorem:exotic-rational-spheres-intro}
  There are infinitely many rational homology 3-spheres $\Sigma_0,
  \Sigma_1, \Sigma_2, \ldots$ with the following properties:
  \begin{enumerate}
  \item There is a homology equivalence (i.e., a map inducing
    isomorphisms on homology groups) $\Sigma_i \to \Sigma_0$ for
    any~$i$.
  \item Each $\Sigma_i$ has vanishing multisignatures.
  \item For any finite dimensional unitary representation $\theta$ of
    the algebraic closure of $\pi_1(\Sigma_i)$,
    $\tilde\eta(\Sigma_i,\theta\circ p_{\pi_1(\Sigma_i)})$ vanishes.
  \item $\rho_n(\Sigma_i)=0$ for any~$n$.
  \item $\Sigma_i$ and $\Sigma_j$ are not homology cobordant for $i\ne
    j$.
  \end{enumerate}
\end{theorem}

In the proof of Theorem~\ref{theorem:exotic-rational-spheres-intro},
to distinguish the homology cobordism classes of manifolds $\Sigma_i$
we compute and compare our invariants defined from $p$-towers $M_n \to
\cdots\to M_0$ of the manifolds $\Sigma_i$ with a fixed height $n$ and
with fixed deck transformation groups $\Gamma_0,\ldots,\Gamma_{n-1}$.
Roughly speaking, we construct primes $\p_i$ which are ``algebraically
dual'' to the intersection form defect invariants of the $\Sigma_i$
with respect to the norm residue symbols in the following sense:
\begin{enumerate}
\item If $i\ne j$, then for any $M_{n} \to \cdots \to M_{0}=\Sigma_j$
  and any $\phi\colon \pi_1(M_{n})\to \Z_d$,
  $(\dis\lambda(M_{n},\phi),D)_{\p_i}$ is trivial.
\item For each $i$, there exist $M_{n} \to \cdots \to M_{0}=\Sigma_i$
  and $\phi\colon \pi_1(M_{n})\to \Z_d$ such that
  $(\dis\lambda(M_{n},\phi), D)_{\p_i}$ is nontrivial.
\end{enumerate}
From Theorem~\ref{theorem:intro-homology-cobordism-invariance} and the
existence of the algebraically dual primes $\p_i$, it follows that
$\Sigma_i$ is not homology cobordant to $\Sigma_j$.  See
Section~\ref{section:homology-cobordism-for-b1=0} for more details and
further discussions.

\subsection{Obstructions to being a slice link and applications to
  iterated Bing doubles}

For a link $L$ in $S^3$, the \emph{zero-surgery manifold} of $L$ is
obtained by performing surgery on $S^3$ along the zero-framing of each
component of~$L$.  It is well known that if two links in $S^3$ are
concordant, then their zero-surgery manifolds are homology cobordant.
Therefore the invariants introduced above give rise to link
concordance invariants.  In particular we show the following result as
a consequence of
Theorem~\ref{theorem:intro-homology-cobordism-invariance}:

\begin{theorem}
  \label{theorem:slice-obstruction}
  If $L$ is a slice link with zero-surgery manifold $M$, then for any
  $p$-tower
  \[
  M_n \to M_{n-1} \to \cdots \to M_0=M
  \]
  and for any $\phi\colon\pi_1(M_n) \to \Z_d$ with $d$ a power of $p$,
  $\lambda(M_n,\phi)\in L^0(\Q(\zeta_d))$ is well-defined and vanishes.
\end{theorem}

We apply Theorem~\ref{theorem:slice-obstruction} to iterated Bing
doubles.  Figure~\ref{fig:bing-double} illustrates the construction of
the $n$th iterated Bing double $BD_n(K)$ of a knot~$K$.  In this paper
Bing doubles are always untwisted, i.e., in
Figure~\ref{fig:bing-double}, the parallel strands passing through the
box are twisted in such a way that their linking number is zero.

\begin{figure}[ht]
  \begin{center}
    \includegraphics[scale=.8]{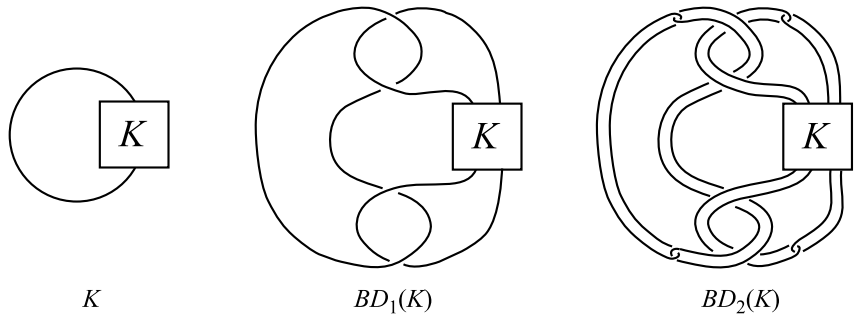}
  \end{center}
  \caption{Iterated Bing doubles of a knot $K$.}\label{fig:bing-double}
\end{figure}

Recently the problem of detecting non-slice Bing doubles have been an
interesting subject of research, partly motivated by the 4-dimensional
topological surgery problem.  It is known that many known link
concordance invariants vanish for (iterated) Bing doubles.  For an
excellent discussion on this, the reader is referred to Cimasoni's
paper~\cite{Cimasoni:2006-1}.  Harvey~\cite{Harvey:2006-1} and
Teichner (unpublished) proved, independently, that the integral of the
Levine-Tristram signature function $\sigma_K$ of $K$ over the unit
circle is zero if the Bing double $BD_1(K)$ is slice.  Cimasoni proved
a stronger result from a stronger assumption; namely, he proved that
if $L$ is boundary slice, i.e., if there are disjoint 3-manifolds in
the 4-ball bounded by the union of slice disks and Seifert surfaces of
components of $L$, then $K$ is algebraically
slice~\cite{Cimasoni:2006-1}.  Cochran, Harvey and Leidy announced a
result that there exist algebraically slice knots $K$ with non-slice
$BD_n(K)$.

As illustrated in Figure~\ref{figure:infection-example} in case of the
figure eight knot, $BD_n(K)$ is obtained from a trivial link by
infection by~$K$.  It is folklore that the infection curve producing
$BD_n(K)$ is in the $n$th derived subgroup (see
Section~\ref{section:iterated-bing-double} for more details).  So for
larger $n$, one needs to investigate geometric information from higher
terms of the derived series.  Using our invariants, we can detect the
non-sliceness of $BD_n(K)$ in several interesting cases.  First, we
generalize the result of Harvey and Teichner:

\begin{theorem}
  \label{theorem:signature-of-bing-double}
  For any $n$, the Levine-Tristram signature function $\sigma_K$ is
  determined by the iterated Bing double $BD_n(K)$ of~$K$.  In
  particular, if $\sigma_K$ is nontrivial, then for any~$n$, $BD_n(K)$
  is not slice.
\end{theorem}

We remark that even for $n=1$
Theorem~\ref{theorem:signature-of-bing-double} is stronger than the
result of Harvey and Teichner since $\sigma_K$ may be nontrivial even
when the integral of $\sigma_K$ is zero.

Because previously known obstructions are signatures, it has been
unknown whether the Bing double of $K$ can be non-slice for a torsion
knot $K$ (see also
Remark~\ref{remark:vanishing-of-signatures-for-bd_n}).  In particular,
the following question on the second-simplest knot has been asked by
several authors, including Schneiderman-Teichner
\cite{Schneiderman-Teichner:2006-1},
Cochran-Friedl-Teichner~\cite{Cochran-Friedl-Teichner:2006-1}, and
Cimasoni~\cite{Cimasoni:2006-1}: is the Bing double of the figure
eight knot a slice link?  We give a first answer to this question.

\begin{theorem}
  \label{theorem:non-slice-bing-souble-intro}
  There are infinitely many amphichiral knots $K$, including the
  figure eight knot, such that $BD_n(K)$ is not slice for all~$n$.
\end{theorem}

\begin{figure}[ht]
  \begin{center}
    \includegraphics[scale=.7]{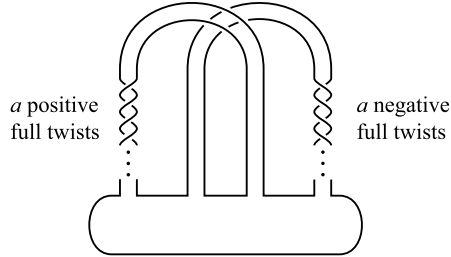}
  \end{center}
  \caption{Amphichiral knots $K_a$ used for infection.}
  \label{fig:infection-amph-knot}
\end{figure}

Our proof of Theorem~\ref{theorem:non-slice-bing-souble-intro} is
constructive; we prove that there is a sequence $1=a_1, a_2, \ldots$
of integers such that $BD(K_{a_i})$ is not slice for any $a_i$ and
$n$, where $K_{a}$ is the knot shown in
Figure~\ref{fig:infection-amph-knot}.  The sequence $\{a_n\}$~is
constructed by number theoretic arguments based on an investigation of
the norm residue symbols of discriminants.

Subsequent to this work, in \cite{Cha-Livingston-Ruberman:2006-1}
Livingston, Ruberman, and the author generalized the above results for
$n=1$ by proving that if $BD_1(K)$ is slice then $K$ is algebraically
slice, based on the ideas of rational knot concordance studied in a
monograph of the author~\cite{Cha:2003-1}.

\subsection{Torsion in the Cochran-Orr-Teichner filtration of link
  concordance}

Recently, as part of the on-going effort to understand concordance,
the structure of the solvable filtration has been studied actively.
The notion of $(n)$-solvability of knots and links was first
introduced by Cochran, Orr, and Teichner
in~\cite{Cochran-Orr-Teichner:1999-1}.  Very roughly speaking, a knot
or link is $(n)$-solvable if there is a topological 4-manifold bounded
by the zero-surgery manifold such that the twisted intersection form
of the cover associated to the $n$th derived subgroup of the
fundamental group looks like that of a slice disk complement.
They also defined $(n.5)$-solvability as a refinement between $(n)$-
and $(n+1)$-solvability.  The sets $\F_{(n)}$ of (concordance classes
of) $(n)$-solvable links form a filtration of the set of link
concordance classes, which is often referred to as
Cochran-Orr-Teichner's \emph{solvable filtration}.

It is known that there exist (infinitely many) nontrivial elements in
any depth of the filtration, i.e., in $\F_{(n)}$ for any~$n$.  In case
of knots, it was proved for $n=2$ by Cochran-Orr-Teichner
\cite{Cochran-Orr-Teichner:1999-1, Cochran-Orr-Teichner:2002-1}, and
for all higher~$n$ by Cochran-Teichner \cite{Cochran-Teichner:2003-1}.
(For $n<2$ the result is classical; e.g., see
\cite{Levine:1969-2,Jiang:1981-1}.)  Harvey first studied the case of
links and proved a nontriviality result in the solvable filtration of
links modulo ``local knotting'', using her invariant
$\rho_n$~\cite{Harvey:2006-1}.  Note that (the sum of $m$ copies of)
the solvable filtration of knots injects into that of ($m$-component)
links; Harvey's result shows that the solvable filtration of links has
its own perculiar complication even modulo the sophistication from
knots.  Taehee Kim considered a similar filtration for
double-sliceness of knots, and proved an analogous nontriviality
result for the filtration~\cite{Kim:2006-1}.

An open question is whether there is a nontrivial torsion element in
$\F_{(n)}$ for higher~$n$.  Here, generalizing the notion of torsion
knots, we say that a link $L$ is \emph{torsion} if for some $r\ne 0$ a
connected sum of $r$ copies of $L$ is slice for some choice of disk
basings.  (We remark that in order to construct a connect sum of
links, one needs to choose a disk basing for each summand link in the
sense of~\cite{Habegger-Lin:1990-1}, or equivalently, choose a string
link whose closure is the link.  For knots a connected sum is
well-defined independent of disk basings.)  All previously known
nontrivial elements in $\F_{(n)}$ for higher $n$ are not torsion, that
is, of infinite order (even modulo $\F_{(n+h)}$ for some $h>0$).
$L^2$-signatures have been used as the only computable obstruction to
being $(n)$-solvable for higher $n$, but so far, they have failed to
detect torsion.  We remark that for lower $n$, classical invariants
other than the $L^2$-signatures were used to detect torsion; it is a
well-known that there are torsion elements in~$\F_{(0)}$ detected by
the Alexander polynomial, and in \cite{Livingston:1999-2} Livingston
proved that there are torsion elements in~$\F_{(1)}$ using the
Casson-Gordon invariant.

The following refinement of Theorem~\ref{theorem:slice-obstruction}
gives a new obstruction for a \emph{link} to be $(n)$-solvable:



\begin{theorem}
  \label{theorem:solvablity-obstruction}
  If $L$ is an $(n)$-solvable link with zero-surgery manifold $M$,
  then for any $p$-tower
  \[
  M_k \to M_{k-1} \to \cdots \to M_0=M
  \]
  with height $k<n$ and for any $\phi\colon\pi_1(M_k) \to \Z_d$ with
  $d$ a power of $p$, $\lambda(M_k,\phi)\in L^0(\Q(\zeta_d))$ is
  well-defined and vanishes.
\end{theorem}

An analogue for 3-manifolds also holds (see
Theorem~\ref{theorem:Witt-invariant-for-solvable-manifold}.)  As an
application of Theorem~\ref{theorem:solvablity-obstruction}, we prove
a first result on the existence of ``torsion'' at an arbitrary depth
of the solvable filtration. 

\begin{theorem}
  \label{theorem:torsion-in-COT-filtration}
  For any $n$, there are infinitely many boundary links $L_i$ in $S^3$
  which are 2-torsion and $(n)$-solvable but not $(n+2)$-solvable.
\end{theorem}

The links $L_i$ in Theorem~\ref{theorem:torsion-in-COT-filtration} are
constructed by taking the $n$th iterated Bing doubles of certain
carefully chosen amphichiral knots.

We remark that it can be shown that the 2-torsion links $L_i$ are
``independent over $\Z_2$'' modulo $\F_{(n+2)}$ for an
\emph{arbitrary} choice of disk basings, in the following sense: for
any $a_i\in\{0,1\}$, if a connected sum of the $a_iL_i$ is
$(n+2)$-solvable for some disk basings, then $a_i=0$ for all~$i$.  A
proof of this independence result will appear in a subsequent paper
\cite{Cha:2007-2} because of additional technicalities required to
treat the sophistication from disk basings.

We remark that there are some previously known results on the
existence of independent infinite order elements in the
Cochran-Orr-Teichner filtration.  In case of knots, for $n\le 2$ there
are infinitely many infinite order elements in $\F_{(n)}$ which are
independent (over $\Z$) modulo~$\F_{(n.5)}$, that is,
$\F_{(n)}/\F_{(n.5)}$ is of infinite rank.  (For $n\le 1$, it is a
classical result, and for $n=2$, it was shown
in~\cite{Cochran-Orr-Teichner:2002-1}.)  In case of links, Harvey
considered the solvable filtration $\{\BF_{(n)}\}$ of the boundary
string link concordance group instead of spherical links, so that the
difficulty from disk basings is avoided, and proved that there are
infinite order boundary string links which generate an infinite rank
subgroup in (the abelianization of) $\BF_{(n)}/\BF_{(n+1)}$, using her
invariant $\rho_n$~\cite{Harvey:2006-1}.  It was generalized to
$\BF_{(n)}/\BF_{(n.5)}$ in~\cite{Cochran-Harvey:2006-01}.

Theorem \ref{theorem:solvablity-obstruction} also enables us to prove
the following result on infinite order elements: there are infinitely
many $(n)$-solvable boundary link $L'_i$ which are not
$(n+1)$-solvable but have vanishing Harvey's invariant $\rho_n$ (see
Corollary~\ref{corollary:n-solvable-links-with-vanishing-harvey-invariants}.)
It can also be shown that these $L'_i$ are ``independent over $\Z$''
modulo $\F_{(n+1)}$ for an arbitrary choice of disk basings, in the
sense that for any $a_i\in\Z$, if a connected sum of the $a_iL'_i$ is
$(n+1)$-solvable for some disk basings, then $a_i=0$ for
all~$i$~\cite{Cha:2007-2}.  Considering the subgroup generated by
string links whose closures are the $L'_i$, it can be shown that the
kernel of $\rho_n$ is large enough to contain a subgroup whose
abelianization is isomorphic to $\Z^\infty$~\cite{Cha:2007-2}.  We
also remark that previously known signature invariants of link
concordance, including $\rho_n$, vanish on the 2-torsion links $L_i$
given in Theorem~\ref{theorem:torsion-in-COT-filtration}.  (See
Remark~\ref{remark:vanishing-of-signatures-for-bd_n}.)  So, the kernel
of $\rho_n$ also contains infinitely many independent 2-torsion links.
We remark that Cochran-Harvey-Leidy have recently announced a proof
that the kernel of $\rho_n$ contains an infinite cyclic subgroup using
a different technique.

In addition, our results also hold for the \emph{grope filtration} of
links which is defined and investigated
in~\cite{Cochran-Orr-Teichner:1999-1, Cochran-Orr-Teichner:2002-1,
  Cochran-Teichner:2003-1,Harvey:2006-1}.  In particular, for any
$n\ge 2$ there are infinitely many 2-torsion elements in the $n$th
term of the grope filtration of links (which are independent in an
appropriate sense).  For more details, see
Section~\ref{section:computation-for-knot-surgery-manifold}.

\subsection*{Acknowledgements}
The author thanks an anonymous referee for helpful comments.  This work
was supported by Korea Research Foundation Grant KRF-2004-041-C00047.

\section{An Atiyah-type lemma and intersection form defects}
\label{section:Witt-class-defect}

In this section we define an invariant of 3-manifolds which is
essentially a Witt class defect of intersection forms.  We start with
a discussion on an Atiyah-type result for Mishchenko-Ranicki's
symmetric signature of topological 4-manifolds, which will be used to
show the well-definedness of the invariant.

\subsection{An Atiyah-type lemma}

We recall some terms of the $L$-theory which is used to state our
Atiyah-type result in a general context.  A standard reference to the
terms is Ranicki's book~\cite{Ranicki:1981-1}.  For a group $\pi$, we
regard the integral group ring $\Z\pi$ as a ring with involution via
$(\sum_{g\in \pi} n_g g)^- = \sum_{g\in \pi} n_g g^{-1}$.  Let
$L^n(\Z\pi)$ be the \emph{symmetric $L$-group over $\Z\pi$}, that is,
the abelian group of cobordism classes of symmetric Poincar\'e chain
complexes of dimension $n$ over~$\Z\pi$.  We say that $X$ is a
\emph{geometric $n$-dimensional Poincar\'e space} if $X$ is a finite
CW-complex satisfying the Poincar\'e duality for any coefficients.
The cellular chain complex $C_*(X;\Z\pi_1(X))$ endowed with the chain
equivalence given by the cap product with the orientation class gives
rise to element $\sigma^*(X) \in L^n(\Z\pi_1(X))$ which is called the
\emph{Mishchenko-Ranicki symmetric signature}.

We say that a space $X$ is \emph{over a group $\Gamma$} if it is
endowed with the homotopy class of a map $\phi\colon X \to B\Gamma$,
where $B\Gamma$ denotes the classifying space for $\Gamma$, or
equivalently, $K(\Gamma,1)$.  When the choice of $\phi$ is clearly
understood from the context, we will not specify it explicitly.  For a
Poincar\'e space $X$ of dimension $n$ which is over $\Gamma$, we
denote by $\sigma^*_\Gamma(X)$ the image of $\sigma^*(X)$ under the
natural map
\[
L^n(\Z\pi_1(X)) \to L^n(\Z\Gamma).
\]

We will focus on the case of topological 4-manifolds.  A closed
topological 4-manifold $W$ over $\Gamma$ has a canonical homotopy type
of a 4-dimensional Poincare space, and therefore $\sigma^*_\Gamma(W)
\in L^4(\Z\Gamma)$ is defined.  In particular, when $\Gamma$ is a
trivial group, we obtain an element in $L^4(\Z) \cong \Z$ which is
equal to the ordinary signature~$\sigma(W)$.  Let
\[
i^*_\Gamma \colon \Z=L^4(\Z) \to L^4(\Z\Gamma)
\]
be the abelian group homomorphism induced by the inclusion of a
trivial group into~$\Gamma$.

\begin{lemma}
  \label{lemma:atiyah-type-result}
  Suppose $H_4(\Gamma)=0$.  Then for any closed topological 4-manifold
  $W$ over $\Gamma$, $\sigma^*_\Gamma(W) = i^*_\Gamma \sigma(W)$ in
  $L^4(\Z\Gamma)$.
\end{lemma}

\begin{proof}
  Our proof follows the lines of the bordism theoretic approach to
  Atiyah-type theorems, based on the fact that the association $W \to
  \sigma^*_\Gamma(W)$ gives rise to a homomorphism
  \[
  \sigma^*_\Gamma\colon \Omega_4^{top}(B\Gamma) \to L^4(\Z\Gamma),
  \]
  where $\Omega_4^{top}(B\Gamma)$ is the 4-dimensional topological
  bordism group over~$B\Gamma$.  First we claim that
  $\Omega_4^{top}(B\Gamma)$ is generated by $\C P^2$, which is
  automatically endowed with a constant map into $B\Gamma$, being a
  simply connected space.  To show the claim, we consider the
  Atiyah-Hirzebruch spectral sequence for bordism groups:
  \[
  E_{p,q}^2 = H_p(\Gamma;\Omega_q^{top}) \Longrightarrow
  \Omega_n^{top}(B\Gamma)
  \]
  where $\Omega_n^{top}$ denotes the topological cobordism group of
  $n$-manifolds.  Note that $\Omega^{top}_0=\Z$,
  $\Omega^{top}_2=\Omega^{top}_3=0$, and $\Omega^{top}_4=\Z$ generated
  by $\C P^2$.  Also we have $E^2_{4,0}=H_4(\Gamma;\Z)=0$ by the
  hypothesis.  So it follows that all the $E^2_{p,q}$ terms vanish for
  $p+q=4$ but $E^2_{0,4}=H_0(\Gamma;\Omega_4^{top}) \cong
  \Omega_4^{top}$.  Since $d^r_{0,4}\colon E^r_{0,4} \to E^r_{-r,r+3}$
  is obviously trivial, $E^r_{0,4}$ is a quotient of $E^2_{0,4}$ for
  all~$r\ge 2$.  (In fact, the only case that $d^r_{r,5-r}\colon
  E^r_{r,5-r} \to E^r_{0,4}$ is potentially nontrivial is when $r=5$.)
  From this it follows that $E^\infty_{p,q}=0$ for all $p+q=4$ but
  $E^\infty_{0,4}$, which is a quotient of $\Omega^{top}_4\cong \Z$.
  This shows that $\Omega_4(B\Gamma)$ is generated by the bordism
  class of~$\C P^2$.

  By the claim, it suffices to consider the special case that $W$ is
  simply connected.  The symmetric signature $\sigma^*(W)$ of $W$ is
  in $L^4(\Z\pi_1(W))=L^4(\Z)=\Z$, and indeed equal to the ordinary
  signature~$\sigma(W)$.  Now by the definition, $\sigma^*_\Gamma(W) =
  i^*_\Gamma \sigma(W)$.
\end{proof}

\subsection{Intersection form defects}

We now define an invariant for 3-manifolds which lives in an
$L$-group.  For technical and computational convenience, we pass from
the $L$-theory over $\Z\Gamma$ to that of a (skew-)field; let $\K$ be
a (skew-)field with involution, and fix a homomorphism $\Z\Gamma \to
\K$ between rings with involutions, which induces a map
\[
L^4(\Z\Gamma) \to L^4(\K).
\]
Since $\K$ is a skew-field, every $\K$-module is free.  Also,
$L^4(\K)$ is canonically identified with $L^0(\K)$, which is the Witt
group of nonsingular hermitian forms on finitely generated (free)
modules over~$\K$.  For a closed topological 4-manifold $W$ over
$\Gamma$, the image of $\sigma^*_\Gamma(W)$ in $L^0(\K)$ is
represented by the intersection form of $W$ with $\K$-coefficients.
More precisely, on the $\K$-coefficient homology module
\[
H_2(W;\K)=H_2(C_*(W;\Z\Gamma)\otimes_{\Z\Gamma}\K)
\]
of $W$, the $\K$-coefficient intersection form
\[
\lambda_\K(W) \colon H_2(W;\K) \times H_2(W;\K) \to \K
\]
is defined as a nonsingular hermitian form over~$\K$.  Its Witt class
$[\lambda_\K(W)]\in L^0(\K)$ is the image of $\sigma^*_\Gamma(W)$.

If $W$ is not closed, then the above $\lambda_\K(W)$ is not
necessarily nonsingular.  However, it can be seen that, letting $\bar
H_2(W;\K)$ be the image of
\[
H_2(W;\K) \to H_2(W,\partial W;\K)
\]
which is automatically finitely generated and $\K$-free,
$\lambda_\K(W)$ gives rise to a nonsingular hermitian form
\[
\bar H_2(W;\K) \times \bar H_2(W;\K) \to \K.
\]
We will denote it by $\lambda_\K(W)$ as an abuse of notation.  Indeed
it is well-defined and nonsingular since
\[
H_2(W,\partial W;\K) \cong H^2(W;\K) \cong \Hom(H_2(W;\K),\K)
\]
by the Poincar\'e duality and the universal coefficient theorem over
the skew-field~$\K$.  Therefore, even for the case that $W$ is not
closed, we can think of the Witt class $[\lambda_\K(W)] \in L^0(\K)$.

Let $S$ be a fixed nonempty multiplicatively closed subset of~$\Z$.

\begin{definition}
  Suppose $M$ is a closed 3-manifold over $\Gamma$ via $\phi\colon M
  \to B\Gamma$ such that the bordism class of $(M,\phi)$ is
  $S$-torsion in $\Omega^{top}(B\Gamma)$, i.e., $r(M,\phi)=0$ in
  $\Omega^{top}(B\Gamma)$ for some $r\in S$.  Then there is a compact
  topological 4-manifold $W$ endowed with a map $W\to B\Gamma$ which
  is bounded by $rM$ over~$\Gamma$.  Define
  \[
  \lambda(M,\phi) = \frac 1r \otimes \big([\lambda_\K(W)] - i^*_\K
  \sigma(W)\big) \in S^{-1}\Z \otimes_\Z L^0(\K)
  \]
  where $i^*_\K \colon \Z=L^0(\Z) \to L^0(\K)$ is the map induced by
  the canonical inclusion $\Z\to \K$.
\end{definition}

%

\begin{lemma}
  If $H_4(\Gamma)=0$, then $\lambda(M,\phi)$ is well-defined for any
  $(M,\phi)$ which is $S$-torsion in $\Omega^{top}_3(B\Gamma)$.
\end{lemma}

\begin{proof}
  Suppose $W$ and $W'$ are null-bordisms of $rM$ and $r'M$
  over~$\Gamma$, respectively, where $r,r'\in S$.  Let $V=r'W
  \cup_{rr'M} r(-W')$ be the manifold obtained by gluing $r'W$ and
  $r(-W')$ along~$rr'M$.  Then $V$ is over~$\Gamma$.  The standard
  Novikov additivity argument shows
  \[
  [\lambda_\K(V)] = r'[\lambda_\K(W)]-r[\lambda_\K(W')] \quad
  \text{and}\quad \sigma(V)=r'\sigma(W)-r\sigma(W').
  \]
  By Lemma~\ref{lemma:atiyah-type-result}, $[\sigma_\Gamma^*(V)] =
  i^*_\Gamma \sigma(V)$, and therefore, $[\lambda_\K(V)] = i^*_\K
  \sigma(V)$.  It follows that
  \[
  r'\big([\lambda_\K(W)]-i^*_\K \sigma(W)\big) =
  r\big([\lambda_\K(W')]-i^*_\K \sigma(W')\big) \qedhere
  \]
\end{proof}

\begin{remark}\noindent\par\Nopagebreak
  \begin{enumerate}
  \item Our $\lambda$ has naturality with respect to $S$ in the
    following sense: let $\M_S$ be the collection of pairs $(M,\phi)$
    which are $S$-torsion in~$\Omega^{top}_3(B\Gamma)$.  Then
    $\lambda(-,-)$ can be viewed as a function on $\M_S$.  For $S'
    \subset S$, we have $\M_{S'} \subset \M_S$ and there is a
    commutative diagram
    \[
    \begin{diagram}
      \node{\M_{S'}} \arrow{e,J} \arrow{s,l}{\lambda}
      \node{\M_S} \arrow{s,r}{\lambda}
      \\
      \node{S^{\prime -1} \Z \otimes L^0(\K)} \arrow{e,t}{f_{S',S}}
      \node{S^{-1} \Z \otimes L^0(\K)}
    \end{diagram}
    \]
    where $f_{S',S}$ is induced by $S'^{-1}\Z \to S^{-1}\Z$.  The
    kernel of $f_{S',S}$ is nontrivial in general.  Indeed, it
    consists of the $S$-torsion part:
    \[
    \Ker f_{S',S} = \{ x\in S^{\prime -1} \Z \otimes L^0(\K) \mid
    r\cdot x = 0 \text{ for some } r\in S\}
    \]
    Therefore, for a larger $S$, one has the advantage that
    $\lambda(M,\phi)$ is defined on a larger collection $\M_S$, but
    more (torsion) information is lost.
  \item In particular when $S=\{1\}$ and $(M,\phi)$ is null-bordant,
    $\lambda(M,\phi)$ lives in $L^0(\K)$, without losing any torsion
    information.  This special case will be used in our applications
    discussed later.
  \item If one wanted to extract information on the torsion-free part
    only (e.g., by considering signature-type invariants of
    $L^0(\K)$), then $S=\Z-\{0\}$ could be used to define
    $\lambda(M,\phi) \in \Q\otimes L^0(\K)$ for any $(M,\phi)$ which
    has finite order in~$\Omega_3^{top}(B\Gamma)$.
  \end{enumerate}
\end{remark}

\section{$p$-towers and homology cobordism}
\label{section:homology-cobordism-invariants}

Let $R$ be an abelian group.  We say that two closed 3-manifolds $M$
and $M'$ are \emph{$R$-homology cobordant} if there is a compact
4-manifold $W$ such that $\partial W=M\cup -M'$ and the inclusions of
$M$ and $M'$ into $W$ are $R$-homology equivalences, that is, the
induced maps $H_i(M;R) \to H_i(W;R)$, $H_i(M';R) \to H_i(W;R)$ are
isomorphisms for all~$i$.  Such a manifold $W$ is called an
\emph{$R$-homology cobordism}.  When $R=\Z$, one usually says that $M$
and $M'$ are homology cobordant.  We will often consider the case that
$R=\Z_p$, the abelian group of residue classes modulo~$p$,
where $p$ is prime.

The aim of this section is to show the homology cobordism invariance
of the invariant $\lambda(M,\phi)$ which is defined in the previous
section for $\phi\colon M \to B\Gamma$.  As the first step, we
investigate the case that $\Gamma$ is a $p$-group and $\phi$ factors
through~$W$.  As an abuse of notation, for a CW-complex $X$, we will
regard (the homotopy class of) $\phi\colon X\to B\Gamma$ as a group
homomorphism $\phi\colon \pi_1(X) \to \Gamma$.  (For disconnected $X$,
one may adopt a convention that $\pi_1(X)$ designates the free product
of the fundamental groups of components of~$X$.)

As in the previous section, we assume $H_4(\Gamma)=0$ and fix a
multiplicatively closed subset $S$ of $\Z$ and a map $\Z\Gamma\to \K$.
In addition, we assume that $\K$ is of characteristic zero.

\begin{proposition}
  \label{proposition:homology-cobordism-and-p-covers}
  Suppose $W$ is a $\Z_p$-homology cobordism between 3-manifolds $M$
  to~$M'$, and $\psi\colon\pi_1(W) \to \Gamma$ is a group homomorphism
  into a $p$-group~$\Gamma$.
  \begin{enumerate}
  \item Let $\phi\colon \pi_1(M) \to \Gamma$ and $\phi'\colon
    \pi_1(M') \to \Gamma$ be the restrictions of~$\psi$.  Then
    $(M,\phi)$ is $S$-torsion in $\Omega^{top}_3(B\Gamma)$ if and only
    if so is $(M',\phi')$.  If it is the case, then
    $\lambda(M,\phi)=\lambda(M',\phi')$ in $S^{-1}\Z\otimes L^0(\K)$.
  \item Let $W_\Gamma$, $M_\Gamma$, and $M'_\Gamma$ be the
    $\Gamma$-covers of $W$, $M$, and $M'$ determined by $\psi$,
    $\phi$, and $\phi'$, respectively.  Then $W_\Gamma$ is a
    $\Z_p$-homology cobordism between $M_\Gamma$ and~$M'_\Gamma$.
  \end{enumerate}
\end{proposition}

In proving
Proposition~\ref{proposition:homology-cobordism-and-p-covers} and
other results in this section, the following result which is
essentially due to Levine (see the proof of Proposition 3.2 of
\cite{Levine:1994-1}) plays a crucial role.

\begin{lemma}[Levine]
  \label{lemma:levine's-lemma}
  Suppose $\Gamma$ is a $p$-group, and $C_*$ is a chain complex
  consisting of free $\Z\Gamma$-modules such that $\bigoplus_{i \le n}
  C_i$ is finitely generated.  If $H_i(C_*\otimes_{\Z\Gamma} \Z_p)$
  vanishes for $i\le n$ (where $\Z_p$ is regard as a $\Z\Gamma$-module
  with trivial $\Gamma$-action), then $H_i(C_*\otimes_{\Z} \Z_p)$
  vanishes for $i\le n$.
\end{lemma}

We remark that in~\cite{Levine:1994-1}, $\Z_{(p)}=$ (the localization
of $\Z$ away from~$p$) is used instead of~$\Z_p$.  Since
$H_i(C_*\otimes_\Z \Z_p)$ vanishes for $i\le n$ if and only if so does
$H_i(C_*\otimes_\Z \Z_{(p)})$ for any finitely generated free chain
complex $C_*$ over $\Z$, Lemma~\ref{lemma:levine's-lemma} is
equivalent to Levine's original statement in~\cite{Levine:1994-1}.

For an abelian group $R$, we say that a map $f\colon Y \to X$ is
\emph{$n$-connected with respect to $R$-coefficients} if
$H_i(X,Y;R)=0$ for $i \le n$.  If $X$ and $Y$ have finite
$n$-skeletons, then $f\colon X \to Y$ is $n$-connected with respect to
$\Z_p$-coefficients if and only if $f$ is $n$-connected with respect
to $\Z_{(p)}$-coefficients.  The following is an immediate consequence
of Levine's lemma:

\begin{lemma}
  \label{lemma:connectedness-of-p-covers}
  Suppose $X$ and $Y$ are CW-complexes with finite $n$-skeletons,
  $f\colon Y\to X$ is $n$-connected with respect to
  $\Z_p$-coefficients, and $\pi_1(X) \to \Gamma$ is a map into a
  $p$-group~$\Gamma$.  Then $H_i(X,Y;\Z_{(p)}\Gamma)=0$ for $i\le n$.
  In other words, denoting the associated $\Gamma$-covers of $X$ and
  $Y$ by $X_\Gamma$ and $Y_\Gamma$, respectively, the lift $Y_\Gamma
  \to X_\Gamma$ of $f$ is $n$-connected with respect to
  $\Z_p$-coefficients.
\end{lemma}

\begin{proof}
  [Proof of
  Proposition~\ref{proposition:homology-cobordism-and-p-covers}]%
  (1) If $(M,\phi)$ is $S$-torsion in $\Omega^{top}_3(B\Gamma)$, then
  there is a null-cobordism, say $V$, of $rM$ for some $r \in S$.
  Attaching $rW$ to $V$ along $rM$, we obtain a null-cobordism of
  $rM'$ over~$\Gamma$.  It follows that $(M',\phi')$ is $S$-torsion in
  $\Omega^{top}_3(B\Gamma)$.  The converse is proved similarly.

  By the definition of $\lambda(-,-)$ and a Novikov additivity
  argument, it can be seen that
  \[
  \lambda(M,\phi)-\lambda(M',\phi')=[\lambda_\K(W)]-i^*_\K
  \sigma(W).
  \]
  Since $H_*(W,M;\Z_p)=0$, $\sigma(W)$ vanishes.  Also, by
  Lemma~\ref{lemma:connectedness-of-p-covers}, it follows that
  $H_*(W,M;\Z_{(p)}\Gamma)=0$.  Since $\K$ has characteristic zero,
  $\Z \to \Z\Gamma \to \K$ is injective.  Since $\K$ is a division
  ring, the map $\Z\Gamma \to \K$ factors through $\Z_{(p)}\Gamma$.
  It follows that $H_*(W,M;\K)=0$ so that $[\lambda_\K(W)]=0$.
  
  (2) Since $H_*(W,M;\Z_p)=0$, $H_*(W_\Gamma,M_\Gamma;\Z_p)=0$ by
  Lemma~\ref{lemma:connectedness-of-p-covers}.  Similarly
  $H_*(W_\Gamma,M'_\Gamma;\Z_p)=0$.
\end{proof}  

In general, determining whether a given $\phi$ factors through
$\pi_1(W)$ for some homology cobordism $W$ is known to be a difficult
problem.  As an easy special case, if $\Gamma$ is abelian, then $\phi$
factors through $\pi_1(W)$ for any homology cobordism $W$, since any
map into an abelian group factors through $H_1(-)$.  In order to
investigate further sophistication beyond the abelianization, we consider
certain towers of covers of $M$ in the next subsection.

\subsection{$p$-towers}

Let $X$ be a CW-complex.  As in the introduction, a tower
\[
X_n \to \cdots \to X_1 \to X_0 =X
\]
of covering maps is called a \emph{$p$-tower of height $n$} if each
$X_{i+1} \to X_{i}$ is a regular cover whose covering transformation
group is an abelian $p$-group~$\Gamma_i$.  Such a tower is determined
by inductively choosing maps $\phi_i\colon \pi_1(X_i) \to \Gamma_i$.
In this paper, we always assume that a $p$-tower consists of connected
spaces, unless stated otherwise.  In other words, ($X$ is connected
and) the maps $\phi_i$ are surjective.

Obviously, given a map $f\colon Y \to X$, a $p$-tower for $X$ induces
one for $Y$ via pullback.  We will frequently consider $f$ which gives
rise to a 1-1 correspondence between $p$-towers for $X$ and $Y$.  We
give a formal definition below.

\begin{definition}
  \label{definition:t-equivalence}
  A map $f_0\colon Y_0 \to X_0$ is called a \emph{$p$-tower map
    of height $n$} if for any abelian $p$-groups
  $\Gamma_0,\ldots,\Gamma_n$, the following holds:
  \begin{enumerate}
  \item[(0)] $f_0$ induces a bijection
    \[
    f_0^*\colon \Hom(\pi_1(X_0),\Gamma_0) \to
    \Hom(\pi_1(Y_0),\Gamma_0).
    \]
  \item[(1)] Let $X_1$ and $Y_1$ be the covers of $X_0$ and $Y_0$
    determined by a map $\phi_0\colon \pi_1(X_0) \to \Gamma_0$ and
    $\psi_0=f_0^*(\phi_0^{\vphantom{*}})\colon \pi_1(Y_0) \to
    \Gamma_0$, respectively.  Then the lift $f_1\colon Y_1 \to X_1$ of
    $f_0$ induces a bijection
    \begin{gather*}
      f_1^*\colon \Hom(\pi_1(X_1),\Gamma_1) \to
      \Hom(\pi_1(Y_1),\Gamma_1).\\
      \vdots
    \end{gather*}
    \item[($n$)] Let $X_n$ and $Y_n$ be the covers of $X_{n-1}$ and
    $Y_{n-1}$ determined by a map $\phi_{n-1}\colon \pi_1(X_{n-1}) \to
    \Gamma_{n-1}$ and
    $\psi_{n-1}=f_{n-1}^*(\phi_{n-1}^{\vphantom{*}})\colon
    \pi_1(Y_{n-1}) \to \Gamma_{n-1}$, respectively.  Then the lift
    $f_n\colon Y_n \to X_n$ of $f_{n-1}$ induces a bijection
    \[
    f_n^*\colon \Hom(\pi_1(X_n),\Gamma_n) \to
    \Hom(\pi_1(Y_n),\Gamma_n).
    \]
  \end{enumerate}
  If the above condition is satisfied for all $n$, then $f_0$ is
  called a \emph{$p$-tower map}.
\end{definition}

\begin{remark}
  \label{remark:composition-of-t-equivalence}
  For $f\colon Y \to X$ and $g\colon Z \to Y$, if any two of $f$, $g$,
  and $f\circ g$ are $p$-tower maps, then the third is also a
  $p$-tower map.
\end{remark}

The following simple observation says that in our case pullback
preserves the connectedness condition.

\begin{lemma}
  Suppose $G$ and $\pi$ are finitely presented groups and $f\colon G
  \to \pi$ is a homomorphism inducing an injection $f^*\colon
  \Hom(\pi,\Z_p) \to \Hom(G,\Z_p)$.  Then for any abelian $p$-group
  $\Gamma$, a map $\phi\colon \pi \to \Gamma$ is surjective if and
  only if so is $\phi f\colon G \to \Gamma$.
\end{lemma}

\begin{proof}
  The if part is obvious.  For the only if part, observe that for any
  abelian $p$-group $A$ and a power $r$ of $p$ such that $r\ge |A|$,
  $\Hom(-,A)\cong \Hom(H_1(-)\otimes\Z_r,A)$ and a group homomorphism
  $B \to A$ is surjective if and only if so is the induced map
  $H_1(B)\otimes\Z_r \to A$.  So we may assume that both $G$ and $\pi$
  are abelian $p$-groups by applying $H_1(-)\otimes\Z_r$ to $G$
  and~$\pi$ where $r=|\Gamma|$.

  Let $C$ be the cokernel of~$f$.  Since $f^*$ is injective and
  $\Hom(-,\Z_p)$ is left exact, $\Hom(C,\Z_p)=0$.  Thus $C$ is
  $p$-torsion free, that is, there is no nontrivial element in $C$
  whose order is a power of $p$.  However, being a quotient of a
  $p$-group, $C$ is a $p$-group.  It follows that $C=0$, that is, $f$
  is surjective.  So $\phi f$ is surjective whenever so is~$\phi$.
\end{proof}

\begin{lemma}
  \label{lemma:2-connected-maps-induce-t-equivalence}
  If $X$ and $Y$ are CW-complexes with finite 2-skeletons and $f\colon
  Y\to X$ is 2-connected with respect to $\Z_p$-coefficients, then $f$
  is a $p$-tower map.
\end{lemma}

\begin{proof}
  We claim the following: for any CW-complexes $X$ and $Y$ with finite
  $2$-skeletons and any abelian $p$-group $\Gamma$, if $f\colon Y \to
  X$ is a 2-connected with respect to $\Z_p$-coefficients, then
  \[
  f^*\colon \Hom(\pi_1(X), \Gamma) \to \Hom(\pi_1(Y),\Gamma)
  \]
  is a one-to-one correspondence.  For, let $r=|\Gamma|$.  Then
  \[
  \Hom(\pi_1(X),\Gamma) \cong \Hom(H_1(X)\otimes
  \Z_r,\Gamma)=\Hom(H_1(X;\Z_r),\Gamma)
  \]
  and similarly for~$Y$.  Since $H_i(X,Y;\Z_p)=0$ for $i\le 2$,
  $H_i(X,Y;\Z_r)=0$ for $i\le 2$.  From the long exact sequence for
  $(X,Y)$, it follows that $Y \to X$ induces an isomorphism on
  $H_1(-;\Z_r)$.  This proves the claim.

  Now, we use an induction on $n$ to show that
  Definition~\ref{definition:t-equivalence} ($n$) holds, and in
  addition, that the map $f_n\colon Y_n \to X_n$ is 2-connected with
  respect to $\Z_p$-coefficients.  By the above claim,
  Definition~\ref{definition:t-equivalence} (0) holds.  Also, $f$ is
  2-connected with respect to $\Z_p$-coefficients by the hypothesis.
  Suppose Definition~\ref{definition:t-equivalence} ($n-1$) holds and
  $f_{n-1}$ is 2-connected with respect to $\Z_p$-coefficients.  Then
  by Lemma~\ref{lemma:connectedness-of-p-covers}, the lift $f_n\colon
  Y_n \to X_n$ is 2-connected with respect to $\Z_p$-coefficients.  By
  the above claim, it follows that $f_n$ induces a bijection on
  $\Hom(\pi_1(-),\Gamma_n)$, that is,
  Definition~\ref{definition:t-equivalence} $(n)$ holds.
\end{proof}

Lemma~\ref{lemma:2-connected-maps-induce-t-equivalence} enables us to
apply Proposition~\ref{proposition:homology-cobordism-and-p-covers}
inductively to $p$-towers.  As an immediate consequence, one obtains
a sequence of homology cobordism invariants, as stated below.  From
now on, a cyclic group $\Gamma=\Z_{d}$ is always endowed with the map
$\Z\Gamma \to \K=\Q(\zeta_d)$ sending $1\in \Z_d$ to
$\zeta_d=\exp(2\pi\sqrt{-1}/d)$.  Note that $H_4(\Z_d)=0$.

\begin{theorem}
  \label{theorem:invariants-from-covering-tower}
  Suppose $W$ is a $\Z_p$-homology cobordism between 3-manifolds $M$
  and~$M'$.  Then the following holds:
  \begin{enumerate}
  \item $M \to W$ and $M' \to W$ are $p$-tower maps.
  \item For a given $p$-tower 
    \[
    M_n \to \cdots \to M_1 \to M_0=M,
    \]
    let 
    \begin{gather*}
      W_n \to \cdots \to W_1 \to W_0=W,\\
      M'_n \to \cdots \to M'_1 \to M'_0=M'
    \end{gather*}
    be the $p$-towers of $W$ and $M'$ which correspond to the $M_i$
    via pullback along the $p$-tower maps $M \to W \leftarrow M'$.
    Then $W_i$ is a $\Z_p$-homology cobordism between $M_i$ and $M'_i$
    for each~$i$.
  \item Let $d$ be a power of $p$ and
    \[
    \Hom(\pi_1(M_n),\Z_d) \approx \Hom(\pi_1(M'_n),\Z_d)
    \]
    be the bijection induced by the $p$-tower maps $M \to W \leftarrow
    M'$.  For any $\phi_n\colon\pi_1(M_n)\to \Z_d$ and the
    corresponding $\phi'_n\colon\pi_1(M'_n)\to \Z_d$, $(M_n,\phi_n)$
    is $S$-torsion in $\Omega^{top}_3(B\Z_d)$ if and only if so is
    $(M'_n,\phi'_n)$, and if it is the case then
    \[
    \lambda(M_n,\phi_n) = \lambda(M'_n,\phi'_n) \text{ in } S^{-1}\Z
    \otimes L^0(\Q(\zeta_{d})).
    \]
  \end{enumerate}
\end{theorem}

Since a homology cobordism is a $\Z_p$-homology cobordism,
Theorem~\ref{theorem:intro-homology-cobordism-invariance} follows
immediately from
Theorem~\ref{theorem:invariants-from-covering-tower}.


\subsection{Algebraic closures and $p$-towers}

For later use, we generalize
Lemma~\ref{lemma:2-connected-maps-induce-t-equivalence} by weakening
the 2-connectedness condition.  For a group $G$, let $\widehat G$ be
the algebraic closure with respect to $\Z_{(p)}$-coefficients in the
sense of~\cite{Cha:2004-1}.  It is known that the association $G
\mapsto \widehat G$ is a functor on the category of groups and there
is a natural transformation $p_G \colon G \to \widehat
G$~\cite{Cha:2004-1}.  The most essential property of $\widehat G$ is
the following: as in case of spaces, we say that a group homomorphism
$f\colon \pi\to G$ is $n$-connected with respect to $R$-coefficients
if $H_i(f;R)=0$ for $i\le n$.  Then, whenever $\pi$, $G$ are finitely
presented and $f\colon \pi \to G$ is 2-connected with respect to
$\Z_{p}$(or equivalently $\Z_{(p)}$)-coefficients, $f$ induces an
isomorphism $\widehat\pi \to \widehat G$.  For proofs of these facts
and more details, see~\cite{Cha:2004-1}.  In fact the functor $G
\mapsto \widehat G$ is initial among those satisfying this inverting
property whenever $f$ is a homomorphism between finitely presented
groups which is 2-connected with respect to $\Z_{p}$-coefficients.

\begin{proposition}
  \label{proposition:algebraic-closure-and-t-equivalences}
  If $X$ and $Y$ are CW-complexes with finite 2-skeletons and $f\colon
  X\to Y$ induces an isomorphism $\widehat{\pi_1(X)} \to
  \widehat{\pi_1(Y)}$, then $f$ is a $p$-tower map.
\end{proposition}

\begin{proof}
  By \cite{Cha:2004-1}, there is a sequence of 2-connected (with
  respect to $\Z_p$-coefficients) maps
  \[
  \pi_1(X)=G_0 \to G_1 \to \cdots
  \]
  on finitely presented groups $G_k$ such that
  $\widehat{\pi_1(X)}\cong\varinjlim G_k$ and the natural map
  $p_{\pi_1(X)}\colon \pi_1(X) \to \widehat{\pi_1(X)}$ is the limit
  map $G_0 \to \varinjlim G_k$.  Since $\pi_1(Y)$ is finitely
  presented, the composition
  \[
  \pi_1(Y) \to \widehat{\pi_1(Y)} \xrightarrow{\cong}
  \widehat{\pi_1(X)}
  \]
  factors through some~$G_k$, so that we have the following
  commutative diagram:
  \[
  \begin{diagram}
    \dgARROWLENGTH=.8\dgARROWLENGTH
    \node{\pi_1(X)} \arrow{e} \arrow{s}
    \node{G_k} \arrow{e}
    \node{\widehat{\pi_1(X)}} \arrow{s,r}{\cong}
    \\
    \node{\pi_1(Y)} \arrow[2]{e} \arrow{ne,..}
    \node[2]{\widehat{\pi_1(Y)}}
  \end{diagram}
  \]
  Since the horizontal maps induce isomorphisms on $H_1(-;\Z_p)$, the
  map $\pi_1(Y) \to G_k$ induces an isomorphism on $H_1(-;\Z_p)$.
  Since $H_2(\pi_1(X);\Z_p) \to H_2(G_k;\Z_p)$ is surjective,
  $H_2(\pi_1(Y);\Z_p) \to H_2(G_k;\Z_p)$ is surjective.  It follows
  that $\pi_1(Y) \to G_k$ is 2-connected with respect to
  $\Z_p$-coefficients.

  Let $Z=K(G_k,1)$ be the Eilenberg-MacLane space.  Since $G_k$ is
  finitely presented, we may assume that $Z$ has finite 2-skeleton.
  Consider maps $X \to Z$ and $Y \to Z$ which induce our $\pi_1(X) \to
  G_k$ and $\pi_1(Y) \to G_k$.  Since $\pi_1(X) \to G_k$ and $\pi_1(Y)
  \to G_k$ are 2-connected with respect to $\Z_p$-coefficients, so are
  $X \to Z$ and $Y \to Z$, and therefore they are $p$-tower maps by
  Lemma~\ref{lemma:2-connected-maps-induce-t-equivalence}.  By
  Remark~\ref{remark:composition-of-t-equivalence}, it follows that
  $f\colon X\to Y$ is a $p$-tower map.
\end{proof}

\begin{remark}
  \label{remark:Levine's-closure-and-t-equivalence}
  In~\cite{Levine:1989-1} Levine defined another algebraic closure of
  a group~$G$ (much earlier than~\cite{Cha:2004-1}).  Levine's
  algebraic closure has the following property which is appropriate
  for studying manifold embeddings into a simply connected ambient
  space: if $\pi$, $G$ are finitely presented, $f\colon \pi \to G$ is
  integrally 2-connected, and $f(\pi)$ normally generates $G$, then
  $f$ induces an isomorphism on Levine's algebraic closures.
  Comparing this with the universal property of the algebraic closure
  $\widehat G$ of \cite{Cha:2004-1}, it follows easily that there is a
  natural transformation from Levine's algebraic closure to $\widehat
  G$.  An immediate consequence is that if $\pi\to G$ induces an
  isomorphism on Levine's algebraic closures, then $\widehat \pi \to
  \widehat G$ is also an isomorphism.  Therefore
  Proposition~\ref{proposition:algebraic-closure-and-t-equivalences}
  applies for $f\colon X\to Y$ which induces an isomorphism on
  Levine's algebraic closures.
\end{remark}

\begin{remark}
  The 2-connectedness assumption in
  Lemma~\ref{lemma:2-connected-maps-induce-t-equivalence} is stronger
  than the assumption of
  Proposition~\ref{proposition:algebraic-closure-and-t-equivalences}.
  For example, a $p$-tower map considered in
  Proposition~\ref{proposition:t-equivalence-of-F-hat-links} induces
  an isomorphism on $\widehat{\pi_1(-)}$ while it is not $H_2$-onto.
\end{remark}

\section{Computation of intersection form defects}
\label{section:basic-properties}

In this section we suppose that $\Gamma$ is endowed with
$\Z\Gamma\to \K$ and $H_4(\Gamma)=0$.

\subsection{Connected sum}

Suppose $M$ and $M'$ are closed 3-manifolds $M$ endowed with
$\phi\colon \pi_1(M) \to \Gamma$ and $\phi'\colon \pi_1(M') \to
\Gamma$, respectively.  Let $\psi\colon \pi_1(M\# M') \to \Gamma$ be
the map induced by $\phi$ and $\phi'$, regarding $\pi_1(M\# M')$ as
the free product $\pi_1(M) * \pi_1(M')$.  Note that any map
$\pi_1(M\#M') \to \Gamma$ is of this form.

\begin{lemma}
  \label{lemma:additivity-under-connected-sum}
  If $(M,\phi)$ and $(M',\phi')$ are $S$-torsion in
  $\Omega^{top}_3(B\Gamma)$, then $(M \# M', \psi)$ is $S$-torsion in
  $\Omega^{top}_3(B\Gamma)$ and
  \[
  \lambda(M \# M', \psi)=\lambda(M,\phi)+\lambda(M',\phi')
  \]
  in $S^{-1}\Z \otimes L^0(\K)$.
\end{lemma}

\begin{proof}
  Choose $W$ and $W'$ such that $\partial W=rM$ and $\partial W'=r'M'$
  over $\Gamma$ for some $r$ and $r'$ in~$S$.  Consider $r'W$ and
  $rW'$ which have boundaries $rr'M$ and $rr'M'$, respectively.  Let
  $M_i$ be the $i$th copy of $M$ in $\partial (r'W)$ for
  $i=1,\ldots,rr'$.  Choose a 4-ball $B_i$ in $r'W$ which is disjoint
  to $M_j$ for $j\ne i$ and intersects $M_i$ at a 3-ball contained in
  $\partial B_i$.  Choose a 4-ball $B_i'$ in $rW'$ for each
  $i=1,\ldots,rr'$ in a similar way.  Let $V = (r'W \cup rW') / \sim$
  where $B_i \subset r'W$ and $B_i' \subset rW'$ are identified for
  each $i$.  It can be seen that $\partial V = rr'(M\# M')$ over
  $\Gamma$.  Therefore $(M\# M',\psi)$ is $S$-torsion in
  $\Omega_3^{top}(B\Gamma)$ and we can compute $\lambda(M\# M',\psi)$
  using the intersection form of~$V$.  Since each $B_i$ is
  contractible, a standard Mayer-Vieotoris argument shows that
  \[
  [\lambda_\K(V)] = r'[\lambda_\K(W)] + r[\lambda_\K(W')],
  \]
  and similarly for the ordinary signature~$\sigma(V)$.  From this the
  desired additivity of $\lambda(-,-)$ follows.
\end{proof}

The following example will be used later to show that our invariant
vanishes for some manifolds:

\begin{lemma}
  \label{lemma:vanishing-for-S^1xD^2}
  Let $M$ be the connected sum of disjoint copies of $S^1\times S^2$.
  Then for any $\phi\colon\pi_1(M) \to \Gamma$, $\lambda(M,\phi)$ is
  well-defined as an element in $L^0(\K)$ and vanishes.
\end{lemma}

\begin{proof}
  By Lemma~\ref{lemma:additivity-under-connected-sum}, we may assume
  that $M=S^1\times S^2$.  Then, letting $W=S^1\times D^3$, $\partial
  W=M$ over~$\pi_1(M)$.  So $(M,\phi)=0$ in $\Omega^{top}_3(B\Gamma)$,
  and $\lambda(M,\phi)$ can be computed from the intersection form
  of~$W$.  Since $W$ has the homotopy type of a 1-complex,
  $H_2(W;\Q)=0=H_2(W;\K)$.  It follows that $\lambda(W,\phi)=0$ for
  any~$\phi$.
\end{proof}

\subsection{Toral sum}

We consider a special case of toral sum described below.  Let $M$ and
$M'$ be closed 3-manifolds, and $T$ and $T'$ are solid tori embedded
in $M$ and $M'$, respectively.  Choose an orientation reversing
homeomorphism $h\colon T'\to T$, and let $N$ be the manifold obtained
by gluing the boundaries of $M-\inte T$ and $M'-\inte T'$
along~$h|_\partial$.  Suppose there is a retract $s\colon M' \to T'$
of~$T'$.  For a homomorphism $\phi\colon \pi_1(M) \to \Gamma$, let
$\phi'$ be the composition
\[
\pi_1(M') \xrightarrow{s_*} \pi_1(T') \xrightarrow{h_*} \pi_1(T) \to
\pi_1(M) \xrightarrow{\phi} \Gamma.
\]
Then $\phi$ and $\phi'$ induce a map $\psi\colon \pi_1(N) \to \Gamma$.

\long\def\ignoreme{
\[
{
  \fboxsep=.3em \fboxrule=0mm
  \begin{diagram}\dgHORIZPAD=0mm \dgVERTPAD=0mm \dgARROWLENGTH=.2em
    \node[2]{\mathstrut} \arrow{e,t,-}{s_*}
    \node{} \arrow{se}
    \\
    \node{\fb{\pi_1(M')}}\arrow[3]{see,b}{\phi'} \arrow{ne,-}
    \node[3]{\fb{\pi_1(T')=\pi_1(T)}}
    \arrow[3]{w} \arrow[3]{e}
    \node[3]{\fb{\pi_1(M)}} \arrow[3]{s,r}{\phi}
    \\ \\ \\
    \node[7]{\fb{\Gamma}}
  \end{diagram}
}
\]
i.e., $\phi'=\phi i_* s$.
}

\begin{lemma}
  \label{lemma:toral-sum-addivitiy}
  \noindent\Nopagebreak
  \begin{enumerate}
  \item $(M',\phi')$ is null-bordant over $\Gamma$.  Consequently,
    $(M',\phi')$ is $S$-torsion in $\Omega^{top}_3(B\Gamma)$.
  \item If $(M,\phi)$ is $S$-torsion in $\Omega^{top}_3(B\Gamma)$,
    then so is $(N,\psi)$, and
    \[
    \lambda(N,\psi) = \lambda(M,\phi)+\lambda(M',\phi') \text{ in }
    S^{-1}\Z \otimes L^0(\K).
    \]
  \end{enumerate}
\end{lemma}

\begin{proof}
  It can be seen that $\Omega_3^{top}(\Z)=0$ from the
  Atiyah-Hirzebruch spectral sequence; indeed each $E^2$ term
  $E^2_{p,q}=H_p(\Z;\Omega^{top}_q)$ vanishes for $p+q=3$ since $B\Z$
  has the homotopy type a 1-complex, namely $S^1$, and
  $\Omega^{top}_q=0$ for $1\le q \le 3$.  Therefore $M'$ endowed with
  $s_*\colon \pi_1(M') \to \pi_1(T')=\Z$ has a null-bordism $W'$, i.e.,
  $\partial W'=M'$ over~$\Z$. By the definition of $\phi'$, $\partial
  W'=M'$ over $\Gamma$ as well.  This proves the first assertion.

  Now suppose that $(M,\phi)$ is $S$-torsion
  in~$\Omega^{top}_3(B\Gamma)$.  Choose $W$ such that $\partial W=rM$
  over $\Gamma$ for some $r\in S$.  Let $V$ be the manifold obtained
  by taking the disjoint union of $W$ and $rW'$ and then attaching the
  $j$th copy of $T$ in $\partial W=rM$ to the $j$th copy of $T'$ in
  $\partial (rW')=rM'$.  It can be seen that $\partial V=rN$
  over~$\Gamma$.  Thus $(N,\psi)$ is $S$-torsion in
  $\Omega^{top}_3(B\Gamma)$.

  By Mayer-Vieotoris, we have an exact sequence
  \begin{multline*}
    H_2(S^1\times D^2;\K)^r \to H_2(W;\K)\oplus H_2(W';\K)^r \to H_2(V;\K)
    \\
    \to H_1(S^1\times D^2;\K)^r \to H_1(W;\K)\oplus H_1(W';\K)^r.
  \end{multline*}
  $H_2(S^1\times D^2;\K)=0$ obviously.  The map 
  \[
  H_1(S^1\times D^2;\K) = H_1(T';\K) \to H_1(M';\K) \to H_1(W';\K)
  \]
  is injective since it has a left inverse.  Therefore it follows that
  \[
  H_2(V;\K) \cong H_2(W;\K)\oplus H_2(W';\K)^r.
  \]
  From this we obtain an orthogonal decomposition of the intersection
  form, namely
  \[
  [\lambda_\K(V)] = [\lambda_\K(W)]+r[\lambda_\K(W')].
  \]
  An analogous formula for the ordinary signature is proved by a
  similar argument.  From this the desired additivity follows.
\end{proof}

For later use, we state the following lemma which was proved in the
proof of Lemma~\ref{lemma:toral-sum-addivitiy}~(1).

\begin{lemma}
  \label{lemma:null-bordism-over-Z}
  If $\phi\colon\pi_1(M)\to \Gamma$ factors through $\Z$, then
  $(M,\phi)$ is null-bordant over $\Gamma$ so that $\lambda(M,\phi)$
  is well-defined as an element in $L^0(\K)$.
\end{lemma}

\subsection{$p$-Towers of a toral sum}

Suppose $M$, $M'$, $T$, $T'$, $s$, and $N$ are as above.  Note that
$M-\inte T$ can be viewed as a subspace of both $M$ and~$N$.  From the
existence of the retract $s \colon M' \to T'$, it follows that the
inclusion $M-\inte T \to M$ extends to a map $f\colon N \to M$; we
give a proof below.  Since $s$ is a retract, we have the following
commutative diagram, where the vertical maps are induced by the
inclusions:
\[
\begin{diagram}
  \node{\pi_1(\partial T')} \arrow{s} \arrow{se}
  \\
  \node{\pi_1(M'-\inte T')} \arrow{e,b}{(s|)_*}
  \node{\pi_1(T')}
\end{diagram}
\]
it follows that the inclusion $\partial T' \to T'$ extends to $r\colon
M'-\inte T' \to T'$.  Now, define $f\colon N \to M$ by
\[
f\colon N = (M-\inte T) \cup_\partial (M'-\inte T') \xrightarrow{\id
  \cup r} (M-\inte T) \cup_\partial T' \cong M. \qedhere
\]

Suppose
\[
M_n \to \cdots \to M_1 \to M_0 = M
\]
is a $p$-tower determined by $\{\phi_i\colon \pi_1(M_i) \to
\Gamma_i\}$.  We consider the pullback $p$-tower
\[
N_n \to \cdots \to N_1 \to N_0 = N
\]
of $N$ via~$f$.  By ``lifting'' the toral sum structure of $N$, we can
obtain $N_n$ from $M_n$ via toral sum.  Note that $T$ may not be
lifted to $M_n$; in general, the pre-image of $T\subset M$ under $M_n
\to M$ is a disjoint union of solid tori, say $\widetilde T_{1},
\widetilde T_2, \ldots$, and the restriction $\widetilde T_{j} \to T$
is a $r_{j}$-fold cyclic cover for some divisor $r_{j}$
of~$\prod_{i=0}^{n-1} |\Gamma_i|$.  (In general, $r_{j}$ depends on
$j$ since $M_n \to M$ may not be a regular cover.)  Therefore, we need
to consider a corresponding $r_{j}$-fold covers of $M'$ to construct
$N_n$ from~$M_n$.  Details are as follows.  Let $\widetilde M'_{j}$ be
the $r_{j}$-fold cyclic cover of $M'$ determined by
\[
\pi_1(M') \xrightarrow{s_*} \pi_1(T')=\Z \xrightarrow{proj.}
\Z_{r_{j}}
\]
and $\widetilde T'_{j} \subset \widetilde M'_{j}$ be the pre-image of
$T'$.  Obviously each $\widetilde T'_{j}$ is a solid torus and
$\widetilde T'_{j} \to T'$ is a $r_{j}$-fold cyclic cover which can be
identified with $\widetilde T_{j} \to T$.  Let
\[
N_n=\Big[M_n -\bigcup_{j} \inte \widetilde T_{j} \Big] \cup
\Big[\bigcup_{j} (\widetilde M'_{j}-\inte \widetilde T'_{j}) \Big]
\Big/ \partial \widetilde T_{j} \sim \partial \widetilde T'_{j} \text{
  for } j=1,2,\ldots
\]
Also, $s$ lifts to $\widetilde s_{j}\colon \widetilde M'_{j} \to
\widetilde T'_{j}$.  For any $\phi_n\colon\pi_1(M_n) \to \Z_d$ with
$d$ a power of $p$, $\phi_n$ and the map $\widetilde\phi'_{j}$ defined
to be
\[
\widetilde\phi'_j\colon \pi_1(\widetilde M'_{j})
\xrightarrow{(\widetilde s_{j})_*} \pi_1(\widetilde T'_{j})
=\pi_1(\widetilde T_{j}) \to \pi_1(M_n) \xrightarrow{\phi_n} \Z_d.
\]
induce a map $\psi_n\colon \pi_1(N_n) \to \Z_d$.  The following lemma
follows immediately by applying Lemma~\ref{lemma:toral-sum-addivitiy}
to the toral decomposition of~$N_n$.

\begin{lemma}
  \label{lemma:tower-of-toral-sum}
  $(N_n,\psi_n)$ is $S$-torsion in $\Omega^{top}_3(B\Z_{d})$ if and
  only if so is $(M_n,\phi_n)$, and if it is the case,
  \[
  \lambda(N_n,\psi_n)=\lambda(M_n,\phi_n) + \sum_{j}
  \lambda(\widetilde M'_{j}, \widetilde\phi'_{j}).
  \]
\end{lemma}

\subsection{Infection by a knot}
\label{subsection:knot-infection}

The following toral sum construction will be used to as a main tool to
construct examples in our applications.  Let $K$ be a knot in~$S^3$.
The \emph{zero-surgery manifold} $M_K$ of $K$ is defined to be the
manifold obtained by filling in the exterior of $K$ with a solid
torus, say $T'$, in such a way that the preferred longitude of $K$
bounds a disk in~$T'$.  Note that there is a retract $s\colon M_K \to
T'$ as assumed in the previous subsection.  Let $M$ be a closed
3-manifold and $\alpha$ be a simple closed curve in $M$ with tubular
neighborhood $T$.  Forming a toral sum of $M$ and $M_K$ via a
homeomorphism $h\colon T \cong T'$, we obtain a new manifold~$N$.  We
say that $N$ is obtained from $M$ by \emph{infection by $K$
  along~$\alpha$}.  Note that a meridian of $K$ is identified with a
parallel of $\alpha$, and a preferred longitude of $K$ is identified
with a meridian of~$\alpha$.

As in the above subsection, let
\begin{gather*}
  M_n \to \cdots \to M_1 \to M_0 = M,\\
  N_n \to \cdots \to N_1 \to N_0 = N
\end{gather*}
be a $p$-tower of~$M$ and the $p$-tower of $N$ induced by it,
respectively.  Let $X_r$ be the $r$-fold cyclic cover of~$M_K$.  Then
by Lemma~\ref{lemma:tower-of-toral-sum} $\lambda(N_n,-)$ is determined
by $\lambda(M_n,-)$ and $\lambda(X_{r_{j}}, -)$ where $r_j$ is as
above.  The involved characters of $X_{r_{j}}$ are of the following
form: the canonical surjection $\pi_1(M_K) \to \Z$ sending the
(positive) meridian $\mu$ of $K$ to $1$ restricts to a surjection
$\pi_1(X_r) \to r\Z$, viewing $\pi_1(X_r)$ as a subgroup of
$\pi_1(M_K)$.  Composing it with an appropriate map $r\Z\to \Z_d$, we
define a map $\phi_r^{s,d}\colon \pi_1(X_r) \to \Z_d$ sending $\mu^r$
to $s\in \Z_d$.  Note that $\lambda(X_r,\phi_r^{s,d})$ is always
well-defined as an element in $L^0(\Q(\zeta_d))$ by
Lemma~\ref{lemma:null-bordism-over-Z}.

\begin{lemma}
  \label{lemma:lambda-of-surgery-manifolds}
  Let $A$ be a Seifert matrix of $K$.  Then
  \[
  \lambda(X_r,\phi_r^{s,d})=[\lambda_{r}(A,\zeta_d^s)]-[\lambda_{r}(A,1)]
  \quad\text{in } L^0(\Q(\zeta_d))
  \]
  where $[\lambda_{r}(A,\omega)]$ is the Witt class of (the
  nonsingular part of) the hermitian form represented by the following
  $r\times r$ block matrix:
  \[
  \lambda_{r}(A,\omega) =
  \begin{bmatrix}
    \vphantom{\ddots} A+A^T & -A & & & -\omega^{-1} A^T\\
    \vphantom{\ddots} -A^T & A+A^T & -A \\
    \vphantom{\ddots} & -A^T & A+A^T & \ddots \\
    \vphantom{\ddots} & & \ddots & \ddots & -A \\
    \vphantom{\ddots} -\omega A & & & -A^T & A+A^T
  \end{bmatrix}_{r\times r}
  \]
  For $r=1,2$, $\lambda_{r}(A,\omega)$ should be understood as
  \[
  \begin{bmatrix}
    (1-\omega)A+(1-\omega^{-1})A^T
  \end{bmatrix}
  \quad\text{and}\quad
  \begin{bmatrix}
    A+A^T & -A-\omega^{-1}A^T \\
    -A^T-\omega A & A+A^T
  \end{bmatrix}.
  \]
\end{lemma}

We postpone the proof of Lemma~\ref{lemma:lambda-of-surgery-manifolds}
to the Appendix.

Now, from Lemma~\ref{lemma:tower-of-toral-sum} and
lemma~\ref{lemma:lambda-of-surgery-manifolds}, we obtain a formula for
the intersection form defect invariants of manifolds infected by
knots:

\begin{corollary}
  \label{corollary:lambda-of-infected-manifold}
  Let $\tilde\alpha_1, \tilde\alpha_2,\ldots \subset M_n$ be the
  components of the pre-image of $\alpha\subset M$, and $r_j$ be the
  degree of the covering map $\tilde\alpha_j \to \alpha$.  Let
  $\phi_n\colon \pi_1(M_n) \to \Z_d$ be a character with $d$ a power
  of $p$ and $\psi_n\colon \pi_1(N_n) \to \Z_d$ be the character
  induced by~$\phi_n$.  Then
  \[  
  \lambda(N_n,\psi_n)=\lambda(M_n,\phi_n)+\sum_{j} \Big(
  [\lambda_{r_{j}}(A,\zeta_d^{\phi_n([\tilde\alpha_j])})]-[\lambda_{r_{j}}(A,1)]
  \Big)
  \]
\end{corollary}

In the previous subsection, we defined a map $f\colon N \to M$ for $N$
obtained from $M$ by a toral sum with $M'$; in general, it can be seen
easily that $f$ is not necessarily a $p$-tower map in the sense of
Definition~\ref{definition:t-equivalence}.  However, for knot
infection, we have the following result:

\begin{proposition}
  \label{proposition:t-equivalence-of-infected-manifold}
  Suppose $N$ is obtained from $M$ by knot infection along a solid
  torus $T\subset M$.  Then the map $f\colon N \to M$ is a $p$-tower
  map for any prime~$p$.
\end{proposition}

\begin{proof}
  Let $E$ be the exterior of~$K$.  Using the fact that $H_*(E)\cong
  H_*(S^1\times D^2)$, it can be shown easily that there is a homology
  equivalence $h\colon E \to S^1\times D^2$ which restricts to a
  homeomorphism on the boundary.  Using $h$, the inclusion $M-\inte T
  \to M$ extends to
  \[
  f\colon N=(M-\inte T) \cup_\partial E \xrightarrow{\id\cup h}
  (M-\inte T) \cup_\partial (S^1\times D^2) = M.
  \]
  By a Mayer-Vieotoris argument, it follows that $f$ is a homology
  equivalence.  By
  Lemma~\ref{lemma:2-connected-maps-induce-t-equivalence}, $f$ is a
  $p$-tower map.
\end{proof}

\subsection{Invariants of $L^0(\Q(\zeta_d))$ and norm residue symbols}
\label{subsection:norm-residue-symbol}

We give a quick review of known invariants of $L^0(\Q(\zeta_d))$.  For
more details, the reader is referred to Milnor-Husemoller's
book~\cite{Milnor-Husemoller:1973-1}.  See also
\cite[Chapter~3]{Cha:2003-1}.  

For a nonsingular hermitian form $\lambda$ on a finite dimensional
$\Q(\zeta_d)$-space~$V$, the following invariants are defined:

\begin{description}
\item[Signature] The natural inclusion $\Q(\zeta_d) \to \C$ gives rise to
  \[
  \sign\colon L^0(\Q(\zeta_d)) \to L^0(\C) \cong \Z.
  \]
  In other words, $\sign\lambda$ is the signature of $\lambda$ viewed
  as a hermitian form over~$\C$.
\item[Rank modulo $2$] Let $r$ be the $\Q(\zeta_d)$-dimension of the
  underlying space~$V$.  Since every hyperbolic form has even
  dimension,
  \[
  \rank \lambda = \text{modulo $2$ residue class of }r \in \Z_2
  \]
  is an invariant of the Witt class of $\lambda$.
\item[Discriminant] Since $\lambda$ is nonsingular and hermitian, the
  determinant of a matrix representing $\lambda$ is nonzero and fixed
  under the involution, i.e., $\det \lambda \in \Q(\zeta_d +
  \zeta_d^{-1})^\times$.  The discriminant of $\lambda$ is defined by
  \[
  \dis \lambda= (-1)^{\frac{r(r+1)}{2}} \det \lambda \in
  \frac{\Q(\zeta_d^{\vphantom{-1}}+\zeta_d^{-1})^\times}{\{z\cdot \bar
    z \mid z\in \Q(\zeta_d)^\times\}}
  \]
  Here $z\to \bar z$ denotes the involution on $\Q(\zeta_d)$ induced
  by $\zeta_d^{\mathstrut} \to \zeta_d^{-1}$.  It can be verified
  easily that $\dis \lambda$ is an invariant of the Witt class
  of~$\lambda$.
\end{description}

We remark that $\dis \lambda$ is regarded as an element of a
multiplicative group, while $\sign \lambda$ and $\rank\lambda$ are in
additive groups.  It is known that $\{\sign,\rank,\dis\}$ is a
complete set of invariants of $L^0(\Q(\zeta_d))$ if $d>2$, i.e,
$\zeta_d\ne \pm 1$ \cite{Milnor-Husemoller:1973-1}.

Since $\dis \lambda$ is defined up to multiplication by $z\cdot \bar
z$ where $z\in \Q(\zeta_d)$, detecting a nonvanishing value of
$\dis\lambda$ is a nontrivial problem.  For this purpose, we will
employ some algebraic number theory.  The remaining part of this
subsection is devoted to a quick summary of necessary results.

We call $x\in \Q(\zeta_d^{\vphantom{-1}}+\zeta_d^{-1})^\times$ a
\emph{norm} if $x=z\cdot \bar z$ for some $z \in \Q(\zeta_d)^\times$.
In fact, $L=\Q(\zeta_d)$ is a quadratic extension over
$K=\Q(\zeta_d^{\mathstrut}+\zeta_d^{-1})$, and the norm map $
N^L_K\colon L^\times \to K^\times$ is given by $N^L_K(z)=z\bar z$.
Note that $L=K(\sqrt{D})$ where
\[
D=(\zeta_d^{\vphantom{-1}}+\zeta_d^{-1})^{\vphantom{-}2}_{\vphantom{d}}-4 \in K.
\]

The problem of deciding whether $x\in K^\times$ is a norm or not can
be reduced to the computation of \emph{norm residue symbols}.  For a
prime $\p$ of $K$ (that is, a prime ideal in the ring of algebraic
integers of $K$) and $a,b\in K^\times$, there is defined the norm
residue symbol $(a,b)_\p$ satisfying the following properties:

\begin{enumerate}
\item $(a,b)_\p = +1$ or $-1$.
\item $(a,b)_\p$ is symmetric and bilinear, i.e., $(a,b)_\p =
  (b,a)_\p$ and $(aa',b)_\p = (a,b)_\p (a',b)_\p$.  Consequently
  $(a,b)_\p = (a^{-1}, b)_\p$.
\item If $x$ is a norm, then $(x,D)_\p = 1$ for all~$\p$.
\end{enumerate}

For our purpose, (3) is essential.  From (3) it easily follows that
the induced homomorphism
\[
(-,D)_\p\colon
\frac{\Q(\zeta_d^{\vphantom{-1}}+\zeta_d^{-1})^\times}{\{z\cdot \bar z
  \mid z\in \Q(\zeta_d)^\times\}} \to \{\pm 1\}
\]
is well-defined for each prime~$\p$.  We remark that the converse of
(3) holds if we think of the norm residue symbols for archimedian
valuations as well as those for primes.  For more detailed discussions
on the norm residue symbol, the reader is referred
to~\cite{Serre:1979-1,Cassels-Froehlich:1967-1}.

For the case that $K=\Q$ and $\p$ is (the ideal generated by) an odd
prime~$p$, which will be used in our applications discussed in later
sections, the following lemma provides a formula for the computation
of the norm residue symbols:

\begin{lemma}
  \label{lemma:computation-norm-residue-symbol}
  Suppose $p$ is an odd prime and $a,b\in \Q$.  Then
  \[
  (a,b)_p = \bigg((-1)^{v_p(a)v_p(b)} \cdot
  \frac{a^{v_p(b)}}{b^{v_p(a)}}\bigg)^{\frac{p-1}{2}} \quad
  \text{in }\Z_p,
  \]
  where $v_p(x)$ is the valuation associated to $p$, i.e., writing
  $x=p^r(s/t)$ where $s,t$ are integers relatively prime to $p$,
  $v_p(x)$ is defined to be~$r$.
\end{lemma}

Note that $x^{(p-1)/2}\equiv \pm 1\pmod p$ for any $x\not\equiv 0
\pmod p$.  For a proof of
Lemma~\ref{lemma:computation-norm-residue-symbol}, refer
to~\cite{Serre:1979-1}.  Also, \cite[Section 3.4]{Cha:2003-1} provides
a summary of results including the general case where $K$ is not
necessarily~$\Q$, for non-experts of algebraic number theory.

\section{Exotic homology cobordism types of rational homology
  3-spheres with vanishing signature invariants}
\label{section:homology-cobordism-for-b1=0}

In this section we construct rational homology spheres $\Sigma_i$
satisfying Theorem~\ref{theorem:exotic-rational-spheres-intro}.
Indeed, we will show that $\Sigma_i$ and $\Sigma_j$ are not
$\Z_2$-homology cobordant for $i\ne j$, while each $\Sigma_i$ has
vanishing known signature invariants.

\subsection{Construction of rational homology spheres by infection}
\label{subsection:construction=of-exotic-examples}

We construct the manifolds $\Sigma_i$ as follows.  First we describe a
``seed'' manifold~$M$.  Choose two lens spaces $L_1=L(r_1,s_1)$ and
$L_2=L(r_2,s_2)$ such that $r_i$ is a power of $2$ and all the
multisignatures of $L_i$ vanish for $i=1,2$.  (For example, if we
choose $(r_i, s_i)$ such that $L_i$ bounds a rational 4-ball, then
$L_i$ has vanishing multisignatures.)  Let $M=L_1\# L_2$.
Figure~\ref{fig:infection-on-Q-sphere} illustrates a Kirby diagram
of~$M$, together with a simple closed curve $\alpha$ in $M$.
($r_i/s_i$ represents the surgery slope.)  It can be easily seen that
$\pi_1(M)=\Z_{r_1} * \Z_{r_2}$, and denoting the generators of the
$\Z_{r_1}$ and $\Z_{r_2}$ factors by $x$ and $y$, $\alpha$ represents
the commutator $(x,y)=xyx^{-1}y^{-1}\in\pi_1(M)$.

\begin{figure}[ht]
  \begin{center}
    \includegraphics[scale=.9]{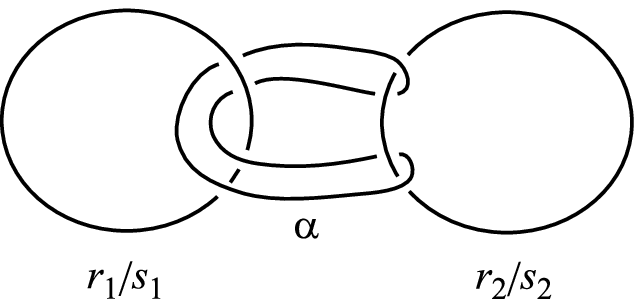}
    \caption{}\label{fig:infection-on-Q-sphere}
  \end{center}
\end{figure}

For a given integer $a$, consider the knot $K_a$ shown in
Figure~\ref{fig:infection-amph-knot}.  Note that $K_a$ is (negatively)
amphichiral, that is, isotopic to its (concordance) inverse~$-K_a$.
The knot $K_a$~has an obvious Seifert surface of genus one, on which
the following Seifert matrix is defined:
\[
A=
\begin{bmatrix}
  a & 1 \\
  0 & -a
\end{bmatrix}
\]


Let $\Sigma_0=M$.  For $i>0$ the manifolds $\Sigma_i$ are obtained by
infection on $M$ along $\alpha$ by $K_{a_i}$, where $a_1,a_2,\ldots$
are integers which will be specified later.  By (the proof of)
Proposition~\ref{proposition:t-equivalence-of-infected-manifold},
there is a homology equivalence $\Sigma_i \to \Sigma_0$ for all~$i$.
Before discussing how to choose the $a_i$, we investigate the Witt
class defect invariants of the infected manifolds.  Let $K=K_a$ and
$N$ be $M$ infected along $\alpha$ by~$K$.

Let
\[
\Gamma_0=\Z_{r_1} \oplus \Z_{r_2}, \quad \Gamma_1=\Z_2, \quad
\Gamma_2=\Z_4.
\]
We will consider $2$-towers of height two with deck transformation
groups $\Gamma_0$ and $\Gamma_1$, and intersection form defect
invariants associated to $\Gamma_2$-valued characters.  In this
section we always assume that a $p$-tower consists of connected
spaces.

For the special case of $a=0$, that is, when $K$ is unknotted and
$N=M$, the invariant vanishes for any such tower and character:

\begin{lemma}
  \label{lemma:vanishing-for-seed-manifold}
  For any $2$-tower $M_2 \to M_1 \to M_0 = M$ with deck transformation
  groups $\Gamma_0$ and $\Gamma_1$ and for any $\phi_2\colon
  \pi_1(M_2) \to \Gamma_2$, $(M,\phi_2)=0$ in
  $\Omega_3^{top}(B\Gamma_2)$, and $\lambda(M_2,\phi_2)=0$ in
  $L^0(\Q(\sqrt{-1}))$.
\end{lemma}

\begin{proof}
  The map $\phi_0\colon \pi_1(M) \to \Gamma_0$ inducing $M_1 \to M$
  factors through $H_1(M)$, which has the same cardinality as
  $\Gamma_0$.  It follows that $\phi_1$ is identical with the
  abelianization map, since $\phi_1$ is surjective.  (Note that the
  $M_i$ are assumed to be connected.)  Therefore $M_1$ is the
  universal abelian cover of~$M$.  Note that $M=L_1 \# L_2$ and the
  universal (abelian) cover of each $L_i$ is $S^3$.  From this it
  follows that $M_1$ is a connected sum of disjoint copies of
  $S^1\times S^2$, and therefore so is~$M_2$.  By
  Lemma~\ref{lemma:vanishing-for-S^1xD^2}, the conclusion follows.
\end{proof}

Next, we investigate the general case:

\begin{proposition}
  \label{proposition:discriminant-of-Q-sphere}
  \mbox{}\Nopagebreak
  \begin{enumerate}
  \item For any $2$-tower $N_2 \to N_1 \to N_0 = N$ with deck
    transformation groups $\Gamma_0$ and $\Gamma_1$ and for any
    $\psi_2\colon \pi_1(N_2) \to \Gamma_2$, $\lambda(N_2,\psi_2)$ is
    always well-defined as an element in $L^0(\Q(\sqrt{-1}))$, and
    \[
    \dis\lambda(N_2, \psi_2) = (2a^2+1)^{n_1} (2a^4+4a^2+1)^{n_2}
    \]
    for some integers $n_1$ and~$n_2$.
  \item For some $2$-tower $N_2 \to N_1 \to N_0 = N$ with deck
    transformation groups $\Gamma_0$ and $\Gamma_1$ and for some
    $\psi_2\colon \pi_1(N_2) \to \Gamma_2$,
    \[
    \dis \lambda(N_2,\psi_2) = (2a^2+1)(2a^4+4a^2+1).
    \]
  \end{enumerate}
  Here, $\dis \lambda(N_2,\psi_2)$ is understood as an element in
  $\Q^\times/\{z\bar z \mid z\in \Q(\sqrt{-1})^\times\}$.
\end{proposition}

In the proof of
Proposition~\ref{proposition:discriminant-of-Q-sphere}, we need the
following lemma.  We define the \emph{(symmetrized) Alexander
  polynomial} $\Delta_A(t)$ of a $2g\times 2g$ matrix $V$ by
\[
\Delta_V(t)=t^{-g}\cdot \det(tV-V^T).
\]
We remark that, in contrast to the case of knots, there is no $(\pm
t^r)$-factor ambiguity in the definition of the Alexander polynomial
of a \emph{matrix}.  $\Delta_V(1)$ is always~$+1$ for any Seifert
matrix~$V$.  For our Seifert matrix $A$ of the knot $K$, we have
\[
\Delta_A(t)=-a^2t + (2a^2+1) -a^2 t^{-1}.
\]

\begin{lemma}
  \label{lemma:discriminant-of-Seifert-matrix}
  \mbox{}\Nopagebreak
  \begin{enumerate}
  \item If $\omega=\zeta_d^k$ and $\Delta_A(\omega) \ne 0$, then
    \[
    \dis [\lambda_1(A,\omega)]= \Delta_A(\omega) \quad \text{in }
    \frac{\Q(\zeta_d^{\vphantom{-1}}+\zeta_d^{-1})^\times}{\{z\bar z
      \mid z\in \Q(\zeta_d)^\times\}}.
    \]
  \item If $\omega=\zeta_d^k$ and $\Delta_A(\sqrt{\omega}) \ne 0 \ne
    \Delta_A(-\sqrt{\omega})$, then
    \[
    \dis[\lambda_2(A,\omega)]=
    \Delta_A(\sqrt{\omega})\Delta_A(-\sqrt{\omega}) \quad \text{in }
    \frac{\Q(\zeta_d^{\vphantom{-1}}+\zeta_d^{-1})^\times}{\{z\bar z
      \mid z\in \Q(\zeta_d)^\times\}},
    \]
    where $\sqrt{\omega}$ and $-\sqrt{\omega}$ denote the zeros of
    $x^2=\omega$.
  \end{enumerate}
\end{lemma}

Indeed it turns out that
Lemma~\ref{lemma:discriminant-of-Seifert-matrix} generalizes to the
case of $\lambda_r(A,\omega)$ with $r>2$ in an obvious form.  Since we
do not need it, we do not address the case that $r>2$; below we prove
it for $r=1,2$ by a straightforward computation.

\begin{proof}
  (1) For $\omega=1$,
  \[
  \lambda_1(A,\omega) = (1-\omega)A + (1-\omega^{-1})A^T
  \]
  is a zero matrix, and so $\dis[\lambda_1(A,\omega)]=1=\Delta_A(1)$
  as claimed.  Suppose $\omega\ne 1$ and $A$ is $2g\times 2g$.  It can
  be easily verified that
  \[
  (-1)^{g(2g+1)} \det \lambda_1(A,\omega) = (\omega-1)^g
  (\omega^{-1}-1)^g \Delta_A(\omega).
  \]
  Since $\Delta_A(\omega)\ne 0$, the matrix $\lambda_1(A,\omega)$ is
  nonsingular.  So $\dis[\lambda_1(A,\omega)]=\Delta_A(\omega)$.

  (2) For $\omega=1$, the matrix 
  \[
  \lambda_2(A,\omega)=
  \begin{bmatrix}
    A+A^T & -A-\omega^{-1}A^T \\
    -A^T-\omega A & A+A^T
  \end{bmatrix}
  \]
  is singular.  By a simple basis change it can be seen that the
  nonsingular part of $\lambda_2(A,1)$ is given by $A+A^T$.  (Note
  that $A+A^T$ is nonsingular for any Seifert matrix~$A$.)  Therefore
  \[
  \dis \lambda_2(A,1)=(-1)^{g(2g+1)} \det(A+A^T) =
  \Delta_A(-1)=\Delta_A(1)\Delta_A(-1)
  \]
  as claimed.  Suppose $\omega\ne 1$.  To compute $\det
  \lambda_2(A,\omega)$, we make a variable change: let
  $G=(A-A^T)^{-1}A$.  Then $(A-A^T)^{-1}A^T=G-1$.  (We denote an
  identity matrix by $1$.)  Since $\det(A-A^T)=1$, we have
  \[
  \det\lambda_2(A,\omega)=\det
  \begin{bmatrix}
    2G-1 & -G-\omega^{-1}(G-1) \\
    -(G-1)-\omega G & 2G-1
  \end{bmatrix}.
  \]
  Now, viewing the above matrix as one over the commutative domain
  $\Q(G)$, the determinant can be computed in a straightforward way;
  this gives us
  \begin{align*}
    \det\lambda_2(A,\omega) &=\det \Big[ (1-\omega) (1-\omega^{-1})
    \Big(G-\frac{1}{1-\sqrt{\omega}}\Big)
    \Big(G-\frac{1}{1+\sqrt{\omega}}\Big) \Big] \\
    &= \det\Big[ A-\frac{1}{1+\sqrt{\omega}}(A-A^T) \Big] \det \Big[
    A-\frac{1}{1-\sqrt{\omega}}(A-A^T) \Big]\\
    &= \Big(\frac{1}{1-\omega}\Big)^g
    \Big(\frac{1}{1-\omega^{-1}}\Big)^g
    \Delta_A(\sqrt{\omega})\Delta_A(-\sqrt{\omega}).
  \end{align*}
  It follows that $\lambda_2(A,\omega)$ is nonsingular and
  $\dis[\lambda_2(A,\omega)] =
  \Delta_A(\sqrt{\omega})\Delta_A(-\sqrt{\omega})$ for $\omega\ne 1$.
\end{proof}

\begin{proof}[Proof of
  Proposition~\ref{proposition:discriminant-of-Q-sphere}~(1)]
  Since $N$ is obtained from $M$ by knot infection, there is a
  $2$-tower $M_2 \to M_1 \to M_0=M$ which gives rise to $N_2 \to N_1
  \to N_0=N$ and there is $\phi_2\colon\pi_1(M_2) \to \Gamma_2$
  inducing the given $\psi_2$ via pullback by
  Proposition~\ref{proposition:t-equivalence-of-infected-manifold}.
  From this it follows that $\lambda(N_2,\psi_2)$ is well-defined as
  an element of $L^0(\Q(\sqrt{-1}))$, by
  Lemma~\ref{lemma:toral-sum-addivitiy} and
  Lemma~\ref{lemma:vanishing-for-seed-manifold}.

  Let $\tilde\alpha_1,\ldots,\tilde\alpha_k\subset M_2$ be the
  components of the pre-image of $\alpha\subset M_0$.  We claim that
  the degree $r_j$ of $\tilde\alpha_j \to \alpha$ is either $1$ or
  $2$.  For, since $\Gamma_0$ is abelian and $\alpha=(x,y)$ is in the
  commutator subgroup $[\pi_1(M),\pi_1(M)]$, the preimage of $\alpha$
  in $M_1$ consists of simple closed curves which are lifts
  of~$\alpha$.  Since $M_2 \to M_1$ is a double covering, its
  restriction on each $\tilde \alpha_j$ is either one-to-one or
  two-to-one.  The claim follows from this.

  Now, from Corollary~\ref{corollary:lambda-of-infected-manifold}, it
  follows that $\lambda(N_2,\psi_2)$ is a linear combination of terms
  of the form $[\lambda_r(A,\sqrt{-1}^s)]$ where $r\in \{1,2\}$ and
  $s\in \{0,1,2,3\}$.  By a straightforward computation using
  Lemma~\ref{lemma:discriminant-of-Seifert-matrix}, we immediately
  obtain the following, which we state as a lemma for later use.

  \begin{lemma}
    \label{lemma:discriminant-of-amphichiral-infection-knots}
    For the Seifert matrix $A$ of our $K_a$,
    \begin{align*}
      \dis[\lambda_1(A,\pm1)] & = 1,\\
      \dis[\lambda_1(A,\pm\sqrt{-1})] &= 2a^2+1, \\
      \dis[\lambda_2(A,\pm1)] & = 1, \text{ and}\\
      \dis[\lambda_2(A,\pm\sqrt{-1})] &= 2a^4+4a^2+1 \quad \text{in }
      \smash[t]{\frac{\Q^\times}{\{z\bar z \mid z\in
          \Q(\sqrt{-1})^\times\}}}.
    \end{align*}
  \end{lemma}

  \noindent This completes the proof of
  Proposition~\ref{proposition:discriminant-of-Q-sphere}~(1).
\end{proof}

\subsection{Combinatorial computation of intersection form defects}
\label{subsection:combinatorial-computation}

In this subsection, we show
Proposition~\ref{proposition:discriminant-of-Q-sphere}~(2) by
computing the invariant for a specific tower and character.  In order
to compute the invariant using
Corollary~\ref{corollary:lambda-of-infected-manifold}, we need to
understand explicitly the behaviour of the pre-image of the infection
curve $\alpha\subset M$.  As a general technique for this, we will
consider a $p$-tower map $X\to M$ of a 2-complex $X$, from which the
pre-images of the infection curve $\alpha$ can be read off
algorithmically.  This provides a combinatorial method to compute the
intersection form defect invariants.  Since we will use the same
technique again in later sections, we illustrate how this method
proceeds in detail.

For convenience, for a 1-complex $X$ and a group $G$, we will describe
a map $\pi_1(X) \to G$ as an assignment of elements in $G$ to 1-cells
of~$X$; such an assignment defines (the homotopy class of) a map $X
\to K(G,1)$ sending the $0$-skeleton of $X$ to a basepoint and sending
1-cells of $X$ to paths in $K(G,1)$ representing the associated
elements in~$G$.

We construct the complex $X$ and a tower of height two with deck
transformation groups $\Gamma_0 = \Z_{r_1}\oplus \Z_{r_2}$ and
$\Gamma_1=\Z_2$ as follows:

\begin{enumerate}
\item[(0)] We start with $S^1\vee S^1$, which is a 1-complex with one
  0-cell $*$ and two 1-cells that we denote by $c_1$ and~$c_2$.  Let
  $X=X_0$ be the complex obtained by attaching two 2-cells to $S^1\vee
  S^1$ along $c_1^{r_1}$ and~$c_2^{r_2}$ respectively.
\item[(1)] Let $X_1$ be the universal abelian cover of~$X_0$.  Then
  $X_1$ is the union of $r_1+r_2$ 2-disks which are lifts of the
  2-cells of $X_0$, and has the homotopy type of its 1-dimensional
  subcomplex (which is indeed a strong deformation retract of $X_1$),
  as illustrated in Figure~\ref{fig:covers-of-Q-sphere}; $\tilde c_1$,
  $\tilde c_2$ are the lifts of $c_1$ and $c_2$ based at the basepoint
  $*\in X_2$.
\item[(2)] Assigning $1\in\Gamma_1=\Z_2$ to the 1-cell denoted by
  $e_1$ in Figure~\ref{fig:covers-of-Q-sphere} and $0\in \Gamma_1$ to
  all the other 1-cells, we obtain a map $\pi_1(X) \to \Gamma_1=\Z_2$
  which gives rise to a double cover $X_2$ of~$X_1$.  $X_2$~has the
  homotopy type of the 1-complex (which is a strong deformation
  retract of $X_2$) illustrated in
  Figure~\ref{fig:covers-of-Q-sphere}.  $*$~is the basepoint.
\end{enumerate}
Finally, assigning $1\in\Gamma_2$ to the 1-cell denoted by $e_2$ in
Figure~\ref{fig:covers-of-Q-sphere} and $0\in\Gamma_2$ to the other
1-cells, we define a map $\phi_2 \colon \pi_1(X_2) \to \Gamma_2=\Z_4$.

\begin{figure}[ht]
  \begin{center}
    \includegraphics[scale=.65]{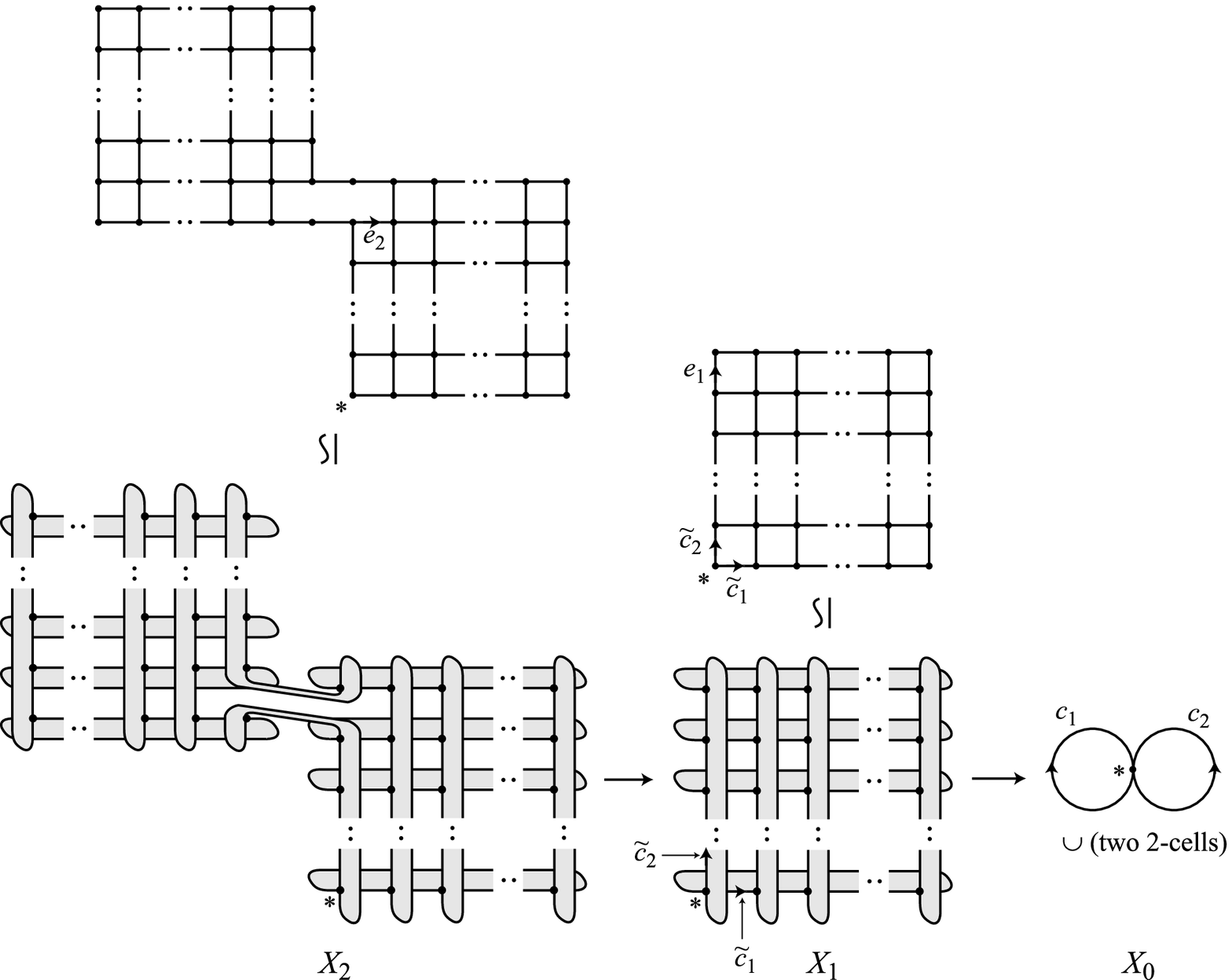}
    \caption{}\label{fig:covers-of-Q-sphere}
  \end{center}
\end{figure}

Obviously there is a map $X \to M$ which induces an isomorphism
$\pi_1(X) \to \pi_1(M)$ sending $[c_1]$ and $[c_2]$ to $x$ and $y$,
respectively.
$X\to M$ is a $p$-tower map for any prime~$p$ since $\pi_1(X) \cong
\pi_1(M)$ (or alternatively by
Lemma~\ref{lemma:2-connected-maps-induce-t-equivalence}).  Therefore,
there is a unique $2$-tower $M_2 \to M_1 \to M_0=M$ which gives rise
to our tower of $X$ via pullback.  Since $N$ is obtained by infection
from $M$, there is a map $N \to M$ which is a $p$-tower map for any
prime $p$, by
Proposition~\ref{proposition:t-equivalence-of-infected-manifold}.
Pullback via $N\to M$ gives rise to an associated $2$-tower $N_2 \to
N_1 \to N_0=N$.  Also, $\phi_2\colon\pi_1(X_2) \to \Gamma_2$ gives
rise to $\psi_2\colon\pi_1(N_2) \to \Gamma_2$.

Now we compute $\lambda(N_2,\psi_2)$, using
Corollary~\ref{corollary:lambda-of-infected-manifold}.  Recall that
the infection is performed along the curve $\alpha$ in~$M$.  Let
$\tilde\alpha_1,\tilde\alpha_2,\ldots \subset M_2$ be the components
of the pre-image of~$\alpha$.  To apply
Corollary~\ref{corollary:lambda-of-infected-manifold} we need to
compute the degree $r_j$ of the covering $\tilde\alpha_i \to \alpha$,
and $\phi_2([\tilde\alpha_j]) \in \Gamma_2=\Z_4$.

We can read off the necessary data $r_j$ and $\phi_2([\alpha_j])$
algorithmically from the 1-complexes in
Figure~\ref{fig:covers-of-Q-sphere}.  As a general case, suppose that
$p\colon \tilde X \to X$ is a finite covering map and $\alpha$ is a
loop in $X$ based at $*\in X$.  For a path $\gamma$ in $X$ based at
$*$, we denote by $\tilde\gamma_v$ the lift of $\gamma$ in $\tilde X$
based at $v\in V=p^{-1}(*)$.  We construct a collection $S$ of loops
in $\tilde X$ as follows: initially let $S$ be the empty set and mark
all $v\in V$ as ``white''.  While there is a ``white'' $v$ in $V$, we
repeat the following: as a new element, insert into $S$ the loop
$\widetilde{(\alpha^{r})}_{v}$ where $r$ is the minimal positive
integer such that $\widetilde{(\alpha^r)}_{v}$ is a loop, and mark the
endpoints of $\widetilde{(\alpha^k)}_{v}$ as ``black'' for all $1\le
k\le r$.  We denote the result by the following notation.

\begin{definition}
  \label{definition:collection-of-lifts}
  We define $\mathcal{L}(\alpha,\tilde X|X)$ to be the collection $S$
  constructed above.
\end{definition}

We apply the above algorithm to our cover $X_2 \to X_0$.  As an abuse
of notation, let
$\alpha=(c_1,c_2)=c_1^{\vphantom{-1}}c_2^{\vphantom{-1}}c_1^{-1}c_2^{-1}$
be the loop in $X_0$ which represents the class
$[\alpha]\in\pi_1(M_0)=\pi_1(X_0)$ of the infection curve, and write
$\mathcal{L}(\alpha,X_2|X_0)=\{\tilde \alpha_j\}$.  For each $\tilde
\alpha_j$, $r_j$ is the integer such that $\tilde \alpha_j$ is a lift
of~$\alpha^{r_j}$.

As examples, in Figure~\ref{fig:lifts-of-Q-sphere} we illustrate four
elements of the collection $\mathcal{L}(\alpha,X_2|X_0)$, say
$\tilde\alpha_1,\ldots,\tilde\alpha_4$, as curves in (the 1-complex
with the homotopy type of)~$X_2$.  The black dot(s) on each
$\tilde\alpha_j$ represents the basepoint(s) giving~$\tilde\alpha_j$.
Note that $r_1=2$ and $r_2=r_3=r_4=1$.

\begin{figure}[ht]
  \begin{center}
    \includegraphics[scale=.68]{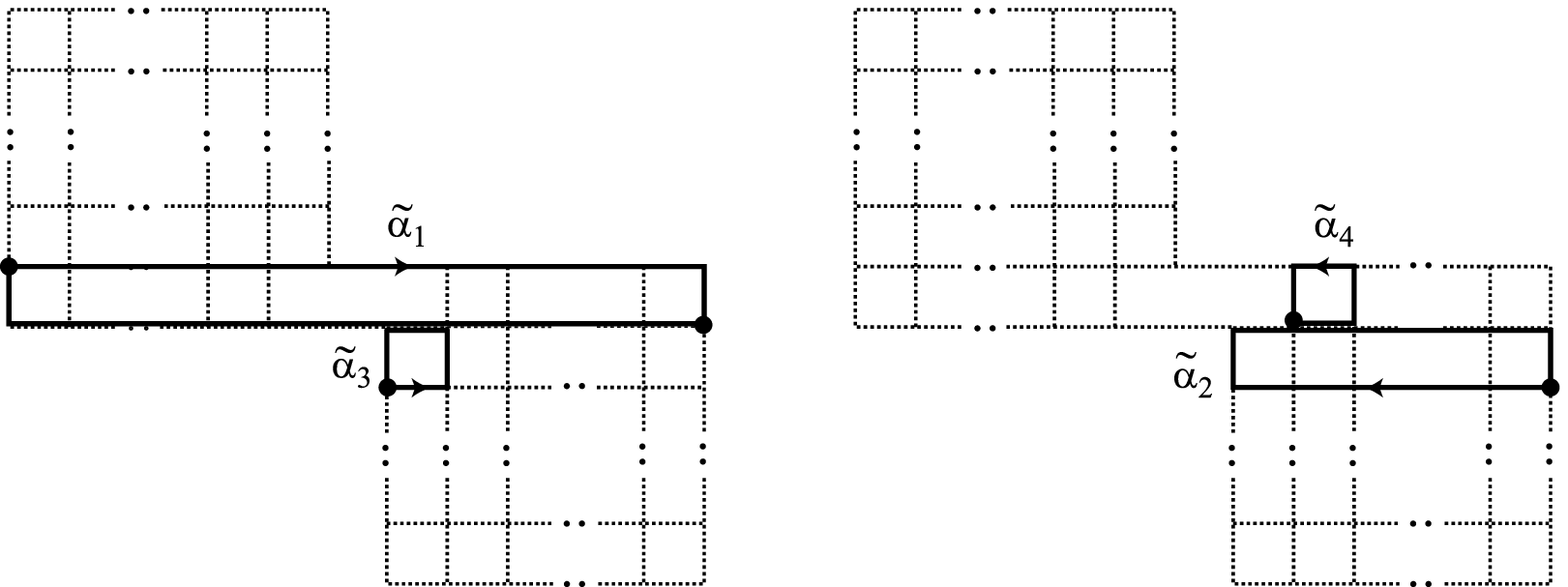}
    \caption{}\label{fig:lifts-of-Q-sphere}
  \end{center}
\end{figure}

We observe the following: among the loops in
$\mathcal{L}(\alpha,X_2|X_0)$, only
$\tilde\alpha_1,\ldots,\tilde\alpha_4$ pass through the 1-cell~$e_2$.
(In fact, for each 1-cell, exactly four loops in
$\mathcal{L}(\alpha,X_2|X_0)$ pass through it.)  From our definition
of $\phi_2$, it follows immediately that
$\tilde\alpha_1,\ldots,\tilde\alpha_4$ are sent to $3$, $1$, $1$, $3
\in \Z_4$ by $\phi_2$, respectively, and all other loops in
$\mathcal{L}(\alpha,X_2|X_0)$ are in the kernel of~$\phi_2$.

Now we can compute the intersection form defect invariant from the
data obtained above.  Let $A$ be a Seifert matrix of $K$.  Then, by
Corollary~\ref{corollary:lambda-of-infected-manifold}, we have
\[
\begin{split}
  \lambda(N_2,\psi_2) &=[\lambda_2(A,-\sqrt{-1})]-[\lambda_2(A,1)]\\
  &\mathrel{\phantom{=}} {} +
  [\lambda_1(A,-\sqrt{-1})]-[\lambda_1(A,1)] \\
  &\mathrel{\phantom{=}} {} + 2\big([\lambda_1(A,\sqrt{-1})] -
  [\lambda_1(A,1)]\big)  \in L^0(\Q(\sqrt{-1})).
\end{split}
\]
From the discriminant computation in
Lemma~\ref{lemma:discriminant-of-amphichiral-infection-knots}, it
follows that
\[
\dis\lambda(N_2,\psi_2) = (2a^2+1)(2a^4+4a^2+1).
\]
This proves
Proposition~\ref{proposition:discriminant-of-Q-sphere}~(2).

\subsection{Realization of independent discriminants}

Recall that our $\Sigma_i$ is $M$ infected by the knot~$K_{a_i}$.  By
the invariance of the intersection form defects under homology
cobordism (Theorem~\ref{theorem:invariants-from-covering-tower}) and
by the computation done in
Proposition~\ref{proposition:discriminant-of-Q-sphere}, if the $a_i$
are chosen in such a way that
\[
(2a_i^2+1)(2a_i^4+4a_i^2+1) \ne
(2a_j^2+1)^{n_1}(2a_j^4+4a_j^2+1)^{n_2} \text{ in
}\frac{\Q^\times}{\{z\bar z\mid z\in \Q(\sqrt{-1})\}}
\]
for any $i\ne j$ and any $n_1,n_2$, then $\Sigma_i$ is not
($\Z_2$-)homology cobordant to $\Sigma_j$ for any $i\ne j$.  This
subsection is devoted to the proof of the existence of such a
sequence~$\{a_i\}$.

In order to distinguish the elements in $\Q^\times$ modulo norms, we
appeal to some number theory.  Recall from
Section~\ref{section:basic-properties} that for each prime $p$ the
norm residue symbol induces a homomorphism
\[
(-,D)_p\colon \frac{\Q^\times}{\{z\bar z\mid z\in \Q(\sqrt{-1})\}} \to
\{\pm 1\}
\]
where $D=-4$.  Since $-4=2^2(-1)$, we may assume that $D=-1$.  In the
following proposition, we will construct a sequence $\{a_i\}$
together with a sequence of primes $\{p_i\}$ which are ``dual'' to the
values of the invariants of the $\Sigma_i$ in the following sense:
\[
\begin{aligned}
  \big((2a_i^2+1)(2a_i^4+4a_i^2+1),-1\big)_{p_i} &= -1 && \text{for
    any $i$, and}\\
  \big((2a_j^2+1)^{n_1}(2a_j^4+4a_j^2+1)^{n_2},-1\big)_{p_i} &= 1 &&
  \text{for any $i\ne j$ and $n_1,n_2$}.
\end{aligned}
\]
The existence of such $p_i$ completes the proof that the manifolds
$\Sigma_i$ obtained from the $a_i$ (including $\Sigma_0$) are not
($\Z_2$-)homology cobordant to each other.

\begin{proposition}
  \label{proposition:dual-primes}
  There are sequences $a_1,a_2,\ldots$ of positive integers and
  $p_1,p_2,\ldots$ of primes such that
  \begin{enumerate}
  \item $(2a_i^2+1,-1)_{p_i}=-1$ and $(2a_i^4+4a_i^2+1,-1)_{p_i}=1$.
  \item $(2a_j^2+1,-1)_{p_i}=1$ and $(2a_j^4+4a_j^2+1,-1)_{p_i}=1$ for
    $i\ne j$.
  \end{enumerate}
\end{proposition}

\begin{proof}
  Note that for an odd prime $p$, by
  Lemma~\ref{lemma:computation-norm-residue-symbol},
  \begin{align*}
    (x,-1)_p & = \Big((-1)^{v_p(x)v_p(-1)}
    \frac{(-1)^{v_p(x)}}{x^{v_p(-1)}} \Big)^{\frac{p-1}{2}} \\
    &= (-1)^{v_p(x)\cdot \frac{p-1}{2}} \\
    &=
    \begin{cases}
      -1 &\text{if } v_p(x)\text{ is odd and } p\equiv -1 \bmod 4\\
      1 &\text{otherwise}.
    \end{cases}
    \tag{$*$}
  \end{align*}

  We will inductively choose $a_n$ and a prime factor $p_n$ of
  $2a_n^2+1$ in such a way that (1) and (2) are satisfied whenever
  $i,j\le n$.  Let $a_1=1$ and $p_1=3$.  From the formula $(*)$ it
  follows that $a_1$ and $p_1$ satisfy~(1).

  Suppose $a_1,\ldots,a_{n-1}$ and $p_1,\ldots,p_{n-1}$ have been
  chosen.  To choose $a_n$, we need the following:

  \begin{lemma}
    \label{lemma:non-square-lemma}
    For any odd integer $q$, there is an odd multiple $a$ of $q$ such
    that $2a^2+1$ is not a square.
  \end{lemma}

  We postpone the proof of Lemma~\ref{lemma:non-square-lemma} and
  continue the proof of Proposition~\ref{proposition:dual-primes}.  By
  Lemma~\ref{lemma:non-square-lemma}, there is a positive odd multiple
  $a_n$~of
  \[
  q=\prod_{j<n} (2a_j^2+1)(2a_j^4+4a_j^2+1)
  \]
  such that $2a_n^2+1$ is not a square; then there is a prime factor
  $p_n$ of $2a_n^2+1$ such that $v_{p_n}(2a_n^2+1)$ is odd.  We may
  assume that $p_n \equiv -1 \bmod 4$; for, if $p\equiv 1\bmod 4$ for
  all prime $p$ such that $v_{p}(2a_n^2+1)$ is odd, then
  $2a_n^2+1\equiv 1 \bmod 4$, but it is impossible since $a_n$ is odd.

  From the formula $(*)$, it follows that $(2a_n^2+1,-1)_{p_n}=-1$.
  Also, $v_{p_n}(2a_n^4+4a_n^2+1)=0$ since $2a_n^2+1$ and
  $2a_n^4+4a_n^2+1$ are relatively prime.  So by the formula $(*)$ we
  have $(2a_n^4+4a_n^2+1,-1)_{p_n}=1$.

  Suppose $j<i$.  Then $p_j$ divides neither $2a_i^2+1$ nor
  $2a_i^4+4a_i^2+1$ since $p_j$ divides~$a_i$.  From the formula
  $(*)$, it follows that
  \[
  (2a_i^2+1,-1)_{p_j}=1=(2a_i^4+4a_i^2+1,-1)_{p_j}.
  \]
  If $p_i$ divides $2a_j^2+1$, then $p_i$ divides $a_i$ by our choice
  of~$a_i$.  However, since $p_i$ is a prime factor of $2a_i^2+1$,
  this implies $p_i=1$, which is a contradiction.  So
  $v_{p_i}(2a_j^2+1)=0$.  Similarly $v_{p_i}(2a_j^4+4a_j^2+1)=0$.  By
  the formula $(*)$, it follows that
  \[
  (2a_j^2+1,-1)_{p_i}=1=(2a_j^4+4a_j^2+1,-1)_{p_i}. \qedhere
  \]
\end{proof}

\begin{proof}[Proof of Lemma~\ref{lemma:non-square-lemma}]
  We may assume that $q$ is positive.  Consider a Diophantine equation
  $x^2=2y^2+1$.  From the theory of Pell's equation and continued
  fractions (for example, refer to~\cite{Hua:1982-1}), it follows that
  all positive solutions $(x_n,y_n)$ are given by the following
  recurrence relation: $x_1=3, y_1=2$, and
  \[
  \begin{split}
    x_{n+1} & = 3x_n+4y_n \\
    y_{n+1} &= 2x_n+3y_n
  \end{split}
  \]
  Choose any positive odd integer~$k$.  If $kq\ne y_n$ for all $n$,
  then for $a=kq$, $2a^2+1$ is not a square.  Suppose $kq=y_n$ for
  some~$n$.  Then $x_n>y_n>q$.  So $y_{n+1} > 5y_n$, and $(k+2)q =
  y_n+2q < 3y_n < y_{n+1}$.  Since $\{y_n\}$ is strictly increasing,
  it follows that $(k+2)q \ne y_i$ for all $i$, that is, for
  $a=(k+2)q$, $2a^2+1$ is not a square.
\end{proof}  

\subsection{Vanishing of signature invariants} 

In this subsection we prove that the multisignatures,
$\eta$-invariants, and Harvey's $L^2$-invariants vanish for
our~$\Sigma_i$.



Recall that the $\eta$-invariant $\tilde\eta(M,\phi)$ is defined for a
closed $3$-manifold $M$ endowed with a finite dimensional unitary
representation $\phi\colon \pi_1(M) \to U(k)$ as a signature defect,
as in~\cite{Atiyah-Patodi-Singer:1975-1,Atiyah-Patodi-Singer:1975-2}.
We need the following formula:

\begin{lemma}
  \label{lemma:additivity-of-eta-inv-under-infection}
  Suppose $M$ is a closed 3-manifold and $\phi\colon \pi_1(M) \to
  U(k)$ is a unitary representation.  Let $K$ be a knot with
  zero-surgery manifold $M_K$, and $N$ be the manifold obtained from
  $M$ by infection by $K$ along a simple closed curve $\alpha\subset
  M$.  Let $\psi$ be the composition 
  \[
  \psi\colon \pi_1(N) \to \pi_1(M) \xrightarrow{\phi} U(k),
  \]
  where the first map is induced by the $p$-tower map $N\to M$ given
  by Proposition~\ref{proposition:t-equivalence-of-infected-manifold}.
  Let $\phi_K$ be the composition
  \[
  \pi_1(M_K)\to H_1(M_K)=\Z \to U(k),
  \]
  where the last map sends a (positive) meridian of $K$ to
  $\phi([\alpha]) \in U(k)$.  Then we have
  \[
  \tilde\eta(N,\psi)=\tilde\eta(M,\phi)+\tilde\eta(M_K,\phi_K).
  \]
\end{lemma}

Since Lemma~\ref{lemma:additivity-of-eta-inv-under-infection} can be
proved using a standard folklore argument using the
Atiyah-Patodi-Singer index theorem, we just give a sketch of a proof:
by Lemma~\ref{lemma:null-bordism-over-Z}, there is a 4-manifold $W_K$
bounded by $M_K$ over~$\Z$.  View $M_K$ as the exterior of $K$ filled
in with a solid torus~$T$.  Attaching $W_K$ and $M\times[0,1]$ along
$T$ and a regular neighborhood of $\alpha$ in $M=M\times 1$, we obtain a
4-manifold $V$ with $\partial V=(-M) \cup N$ over~$U(k)$.  Using the
fact that $T \hookrightarrow W_K$ has a left homotopy inverse, one can
show that the signature defect of $V$ is the sum of those of $M\times
[0,1]$ and $W_K$, which are equal to zero and
$\tilde\eta(M_K,\phi_K)$, respectively.  By the index theorem, this is
exactly $\tilde\eta(N,\psi)-\tilde\eta(M,\phi)$.

In this subsection, $\widehat G$ denotes the algebraic closure of a
group $G$ with respect to $\Z$-coefficients in the sense
of~\cite{Cha:2004-1}, or Levine's algebraic
closure~\cite{Levine:1989-1}.  We denote by $p_G\colon G \to \widehat
G$ the natural map into the algebraic closure.
Following~\cite{Levine:1994-1}, we consider invariants of $M$ of the
form $\tilde\eta(M,\theta\circ p_{\pi_1(M)})$, where $\theta$ is a
representation of $\widehat{\pi_1(M)}$.

\begin{lemma}
  \label{lemma:vanishing-of-multisignature}
  Suppose $M$ is a 3-manifold such that $\tilde\eta(M,\theta\circ
  p_{\pi_1(M)})=0$ for any $\theta\colon \widehat{\pi_1(M)} \to U(k)$.
  Then all multisignatures of $M$ vanish.
\end{lemma}

\begin{proof}
  Since the multisignatures associated to $\Z_d$-valued characters are
  known to be equivalent to the $\eta$-invariants associated to
  representations $\rho$ of $\pi_1(M)$ factoring through $\Z_d$, it
  suffices to to check that every such representation $\rho$ factors
  through $\widehat{\pi_1(M)}$.  Since $\rho$ factors through
  $H_1(M)$, the proof is finished by applying the following property
  of the algebraic closure to the case of $G=\pi_1(M)$: for any group
  $G$, the map $p_G \colon G \to \widehat G$ induces an isomorphism on
  $H_1(-)$.
\end{proof}

Now suppose $M=L_1\#L_2$ is our seed manifold used in the previous
subsections, and $N$ is obtained from $M$ by infection along $\alpha$
by a knot~$K$.  Recall that $L_i$ is a lens space with vanishing
multisignatures.

\begin{lemma}
  \label{lemma:vanishing-of-eta-invariants}
  If the Alexander polynomial of $K$ has no zero in the unit circle,
  then for any $\theta\colon \widehat{\pi_1(N)} \to U(k)$,
  $\tilde\eta(N,\theta\circ p_{\pi_1(N)})$ vanishes.
\end{lemma}

\begin{proof}
  By (the proof of)
  Proposition~\ref{proposition:t-equivalence-of-infected-manifold},
  the $p$-tower map $N\to M$ is 2-connected.  It follows that the
  induced map $\widehat{\pi_1(N)} \to \widehat{\pi_1(M)}$ is an
  isomorphism, by results of \cite{Levine:1989-1,Cha:2004-1}.  (To
  obtain this when $\widehat G$ designates Levine's algebraic closure,
  we need an additional condition that $\pi_1(N)\to \pi_1(M)$ is
  normally surjective, which can be verified easily by a van Kampen
  argument.)  So, for any $\theta\colon \widehat{\pi_1(N)} \to U(k)$,
  there is $\theta'\colon \widehat{\pi_1(M)} \to U(k)$ making the
  following diagram commute:
  \[
  \begin{diagram}
    \node{\pi_1(N)} \arrow{s}\arrow{e,t}{p_{\pi_1(N)}}
    \node{\widehat{\pi_1(N)}} \arrow{s}\arrow{e,l}{\theta}
    \node{U(k)}
    \\
    \node{\pi_1(M)} \arrow{e,r}{p_{\pi_1(M)}}
    \node{\widehat{\pi_1(M)}} \arrow{ne,r}{\theta'}
  \end{diagram}
  \]
  So, $\psi=\theta\circ p_{\pi_1(N)}$ is induced by $\phi=\theta'\circ
  p_{\pi_1(M)}$ as in
  Lemma~\ref{lemma:additivity-of-eta-inv-under-infection}.  It follows
  that $\tilde\eta(N,\psi)$ is the sum of $\tilde\eta(M,\phi)$ and
  $\tilde\eta(M_K,\phi_K)$ for some~$\phi_K$.  We will show that both
  $\tilde\eta(M,\phi)$ and $\tilde\eta(M_K,\phi_K)$ are zero.  Since
  $M=L_1\#L_2$,
  \[
  \tilde\eta(M,\phi) = \tilde\eta(L_1,\phi_1) + \tilde\eta(L_2,\phi_2)
  \]
  for some~$\phi_i$.  Since $\pi_1(L_i)$ is a (finite) cyclic group,
  $\phi_i$ is a sum of 1-dimensional representations.  So we may
  assume that $\phi_i$ is 1-dimensional, that is,
  $\tilde\eta(L_i,\phi_i)$ is a multisignature, which vanishes by our
  choice of~$L_i$.  Let $A$ be a Seifert matrix of~$K$.  It is known
  that $\tilde\eta(M_K,\phi_K)$ is determined by $\{\sign
  (1-\omega)A+(1-\bar\omega)A^T\}_{\omega\in S^1}$ (e.g., see
  \cite{Friedl:2003-4}) Since $\Delta_A(t)$ has no zero on $S^1$,
  $\sign (1-\omega)A+(1-\bar\omega)A^T = 0$ whenever $\omega\in S^1$,
  and therefore, $\tilde\eta(M_K,\rho_K)$ vanishes.
\end{proof}

From Lemma~\ref{lemma:vanishing-of-eta-invariants} and
Lemma~\ref{lemma:vanishing-of-multisignature}, it follows that the
manifolds $\Sigma_i$ constructed in
Section~\ref{subsection:construction=of-exotic-examples} have
vanishing multisignatures and Levine's $\tilde\eta$-invariants, as
claimed in Theorem~\ref{theorem:exotic-rational-spheres-intro} (2)
and~(3).

\long\def\ignoreme{

\begin{proof}
  First we investigate the structure of~$\widehat{\pi_1(N)}$.  For
  this purpose, we need the following properties of the algebraic
  closure:
  \begin{enumerate}
  \item[(i)] If $G$ is abelian, then $\widehat G=G$ and $p_G$ is the
    identity map.
  \item[(ii)] Suppose $\pi \to G$ is 2-connected (and in addition,
    suppose $\pi\to G$ is normally surjective when $\widehat G$ is
    Levine's algebraic closure~\cite{Levine:1989-1}).  Then the
    induced map $\widehat\pi \to \widehat G$ is an isomorphism.
  \item[(iii)] The functor $G \to \widehat G$ commutes with direct
    limit.  In particular, $\widehat{G_1 * G_2}=\widehat{G_1} *
    \widehat{G_2}$, and similarly for amalgamated products.
  \end{enumerate}
  From (i) and (iii), it follows that $\widehat{\pi_1(M)}=\Z_{r_1} *
  \Z_{r_2}$.  By Van Kampen, we can write $\pi_1(N)$ as an
  amalgamated product:
  \[
  \pi_1(N) = \pi_1(M-\alpha) \mathbin{\mathop{*}\limits_{\Z\oplus\Z}}
  \pi_1(S^3-K)
  \]
  Since $\pi_1(S^3-K) \to H_1(S^3-K)=\Z$ is 2-connected, the algebraic
  closure of $\pi_1(S^3-K)$ is $\Z$ by (i) and~(ii).  So, applying the
  algebraic closure functor to the above displayed equation, we have
  \[
  \widehat{\pi_1(N)} = \widehat{\pi_1(M-\alpha)}
  \mathbin{\mathop{*}\limits_{\Z\oplus\Z}} \Z
  \]
  by (i) and (iii).  Note that this says that $\widehat{\pi_1(N)}$ is
  independent of the choice of~$K$.  In particular, viewing $M$ as $M$
  infected by the unknot along $\alpha$, it follows that
  \[
  \widehat{\pi_1(N)} = \widehat{\pi_1(M)} = \Z_{r_1} * \Z_{r_2}.
  \]
  Therefore, each representation $\theta\colon \widehat{\pi_1(N)} \to
  U(k)$ is of the form $\rho_1 * \rho_2$, where $\rho_i \colon
  \Z_{r_i} \to U(k)$ for $i=1,2$.  Denote $\rho=\theta\circ
  p_{\pi_1(N)}$.

  To compute $\tilde\eta(N,\rho)$, we choose 4-manifolds $W_1$, $W_2$
  such that $\partial W_i = rL_i$ over $\Z_{r_i}$ for some common
  multiple $r$ of $r_1, r_2$.  Also, we choose $W_{K}$ such that
  $\partial W_{K}=M_{K}$ over $\Z$, where the $\Z$-structure is given
  by $\pi_1(M_{K}) \to H_1(M_{K}) =\Z$.  Such $W_1$, $W_2$, and
  $W_{K}$ exist since the $\Omega_3(B\Z_{r_i})=\Z_{r_i}$ and
  $\Omega_3(B\Z)=0$.  Choose embeddings of a 3-ball $B^3$ into $L_1$
  and $L_2$ so that $\Sigma_0=(L_1-\inte B^3)\cup_{\partial B^3}
  (L_2-\inte B_3)$, and choose embeddings of $S^1\times D^2$ into $M$
  and $M_{K}$ from which $\Sigma$ is obtained as a toral sum.  Let
  \[
  W = \Big(W_1 \mathbin{\mathop{\cup}\limits_{rB^3}} W_2\Big)
  \mathbin{\mathop{\cup}\limits_{r(S^1\times D^2)}} rW_{K}.
  \]
  Then $\partial W=N$ over $\Z_{r_1} * \Z_{r_2}$.  By a
  Mayer-Vieotoris argument, for any homology coefficient $R$ given as
  $\Z[\Z_{r_1} * \Z_{r_2}] \to R$, we have an exact sequence
  \begin{multline*}
    0 \to H_2(W_1;R)\oplus H_2(W_2;R) \oplus H_2(W_{K};R)^r \to H_2(W;R)\\
    \to H_1(S^1\times D^2;R)^r \to H_1(W_1;R)\oplus H_1(W_2;R) \oplus
    H_1(W_{K};R)^r.
  \end{multline*}
  We claim that $H_1(S^1 \times D^2;R) \to H_1(W_{K};R)$ is injective;
  for, since $W_{K}=\partial M_{K}$ over $\Z$, $S^1\times D^2 \to
  W_{K}$ has a left homotopy inverse through which $W_{K} \to
  B(\Z_{r_1}*\Z_{r_2})$ factors, and so the induced map on $H_1(-;R)$
  is injective.  Therefore we have
  \[
  H_2(W;R) \cong H_2(W_1;R)\oplus H_2(W_2;R) \oplus H_2(W_{K};R)^r.
  \]
  From this it follows that
  \[
  \tilde\eta(N,\rho) = \tilde\eta(L_1,\rho_1) + \tilde\eta(L_2,\rho_2)
  + \tilde\eta(M_{K},\rho_{K})
  \]
  where $\rho_{K}$ is a representation of $\pi_1(M_K)$ factoring
  through $H_1(M_K)=\Z$.

  $\rho_1$, $\rho_2$, and $\rho_K$ are direct sums of 1-dimensional
  representations, since they factor through cyclic groups.  Thus we
  may assume that $\rho_1$, $\rho_2$, $\rho_K$ are 1-dimensional.  Now
  $\tilde\eta(L_i,\rho_i)$ is no more than the multisignature, which
  vanishes by our choice of~$L_i$, for $i=1,2$.  It is known that
  $\tilde\eta(M_K,\rho_K)$ is determined by the Levine-Tristram
  signature~$\sigma_K$ (e.g., see~\cite{Friedl:2003-4}).  By the
  hypothesis, $\sigma_K$ vanishes, and thus so does
  $\tilde\eta(M_K,\rho_K)$.  It follows that $\tilde\eta(N,\rho)=0$.
\end{proof}

} 

For each integer $n\ge 0$, Harvey defined the $L^2$-signature
invariants $\rho_n(N)$~\cite{Harvey:2006-1}.  It is the von
Neumann-Cheeger-Gromov $L^2$-signature defect associated to the
natural map $\pi_1(N) \to \pi_1(N)/\pi_1(N)^{(n)}_H$, where
$G_H^{(n)}$ denotes the $n$th term of the torsion-free derived
series~\cite{Cochran-Harvey:2004-1} for a group~$G$.

\begin{lemma}
  If $N$ is a rational homology sphere, then $\rho_n(N)=0$ for any
  $n\ge 0$.
\end{lemma}

\begin{proof}
  From the definition of the torsion-free derived series
  in~\cite{Cochran-Harvey:2004-1}, it follows that $\pi_1(N)^{(n)}_H =
  \pi_1(N)$ for all~$n$, since the first Betti number $b_1(N)$ is
  zero.  Being the signature defect for untwisted coefficients,
  $\rho_n(N)$ vanishes.
\end{proof}

From this it follows that the manifolds $\Sigma_i$ have vanishing
Harvey's invariants, as claimed in
Theorem~\ref{theorem:exotic-rational-spheres-intro}~(4).  This
completes the proof of
Theorem~\ref{theorem:exotic-rational-spheres-intro}.

\section{Intersection form defects of links}
\label{section:witt-class-defects-of-links}

Two links $L$ and $L'$ in $S^3$ are said to be (topologically)
\emph{concordant} if there is a locally flat $s$-cobordism $C$
embedded in $S^3\times [0,1]$ from $L\times\{0\} \subset S^3\times
\{0\}$ to $L'\times\{1\} \subset S^3\times \{1\}$.  $C$~is called a
\emph{concordance}.  (When $C$ is a smooth submanifold in
$S^3\times[0,1]$, $L$ and $L'$ are said to be \emph{smoothly
  concordant}.)  A link concordant to a trivial link is called a
\emph{slice link}.  Or equivalently, a slice link is a link which
bounds disjoint union of locally flat 2-disks in $D^4$, regarding
$S^3$ as the boundary of~$D^4$.

For a link $L$ in $S^3$, the closed 3-manifold obtained from $S^3$ by
performing surgery along the zero-linking framing of each component of
$L$ is called the \emph{zero-surgery manifold}.  The following fact is
well-known:

\begin{lemma}
  \label{lemma:homology-cobordism-from-concordance}
  If two links are concordant, their zero-surgery manifolds are
  homology cobordant.
\end{lemma}


By Lemma~\ref{lemma:homology-cobordism-from-concordance}, one can
apply Theorem~\ref{theorem:invariants-from-covering-tower} to extract
link concordance invariants from the intersection form defects of
($p$-towers of) zero-surgery manifolds.  In particular, for a slice
link, we obtain the following vanishing theorem.

\begin{theorem}
  \label{theorem:slice-obstruction-from-covering-tower}
  Suppose that $L$ is a slice link with zero-surgery manifold $M$,
  \[
  M_n \to \cdots \to M_1 \to M_0=M
  \]
  is a $p$-tower, and $\phi_n\colon \pi_1(M_n) \to \Z_d$ is a
  character with $d=p^a$.  Then
  \begin{enumerate}
  \item $(M_n,\phi_n)$ is trivial in $\Omega^{top}_3(B\Z_{d})$, and
  \item $\lambda(M_n,\phi_n)=0$ in $L^0(\Q(\zeta_{d}))$.
  \end{enumerate}
\end{theorem}

\begin{proof}
  By Lemma~\ref{lemma:homology-cobordism-from-concordance} and
  Theorem~\ref{theorem:invariants-from-covering-tower}, we may assume
  that $L$ is a trivial link.  Then, $M$ is the connected sum of $m$
  disjoint copies of $S^1\times S^2$, where $m$ is the number of
  components of~$L$.  Therefore, being a cover of $M$, $M_n$ is a
  connected sum of disjoint copies of $S^1\times S^2$.  Appealing to
  Lemma~\ref{lemma:vanishing-for-S^1xD^2}, the proof is completed.
\end{proof}

\subsection{$p$-towers of surgery manifolds of $\widehat F$-links}

In this subsection we will show that there are many highly nontrivial
$p$-towers of the zero-surgery manifold for a large class of links in
$S^3$.  We start with a description of the class of links we think of.

In this section, we denote by $\widehat G$ the algebraic closure of a
group $G$ with respect to $\Z_{(p)}$-coefficients, in the sense
of~\cite{Cha:2004-1}.  Suppose $L$ is an $m$-component link with zero
surgery manifold $M$, and let $X=\bigvee^m S^1$, the wedge of $m$
circles.  We say that a map of $X$ into $S^3-L$ or $M$ is a
\emph{meridian map} if the image of the $i$th circle is an $i$th
meridian of~$L$.  Let $F=\pi_1(X)$ be the free group of rank $m$, and
$\pi=\pi_1(S^3-L)$.  We say that $L$ is a \emph{$\Z_{(p)}$-coefficient
  $\widehat F$-link} if there is a meridian map $X \to S^3-L$ inducing
an isomorphism $\widehat F \to \widehat \pi$ and the preferred
longitudes of $L$ are in the kernel of $\pi \to \widehat \pi$.  We
note that this is a $\Z_{(p)}$-analogue of the notion of an $\widehat
F$-link due to Levine~\cite{Levine:1989-1}; the definition of an
$\widehat F$-link in~\cite{Levine:1989-1} is identical with ours
except that Levine's algebraic closure is used in place of
our~$\widehat G$.  Henceforth, an \emph{$\widehat F$-link} always
designates a $\Z_{(p)}$-coefficient $\widehat F$-link in our sense.
(An $\widehat F$-link in the sense of~\cite{Levine:1989-1} is a
$\Z_{(p)}$-coefficient $\widehat F$-link; it could be an interesting
question whether the converse is true, i.e., whether the two notions
are equivalent.)

By arguments of \cite{Levine:1989-1}, we have the following facts: the
property that $\widehat F \to \widehat \pi$ is an isomorphism is
independent of the choice of a meridian map, and a link concordant to
an $\widehat F$-link is an $\widehat F$-link.  We remark that it is a
long-standing conjecture that any link with vanishing Milnor
$\bar\mu$-invariants is an $\widehat F$-link (in the sense
of~\cite{Levine:1989-1}).

\begin{proposition}
  \label{proposition:t-equivalence-of-F-hat-links}
  Suppose $L$ is an $\widehat F$-link with zero-surgery manifold~$M$.
  Then any meridian map $X=\bigvee^m S^1 \to M$ is a $p$-tower map, in
  the sense of Definition~\ref{definition:t-equivalence}.
\end{proposition}

\begin{proof}
  Since $L$ is an $\widehat F$-link, a meridian map $X \to S^3-L$ is a
  $p$-tower map.  So it suffices to show the inclusion $S^3-L \to M$
  is a $p$-tower map.  To prove this, we need the following two
  properties of the algebraic closure functor $E(G)=\widehat G$:

  \begin{enumerate}
  \item The algebraic closure functor preserves direct limits, that
    is, $E(\varinjlim G_i)$ is the direct limit of the system
    $\{E(G_i)\}$ \emph{in the full subcategory of algebraically closed
      groups}.
  \item The algebraic closure functor is an idempotent, that is,
    $E(E(G))=E(G)$.
  \end{enumerate}
  
  (1) holds since $E\colon \{\text{groups}\} \to \{\text{algebraically
    closed groups}\}$ is a left adjoint of the inclusion functor
  $\{\text{algebraically closed groups}\} \to \{\text{groups}\}$.  For
  a proof of (2), refer to \cite{Levine:1989-1}.

  Let $\pi=\pi_1(S^3-L)$, $G=\pi_1(M)$, $F$ be a free group of rank
  $m$, and $\ell\colon F \to \pi$ be a map sending the $i$th generator
  of $F$ to an $i$th preferred longitude of~$L$.  (Here $m$ is the
  number of components of $L$ as before.)  Then $G=\Coker\{\ell\}$.
  By our hypothesis that $L$ is an $\widehat F$-link, the composition
  $F\xrightarrow{\ell} \pi \to \widehat\pi$ is the zero map.  So it
  induces the zero map $\widehat F \to \widehat{\widehat\pi}$.  By (2)
  above, it follows that $\widehat\ell\colon \widehat F\to
  \widehat\pi$ is the zero map.  Therefore $\widehat\pi\cong
  \Coker\{\widehat\ell\}\cong \widehat G$ by (1) above (recall that
  the cokernel is a direct limit).  Now, by
  Proposition~\ref{proposition:algebraic-closure-and-t-equivalences},
  the inclusion $S^3-L \to M$ is a $p$-tower map.
\end{proof}

\begin{remark}
  An interesting question related to
  Proposition~\ref{proposition:t-equivalence-of-F-hat-links} is the
  following: if $L$ is an $m$-component link with vanishing
  $\bar\mu$-invariants, then is a meridian map of $\bigvee^m S^1$ into
  the surgery manifold of $L$ a $p$-tower map?  An affirmative answer
  may be viewed as an evidence supporting the conjecture that a link
  with vanishing $\bar\mu$-invariant is an $\widehat F$-link.
\end{remark}

As stated in the corollary below, it follows that the character groups
of iterated $p$-covers of the surgery manifold of $L$ are highly
nontrivial, provided that $L$ is an $\widehat F$-link which is not a
knot.

\begin{corollary}
  If $M$ is the zero-surgery manifold of an $m$-component $\widehat
  F$-link $L$ and
  \[
  M_n \to \cdots\to M_1 \to M_0=M
  \]
  is a $p$-tower consisting of connected covers $M_{i+1} \to M_i$ with
  deck transformation groups $\Gamma_i$, then for any abelian group
  $\Gamma_n$, we have
  \label{corollary:size-of-character-group}
  $\Hom(\pi_1(M_n),\Gamma_n)\cong (\Gamma_n)^{r_n}$, where the rank
  $r_n$ is given by
  \[
  r_n=\Big(\prod_{i=0}^{n-1} |\Gamma_i|\Big) (m-1)+1.
  \]
\end{corollary}

\begin{proof} 
  By Proposition~\ref{proposition:t-equivalence-of-F-hat-links}, we
  may assume that $M=\bigvee^m S^1$.  Since $M$ is a 1-complex,
  $H_1(M_n)$ is a free abelian group.  Therefore it suffices to show
  that the first Betti number $b_1(M_n)$ is the number $r_n$ given
  above.  Let $d=|\Gamma_0|\cdots |\Gamma_{n-1}|$.  Since $M_n$ is a
  $d$-fold cover of $M$, the Euler characteristic
  $\chi(M_n)=1-b_1(M_n)$ can also be computed as follows:
  \[
  \chi(M_n) = d\cdot \chi(M) = d(1-m)
  \]
  From this it follows that $b_1(M_n)=r_n$.
\end{proof}

\section{Computation for iterated Bing doubles}
\label{section:iterated-bing-double}

For a link $L$ in $S^3$, let $BD(L)$ be the (untwisted) Bing double
of~$L$.  $BD(L)$ is obtained from $L$ as follows: let $V$ be an
unknotted solid torus in $S^3$, and $L_{orbit}$ be the 2-component
link in $V$ illustrated in Figure~\ref{figure:bing-double-orbit}.  For
a link $L$ with $m$ components, let $h_i$ be a homeomorphism of $V$
onto a tubular neighborhood of the $i$th component of $L$ which sends
a preferred longitude and meridian of $V$ to those of the $i$th
component of $L$ respectively ($1\le i\le m$).  Then $BD(L)=
\bigcup_{i=1}^m h_i(L_{orbit})\subset S^3$.

\begin{figure}[ht]
  \begin{center}
    \includegraphics[scale=.9]{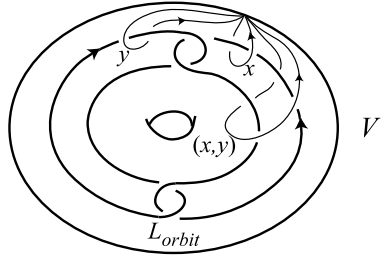}
  \end{center}
  \caption{The link $L_{orbit}$ in an unknotted solid torus $V\subset
    S^3$.}
  \label{figure:bing-double-orbit}
\end{figure}

We define the \emph{$n$th iterated Bing double} $BD_n(L)$ of $L$
inductively by $BD_0(L)=L$ and $BD_n(L)=BD_{n-1}(BD(L))$ for $n>0$.

For a knot $K$ in $S^3$, $BD_n(K)$ is obtained from a trivial link by
infection as follows: let $\alpha$ be a meridional curve of a tubular
neighborhood $U$ of a trivial knot $O$ in $S^3$, and take an iterated
Bing double $BD_n=BD_n(O)$ contained in $U$.  Then $BD_n$ is a trivial
link in $S^3$, and by performing infection on $BD_n$ by $K$ along
$\alpha$, we obtain $BD_n(K)$.
Figure~\ref{figure:bd_n-infection-curve} illustrates $BD_n$ and
$\alpha$ for $n=2$.

\begin{figure}[ht]
  \begin{center}
    \includegraphics[scale=.9]{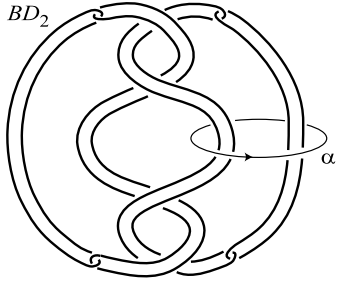}
  \end{center}
  \caption{The infection curve $\alpha$ for $n=2$.}
  \label{figure:bd_n-infection-curve}
\end{figure}


We denote the mirror image of a link $L$ with reversed orientation
by~$-L$.  Recall from the introduction that a link $L$ is said to be
\emph{2-torsion} if a connected sum of $L$ and itself is slice.  Note
that a connected sum of two links is defined by choosing a ``disk
basing'' for each link in the sense of Lin and
Habegger~\cite{Habegger-Lin:1990-1}.  The defining condition of a
2-torsion link $L$ should be understood as that $L\# L$ is slice for
some choice of disk basings.  (We remark that $L\#L$ may not be slice
for some other disk basings even when $L$ is 2-torsion.)  If $L$ is
2-torsion, then $L$ is concordant to~$-L$, and when $L$ is a knot, the
converse is also true.  An amphichiral knot is 2-torsion.

Observe that $BD_n\cup \alpha$ is isotopic to $-(BD_n \cup \alpha)$;
applying a $\pi$-rotation to Figure~\ref{figure:bd_n-infection-curve}
about an horizontal axis, we obtain one from another.  From this, it
follows that $BD_n(K)$ is isotopic to $-(BD(-K))$.  As a consequence,
we have

\begin{lemma}
  \label{lemma:torsion-property-of-bd_n}
  If $K$ is amphichiral, then $BD_n(K)$ is isotopic to $-BD_n(K)$.  If
  $K$ is 2-torsion, then $BD_n(K)$ is 2-torsion.
\end{lemma}

In this section we prove that certain iterated Bing doubles are
2-torsion but not slice, as stated below:

\begin{theorem}
  \label{theorem:non-slice-torsion-bd_n}
  Let $\{a_i\}$ be the sequence of integers given by
  Proposition~\ref{proposition:dual-primes} and $K_{a}$ be the
  amphichiral knot shown in Figure~\ref{fig:infection-amph-knot}.
  Then $BD_n(K_{a_i})$ is not slice for any $i$ and~$n$.
\end{theorem}

Theorem~\ref{theorem:non-slice-bing-souble-intro} stated in the
introduction is an immediate consequence of
Theorem~\ref{theorem:non-slice-torsion-bd_n}.

\begin{remark}
  \label{remark:vanishing-of-signatures-for-bd_n}
  It is known that many known invariants vanish for (iterated) Bing
  doubles.  (Refer to Cimasoni's paper \cite{Cimasoni:2006-1} for an
  excellent discussion on this.)  Harvey's invariant $\rho_k$
  \cite{Harvey:2006-1} and Levine's
  $\tilde\eta$-invariants~\cite{Levine:1994-1} detect some examples of
  non-slice Bing doubles, however, both invariant vanish for our
  $BD_n(K_{a_i})$, as explained below:
  \begin{enumerate}
  \item The 2-torsion property of $BD_n(K_{a_i})$ implies the
    vanishing of~$\rho_k$; since $\rho_k(L)=-\rho_k(-L)$ for any link
    $L$~\cite{Harvey:2006-1},
    $\rho_k(BD_n(K_{a_i}))=-\rho_k(-BD_n(K_{a_i}))=-\rho_k(BD_n(K_{a_i}))$,
    and so $\rho_k(BD_n(K_{a_i}))=0$.  Or alternatively, one may
    appeal to a formula for $\rho_k$ of infected manifold: it is known
    that $\rho_k(BD_n(K))=\rho_k(BD_n)+\epsilon \rho_0(K)= \epsilon
    \rho_0(K)$ for some $\epsilon\in\{0,1\}$~\cite{Harvey:2006-1}, and
    $\rho_0(K)$ is equal to the integral of the Levine-Tristram
    signature function $\sigma_K$ over the unit
    circle~\cite{Cochran-Orr-Teichner:1999-1,Cochran-Orr-Teichner:2002-1}.
    So, if the integral of $\sigma_K$ vanishes (in particular if $K$
    is torsion in the knot concordance group), then $\rho_k(BD_n(K))$
    vanishes.
  \item Using a similar argument (see also the proof of
    Lemma~\ref{lemma:vanishing-of-eta-invariants}), it is shown that
    if $\sigma_K=0$, or equivalently if $K$ is torsion the algebraic
    knot concordance group, then Levine's invariant
    $\tilde\eta(M,\theta\circ p_{\pi_1(M)})$ of the zero-surgery
    manifold $M$ of $BD_n(K)$ vanishes for any unitary representation
    $\theta$ of the algebraic closure of~$\pi_1(M)$.
  \end{enumerate}
\end{remark}


\subsection{$p$-towers for iterated Bing doubles}

The proof of Theorem~\ref{theorem:non-slice-torsion-bd_n} proceeds
similarly to that of
Theorem~\ref{theorem:exotic-rational-spheres-intro} discussed in
Section~\ref{section:homology-cobordism-for-b1=0}.  The outline is as
follows: denote the zero surgery manifolds of $BD_n$ and $BD_n(K)$ by
$M$ and $N$, respectively, and let $X=\bigvee^{2^n} S^1$.  By
Proposition~\ref{proposition:t-equivalence-of-infected-manifold} there
is a $p$-tower map $N \to M$ since $N$ is obtained from $M$ by
infection along~$\alpha$.  Also, the meridian map $X \to M$ sending
the $i$th circle to the $i$th meridian of $BD_n$ is a $p$-tower map
by Proposition~\ref{proposition:t-equivalence-of-F-hat-links} (or more
directly, since it is a $\pi_1$-isomorphism).  We consider $p$-towers
of $M$ and $N$ which are induced by a $p$-tower of $X$ that will be
constructed combinatorially.  We show the nontriviality of the
intersection form defects extracted from the $p$-tower of $N$ by a
computation using
Corollary~\ref{corollary:lambda-of-infected-manifold} as in
Section~\ref{subsection:combinatorial-computation}.  This proves
Theorem~\ref{theorem:non-slice-torsion-bd_n}.

Our construction of a $p$-tower of~$X_0=X$ as follows.  $p=2$ will be
used throughout this section.  We define inductively covers $X_{k}\to
X_{k-1}$ for $1\le k \le n$, some 1-cells $c^{(k)}_i$ of $X_k$ for $1
\le i \le 2^{n-k}$, and a map $\phi_k\colon \pi_1(X_k) \to \Gamma_k =
(\Z_2)^{2^{n-k}}$.  Initially, viewing $X_0$ as a 1-complex with one
$0$-cell $*$ and $2^n$ 1-cells, let $c^{(0)}_i$ be the $i$th
(oriented) 1-cell of $X_0$ and define $\phi_0\colon \pi_1(X_0) \to
\Gamma_0=(\Z_2)^{2^n}$ by assigning to each 1-cell $c^{(0)}_i$ the
$i$th standard basis element $e_i\in \Gamma_0$.  Suppose $k < n$ and
$X_k$, $c^{(k)}_i$, and $\phi_k\colon\pi_1(X_k)\to \Gamma_k$ have been
defined.  We define $X_{k+1}$ to be the $\Gamma_k$-cover of $X_k$
determined by $\phi_k$, and choose a basepoint $*\in X_{k+1}$ from the
preimage of $*\in X_k$.  Let $c^{(k+1)}_i \subset X_{k+1}$ be the lift
of $c^{(k)}_{2i-1} \subset X_k$ based at $*\in X_{k+1}$.
$\phi_{k+1}\colon \pi_1(X_{k+1}) \to \Gamma_{k+1}=(\Z_2)^{2^{n-k-1}}$
is defined to be the map induced by the assignment $c^{(k+1)}_i \to
e_i$, (other cells) $\to 0$.  Finally, let $X_{n+1}$ be the double
cover of $X_n$ determined by $\phi_n\colon \pi_1(X_n) \to \Gamma_n =
\Z_2$.

As an abuse of notation, we denote by $\alpha$ a loop in $X_0$
representing the class of $[\alpha]\in\pi_1(M_0)=\pi_1(X_0)$ of the
infection curve~$\alpha\subset M_0$.  $\alpha$ can be described as an
iterated commutator on the loops $c^{(0)}_i$.  In order to give an
explicit commutator expression, we define inductively loops
$x^{(k)}_i$ in $X_0$ based at $*$ for $0 \le k\le n$ and $1\le i \le
2^{n-k}$ as follows: $x^{(0)}_i = c^{(0)}_i$ and
\[
x^{(k+1)}_i = (x^{(k)}_{2i-1}, x^{(k)}_{2i}) =
x^{(k)}_{2i-1}x^{(k)}_{2i} (x^{(k)}_{2i-1})^{-1}(x^{(k)}_{2i})^{-1}.
\]
Then, $\alpha=x^{(n)}_1$.  For, in
Figure~\ref{figure:bing-double-orbit}, the meridian of the solid torus
$V$ containing $L_{orbit}$ is the commutator of the meridians of the
two components of $L_{orbit}$, and applying this relation inductively,
it follows that $\alpha=x^{(n)}_1$.  Note that $[x^{(k)}_i] \in
\pi_1(X_0)^{(k)} \subset \pi_1(X_k) \subset \pi_1(X_0)$.  From this it
follows that any lift of $\alpha$ in $X_n$ is a loop.

In order to compute intersection form defects, we need to investigate
the collection $\mathcal{L}(\alpha,X_{n+1}|X_0)$ described in
Definition~\ref{definition:collection-of-lifts}, as we did in
Section~\ref{subsection:combinatorial-computation}.  Since any lift of
$\alpha$ in $X_n$ is a loop and since $X_{n+1}$ is a double cover of
$X_n$, some lifts of $\alpha$ in $X_{n+1}$ may not be a loop but any
lift of $\alpha^2$ in $X_{n+1}$ is a loop.  In other words, writing
$\mathcal{L}(\alpha,X_{n+1}|X_0)=\{\tilde\alpha_{j}\}$, each
$\tilde\alpha_{j}$ is a lift of $\alpha^{r_j}$ in $X_{n+1}$ for some
$r_j\in \{1,2\}$.

The essential property of our 2-tower is the following:

\begin{lemma}
  \label{lemma:lift-behaviour-of-bd-infection-curve}
  For any $d>0$ and $0 \le s <d$, there is a character
  $\varphi_{n+1}\colon \pi_1(X_{n+1}) \to \Z_d$ such that all the
  $\tilde\alpha_{j}\in \mathcal{L}(\alpha,X_{n+1}|X)$ are in the
  kernel of $\varphi_{n+1}$ except two, say $\tilde\alpha_{1}$ and
  $\tilde\alpha_{2}$, with the following properties: $r_1=2$, $r_2=1$,
  and $\tilde\alpha_{1}, \tilde\alpha_{2}$ are sent to $s, -s\in \Z_d$
  by $\varphi_{n+1}$, respectively.
\end{lemma}  

Before proving Lemma~\ref{lemma:lift-behaviour-of-bd-infection-curve},
we give a proof of Theorem~\ref{theorem:non-slice-torsion-bd_n} using
our invariants associated to the characters given by
Lemma~\ref{lemma:lift-behaviour-of-bd-infection-curve}.

\begin{proof}[Proof of Theorem~\ref{theorem:non-slice-torsion-bd_n}]
  As before, let $M$ and $N$ be the zero surgery manifolds of
  $BD_n$ and $BD_n(K_{a_i})$, and $X \to M$ be the meridian map.  Let
  \begin{gather*}
    M_{n+1} \to M_n \to \cdots \to M_0=M\\
    \intertext{and}
    N_{n+1} \to N_n \to \cdots \to N_0=N
  \end{gather*}
  be the $2$-towers which correspond, via the 2-tower maps $X\to M
  \leftarrow N$, to the 2-tower of $X$ constructed above.

  Let $\varphi_{n+1}\colon\pi_1(X_{n+1}) \to \Z_4$ be a map given by
  Lemma~\ref{lemma:lift-behaviour-of-bd-infection-curve} applied to
  the case of $d=4$ and $s=1$.  Let $\psi_{n+1}\colon \pi_1(N_{n+1}) \to
  \Z_4$ be the map induced by~$\varphi_{n+1}$.  By
  Lemma~\ref{lemma:lift-behaviour-of-bd-infection-curve}~(1) and
  Corollary~\ref{corollary:lambda-of-infected-manifold} we have
  \[
  \lambda(N_{n+1},\psi_{n+1}) = [\lambda_2(A,\sqrt{-1})] +
  [\lambda_1(A,-\sqrt{-1})] \in L^0(\Q(\sqrt{-1}).
  \]
  where $A$ is a Seifert matrix of~$K_{a_i}$.  Applying the
  discriminant map
  \[
  \dis\colon L^0(\Q(\sqrt{-1}) \to \frac{\Q^\times}{\{z\cdot \bar z
    \mid z\in \Q(\sqrt{-1})^\times\}}
  \]
  defined in Section~\ref{subsection:norm-residue-symbol}, we obtain
  \[
  \dis \lambda(N_{n+1},\psi_{n+1}) = (2a_i^2+1)(2a_i^4+4a_i^2+1)
  \]
  by Lemma~\ref{lemma:discriminant-of-amphichiral-infection-knots}.
  From the norm residue symbol computation in
  Proposition~\ref{proposition:dual-primes}, it follows that $\dis
  \lambda(N_{n+1},\psi_{n+1})$ is nontrivial.  Therefore
  $BD_n(K_{a_i})$ is not slice by
  Theorem~\ref{theorem:slice-obstruction-from-covering-tower}.
\end{proof}

\subsection{Lifts of the infection curve $\alpha$}

In this subsection we complete the proof of
Theorem~\ref{theorem:non-slice-torsion-bd_n} by showing
Lemma~\ref{lemma:lift-behaviour-of-bd-infection-curve}.  In order to
give a precise description of lifts of $x^{(k)}_i$ in $X_k$, we
consider the following cut-paste construction of our tower.  For $0\le
k\le n$, let $Y_k$ be $X_k$ with the 1-cells $c^{(k)}_i$ removed
($1\le i \le 2^{n-k}$).  For each $g\in \Gamma_k$, let $Y_k(g)$ be a
copy of~$Y_k$.  We obtain $X_{k+1}$ by taking the disjoint union
$\bigcup_{g\in\Gamma_k} Y_k(g)$ and then attaching $2^{2(n-k)}$
1-cells $c^{(k)}_i(g)$ which goes from (starting point of $c^{(k)}_i)
\in Y_k(g)$ to (endpoint of $c^{(k)}_i)\in Y_k(g+e_i)$ for
$g\in\Gamma_k$ and $1\le i\le 2^{n-k}$.  We regard $*\in Y_k \subset
X_k$ as the basepoint of $Y_k$, and $*\in Y_k = Y_k(0) \subset
X_{k+1}$ as the basepoint of~$X_{k+1}$, where $0\in \Gamma_k$ is the
(additive) identity.  Then $c^{(k+1)}_i$ is the 1-cell
$c^{(k)}_{2i-1}(0) \subset X_{k+1}$.
Figure~\ref{figure:cover-scheme-of-bd_n} is a schematic diagram of
$Y_k$ and (part of) $X_{k+1}$.

\begin{figure}[ht]
  \begin{center}
    \includegraphics[scale=.85]{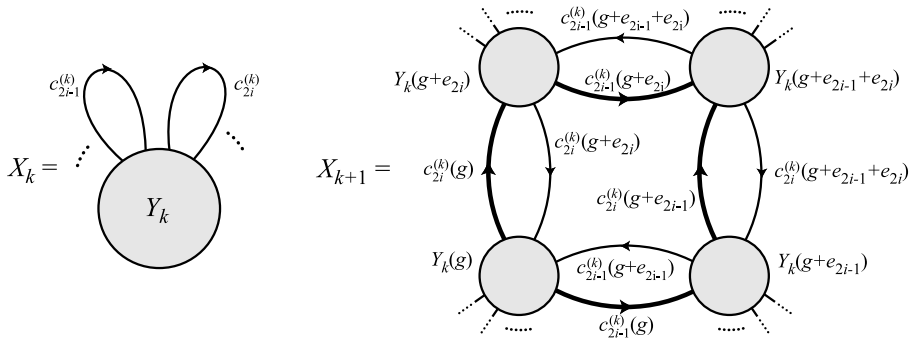}
  \end{center}
  \caption{A schematic diagram of $Y_k$ and $X_{k+1}$.}
  \label{figure:cover-scheme-of-bd_n}
\end{figure}

Let $\bar X_{k+1}$ be the 1-complex obtained by collapsing each
$Y_k(g)\subset X_{k+1}$ to a point for $g\in\Gamma_k$.  For a path
$\gamma$ in $X_{k+1}$, denote its composition with $X_{k+1} \to \bar
X_{k+1}$ by~$\bar\gamma$.  In particular $\bar c^{(k+1)}_i(g)$ is the
image of $c^{(k+1)}_i(g)\subset X_{k+1}$ in $\bar X_{k+1}$.  Note that
the map $\phi_{k+1}\colon\pi_1(X_{k+1}) \to \Gamma_{k+1}$ factors
through $\pi_1(\bar X_{k+1})$.

\begin{lemma}
  \label{lemma:computation-of-lifts}
  For a 0-cell $v$ in $X_{k+1}$, let $\gamma_v$ be the lift of the
  path $x^{(k+1)}_i$ in $X_{k+1}$ based at~$v$.  Note that $v$ lies in
  $Y_{k}(g)$ for some $g\in\Gamma_k$.  Then the following holds:
  \begin{enumerate}
  \item[$(1)$] For $j\ne i$, $\bar\gamma_v$ never meets (the interior
    of) $\bar c^{(k+1)}_j \subset \bar X_{k+1}$.
    $\bar\gamma_v$~passes through $\bar c^{(k+1)}_i$ algebraically
    $\begin{Bmatrix} +1 \\ -1 \\ 0 \end{Bmatrix}$ times
    $\begin{Bmatrix} \text{if }v=*\in Y_{k}(0)\subset X_{k+1}, \hfill\\
      \text{if }v=*\in Y_{k}(e_{2i})\subset X_{k+1}, \hfill \\
      \text{otherwise.} \hfill \end{Bmatrix}$
  \item[$(2)$] If $v\ne *\in Y_{k}(g)$, then $\bar\gamma_v$ is
    null-homotopic (rel $\partial$) in $\bar X_{k+1}$.  If $v=*\in
    Y_{k}(g)$, then $\bar\gamma_v$ is homotopic (rel $\partial$) to a
    4-gon
    \[
    \bar c^{(k)}_{2i-1}(g) \cdot \bar c^{(k)}_{2i}(g+e_{2i-1})
    \cdot \big(\bar c^{(k)}_{2i-1}(g+e_{2i}) \big)^{-1} \cdot \big( \bar
    c^{(k)}_{2i}(g) \big)^{-1}
    \]
    in $\bar X_{k+1}$.  (See Figure~\ref{figure:cover-scheme-of-bd_n},
    where the 4-gon is illustrated as bold edges.)
  \end{enumerate}
\end{lemma}

\begin{proof}
  Denote (1) and (2) stated above by $(1_{k+1})$ and $(2_{k+1})$.  We
  show the lemma by proving the following implications: $(1_{k})
  \Rightarrow (2_{k+1}) \Rightarrow (1_{k+1})$.  Note that for
  $k+1=0$, the initial condition $(1_0)$ holds obviously.
  \def\qedsymbol{}
\end{proof}
  
\begin{proof}[Proof of $(2_{k+1})$ $\Rightarrow$ $(1_{k+1})$]
  Indeed the first statement of $(1_{k+1})$ is proved without using
  $(2_{k+1})$.  For, observe that $c^{(k+1)}_j$ is a lift of the
  1-cell $c^{(0)}_\ell\subset X_0$ representing $x_\ell \in
  \pi_1(X_0)$, where $\ell=(j-1)\cdot 2^{k+1}+1$.  If $i\ne j$, then
  since $x^{(k+1)}_i$ is a word in $x_{(i-1)\cdot 2^{k+1}+1}, \ldots,
  x_{i\cdot 2^{k+1}}$, $x^{(k+1)}_i$ does not contain $x_\ell^{\pm1}$.
  Therefore any lift of $x^{(k+1)}_i$ in $X_{k+1}$ never passes
  through~$c^{(k+1)}_j$.  From this the first conclusion of
  $(1_{k+1})$ follows.
  
  To prove the second statement of $(1_{k+1})$, suppose $(2_{k+1})$
  holds and suppose $\bar\gamma_v$ passes through $\bar c^{(k+1)}_i =
  \bar c^{(k)}_{2i-1}(0)$ algebraically nonzero times.  Then by the
  first statement of $(2_{k+1})$, $v$ should be $*\in Y_{k}(g) \subset
  X_{k+1}$ for some $g\in\Gamma_{k-1}$.  By the second statement of
  $(2_{k+1})$, it follows that $g=0$ or $e_{2i}$ and in each case
  $\bar\gamma_v$ passes through $\bar c^{(k)}_{2i-1}(0)$ algebraically
  $+1$ and $-1$ times, respectively. \def\qedsymbol{}
\end{proof}

\begin{proof}[Proof of $(1_k) \Rightarrow (2_{k+1})$]
  Suppose $(1_k)$ holds.  Let $a_\ell$ be the lift of $x^{(k)}_{\ell}$
  in $X_k$ which is based at $v\in Y_k(g)=Y_k\subset X_k$.  Note that
  $a_{\ell}$ is a loop in $X_k$.  Since $x^{(k+1)}_{i}=x^{(k)}_{2i-1}
  x^{(k)}_{2i} (x^{(k)}_{2i-1})^{-1} (x^{(k)}_{2i})^{-1}$, $\gamma_v$
  is obtained by concatenating some lifts of $a_{2i-1}$, $a_{2i}$,
  $a_{2i-1}^{-1}$, and $a_{2i}^{-1}$ in~$X_{k+1}$.

  We claim that if $\bar a_\ell$ passes through $\bar c^{(k)}_{\ell}$
  algebraically 0 times in $\bar X_{k}$, then for any lift $a'_\ell$
  in $X_{k+1}$ of $a_\ell$, $\bar a'_\ell$ is a loop null-homotopic
  (rel $\partial$) in~$\bar X_{k+1}$.  For, observe the following
  facts: first, by the first statement of~$(1_k)$, $\bar a_\ell$ never
  meets $\bar c^{(k)}_j$ in $\bar X_k$ for $j\ne \ell$.  So, from the
  construction of $X_{k+1}$ from $X_k$, it follows that $\bar a'_\ell$
  is contained in the circle $\bar c^{(k)}_{\ell}(g) \cup \bar
  c^{(k)}_{\ell}(g+e_{\ell}) \subset \bar X_{k+1}$ (see Figure
  \ref{figure:cover-scheme-of-bd_n}).  Second, from the hypothesis of
  the claim it follows that $\bar a'_\ell$ is a loop, and furthermore,
  has degree zero as a map into the circle $\bar c^{(k)}_{\ell}(g)
  \cup \bar c^{(k)}_{\ell}(g+e_{\ell})$.  Thus $\bar a'_\ell$ is
  null-homotopic.

  From the claim, it follows that $\bar\gamma_v$ is not null-homotopic
  in $\bar X_{k+1}$ only if $\bar a_\ell$ passes through $\bar
  c^{(k)}_{\ell}$ algebraically nonzero times for \emph{both}
  $\ell=2i-1$ and~$2i$.  This is equivalent to the condition that
  $v=*\in Y_k(g)$ by the second statement of~$(1_k)$.  In this case,
  looking at the lifts of $a_{2i-1}$, $a_{2i}$, $a_{2i-1}^{-1}$, and
  $a_{2i}^{-1}$ in $X_{k+1}$, it is easily verified that
  $\bar\gamma_v$ is of the desired form.
\end{proof}

\begin{proof}[Proof of Lemma~\ref{lemma:lift-behaviour-of-bd-infection-curve}]
  From the construction discussed at the beginning of this subsection,
  one can see that the 1-complex $X_n$ is as in
  Figure~\ref{figure:nth-cover-of-bd_n}.  Recall that $X_{n+1}$ is the
  double cover of $X_n$ determined by the assignment $c^{(n)} \to 1$,
  (other 1-cells) $\to 0$.  So, taking two disjoint copies of the
  1-complex $R$ obtained from $X_n$ by cutting $c^{(n)}_1 =
  c^{(n-1)}_1(0) \subset X_n$ and then attaching them appropriately,
  one obtains $X_{n+1}$, as illustrated in
  Figure~\ref{figure:nth-cover-of-bd_n}.  Let $e$ be the one of the
  copies $c^{(n-1)}_2(0)$ in $X_{n+1}$ as shown in
  Figure~\ref{figure:nth-cover-of-bd_n}.  Define a character
  $\phi\colon \pi_1(X_{n+1}) \to \Z_d$ by assigning $-s\in\Z_d$ to $e$
  and $0$ to other 1-cells.

  \begin{figure}[ht]
    \begin{center}
      \includegraphics[scale=.8]{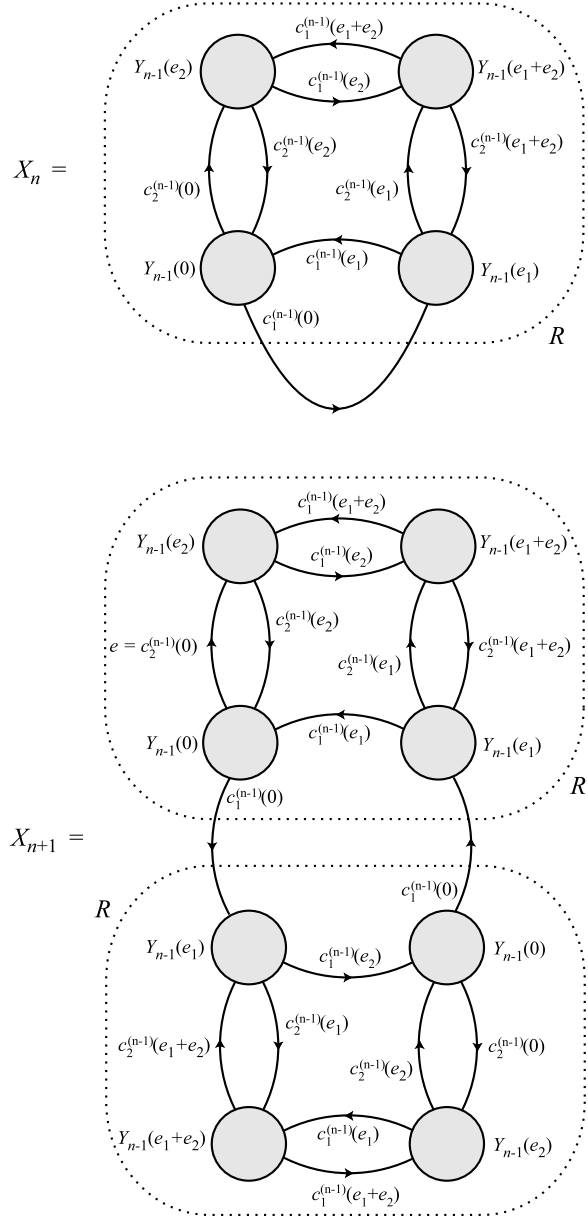}
    \end{center}
    \caption{The covers $X_{n}$ and $X_{n+1}$.}
    \label{figure:nth-cover-of-bd_n}
  \end{figure}

  \begin{figure}[ht]
    \begin{center}
      \includegraphics[scale=.65]{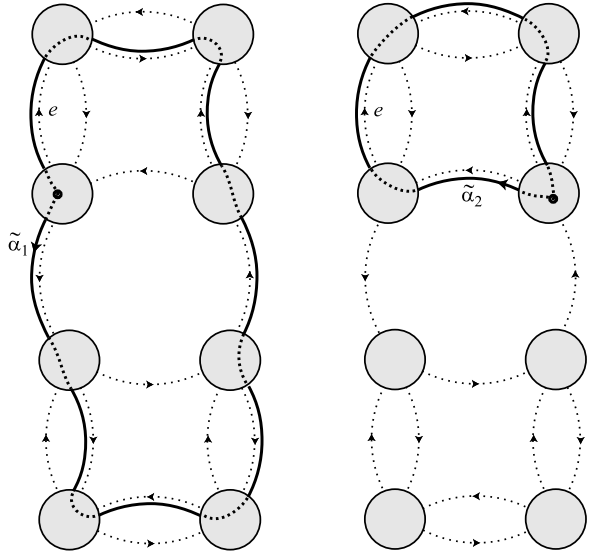}
    \end{center}
    \caption{Some lifts of $\alpha^2$ and $\alpha$ in $X_{n+1}$.}
    \label{figure:lifts-in-n+1st-cover-of-bd_n}
  \end{figure}


  For our purpose, it suffices to investigate loops
  $\tilde\alpha_j\in\mathcal{L}(\alpha,X_{n+1}|X_0)$ which are not
  null-homotopic in
  \[
  X'_{n+1} = (X_{n+1}\text{ with each } Y_{n-1}(-) \text{ in
    Figure~\ref{figure:nth-cover-of-bd_n} collapsed}),
  \]
  since those null-homotopic in $X'_{n+1}$ are in the kernel of $\phi$
  defined above.  Recall that, in
  Lemma~\ref{lemma:computation-of-lifts}, we gave a description of
  loops in $\mathcal{L}(\alpha,X_n|X_0)$ which are not null-homotopic
  in~$\bar X_n$.  Lifting those loops (or their square), one can
  easily sketch the $\tilde\alpha_j$ as curves in~$X'_{n+1}$.  From
  this it can be seen that two of the $\tilde\alpha_j$, say
  $\tilde\alpha_1$ and $\tilde\alpha_3$, are lifts of $\alpha^2$,
  i.e., $r_1=r_3=2$, and there are four other loops that are
  non-null-homotopic in $X'_{n+1}$, say
  $\tilde\alpha_2,\tilde\alpha_4,\tilde\alpha_5, \tilde\alpha_6$, with
  $r_2=r_4=r_5=r_6=1$.


  Also, it can be verified that exactly two of the $\tilde\alpha_j$
  pass through the 1-cell $e\subset X_{n+1}$; $\tilde\alpha_1$ passes
  through $e$ algebraically $-1$ times, and one of
  $\tilde\alpha_2,\tilde\alpha_4, \tilde\alpha_5, \tilde\alpha_6$, say
  $\tilde\alpha_2$, passes through $e$ algebraically $+1$ times.  We
  illustrate $\tilde\alpha_1$ and $\tilde\alpha_2$ in
  Figure~\ref{figure:lifts-in-n+1st-cover-of-bd_n}.  Therefore, it
  follows that $\phi(\tilde\alpha_1)=s$, $\phi(\tilde\alpha_3)=-s$,
  and all other elements in $\mathcal{L}(\alpha,X_{n+1}|X_0)$ are in
  the kernel of~$\phi$.
\end{proof}

\subsection{Bing doubles and the Levine-Tristram signature}
\label{subsection:bing-double-signature}

In this subsection, as a by-product of the proof of
Theorem~\ref{theorem:non-slice-torsion-bd_n}, we prove the following
generalization of the result of Harvey~\cite{Harvey:2006-1} and
Teichner:

\begin{theorem}
  \label{theorem:bd_n-and-signature}
    For any knot $K$ and any positive integer $n$, the Levine-Tristram
    signature function $\sigma_K$ is determined by $BD_n(K)$.  In
    particular, if $\sigma_K$ is nontrivial, then $BD_n(K)$ is not
    slice.
\end{theorem}

For concreteness, we recall a precise definition of the
Levine-Tristram signature function.  For a Seifert matrix $A$ of a
knot $K$ and $\omega\in S^1 \subset \C$,
\[
\sign \lambda_1(A,\omega) = \sign\big((1-\omega)A +
(1-\bar\omega)A^T\big)
\]
is often called the \emph{$\omega$-signature} of~$K$.  It is known
that if $K$ is algebraically slice and $\omega$ is not a zero of the
Alexander polynomial of $K$, then the $\omega$-signature vanishes.
Note that the zero set of the Alexander polynomial is not invariant
under concordance.  This leads us to think of the average
\[
\sigma_K(\omega) = \lim \frac{\sign \lambda_1(A,\omega_+) +\sign
  \lambda_1(A,\omega_-)}{2}
\]
as a concordance invariant, where $\omega_+,\omega_- \in S^1$ approach
$\omega$ from different sides.  The function $\sigma_K\colon S^1 \to
\Z$ is referred to as the Levine-Tristram signature function of~$K$.

\begin{proof}[Proof of Theorem~\ref{theorem:bd_n-and-signature}]
  We use the notations of the previous subsections: let $N$ be the
  zero-surgery manifold of $BD_n(K)$, which is obtained from the
  zero-surgery manifold $M$ of $BD_n$ by infection by $K$
  along~$\alpha$, and $X \to M$ be the meridian map such that
  $[\alpha]\in \pi_1(M)=\pi_1(X)$ is represented by the loop
  $x^{(n)}_1$ in~$X$.  We consider the 2-tower $X_n \to \cdots \to
  X_0=X$ constructed above.  We need the following analogue of
  Lemma~\ref{lemma:lift-behaviour-of-bd-infection-curve}:
  
  \begin{lemma}
    For any $d>0$ and $0 \le s <d$, there is a character
    $\varphi_n\colon \pi_1(X_n) \to \Z_d$ such that all
    $\tilde\alpha_{j}\in \mathcal{L}(\alpha,X_n|X)$ are in the kernel
    of $\varphi_n$ except two, which are sent to $s$ and $-s\in\Z_d$
    by $\varphi_n$, respectively.
  \end{lemma}
  
  \begin{proof}
    Recall that the 1-complex $X_n$ is as in
    Figure~\ref{figure:nth-cover-of-bd_n}.  Define $\phi\colon
    \pi_1(X_n) \to \Z_d$ by assigning $s\in\Z_d$ to the 1-cell
    $c^{(n-1)}_1(0)$, and $0\in\Z_d$ to other 1-cells of~$X_n$.

    By Lemma~\ref{lemma:computation-of-lifts}~(1), all lifts of
    $\alpha$ in $X_n$ are killed by $\phi$ except those based at $*\in
    Y_{k-1}(0)$ or $*\in Y_{k-1}(e_2)$, which we denote by
    $\tilde\alpha_1$ and~$\tilde\alpha_2$.  We illustrate
    $\tilde\alpha_1$ and $\tilde\alpha_2$ in
    Figure~\ref{figure:lifts-in-nth-cover-of-bd_n} (a dot on each
    $\tilde\alpha_i$ represents the point that $\tilde\alpha_i$ is
    based at).  Note that $\tilde\alpha_1$ and $\tilde\alpha_2$ meet
    the 1-cell $c^{(n-1)}_1(0)$ $+1$ and $-1$ times algebraically,
    respectively.  Since $s$ is assigned to $c^{(n-1)}_{1}(0)$, we
    have $\phi(\tilde\alpha_1)=s$ and $\phi(\tilde\alpha_2)=-s$.
    \begin{figure}[ht]
      \begin{center}
        \includegraphics[scale=.65]{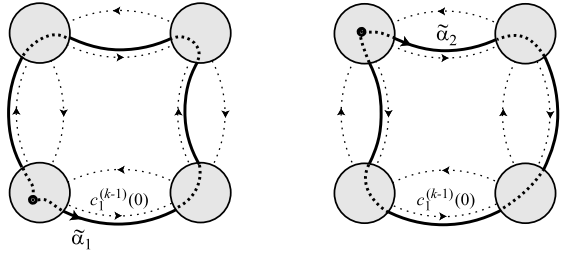}
      \end{center}
      \caption{Some lifts of $\alpha$ in $X_{n}$.}
      \label{figure:lifts-in-nth-cover-of-bd_n}
    \end{figure}
  \end{proof}

  We continue the proof of Theorem~\ref{theorem:bd_n-and-signature}.
  For a given $d=2^r$ and $0\le s<d$, let $\varphi_n\colon
  \pi_1(X_n)=\pi_1(M_n) \to \Z_d$ be a map given by the above lemma.
  Let $\psi_n\colon \pi_1(N_n) \to \Z_d$ be the map induced
  by~$\varphi_n$.  Since $BD_n$ is a trivial link, the invariant
  $\lambda(M_n,-)$ always vanishes by
  Theorem~\ref{theorem:slice-obstruction-from-covering-tower}.
  Therefore, by the above lemma and
  Corollary~\ref{corollary:lambda-of-infected-manifold}, we have
  \[
  \lambda(N_n,\psi_n) = [\lambda_1(A,\zeta_d^s)] +
  [\lambda_1(A,\zeta_d^{-s})] \in L^0(\Q(\zeta_d))
  \]
  where $A$ is a Seifert matrix of~$K$.  We apply the signature map
  $\sign\colon L^0(\Q(\zeta_{d})) \to \Z$ discussed in
  Section~\ref{section:homology-cobordism-invariants}.  Since
  $\lambda_1(A,\omega)$ and $\lambda_1(A,\omega^{-1})$ are the
  transpose of each other, they have the same signature.  So we have
  \[
  \sign \lambda(N_n,\psi_n)= 2\sign \lambda_1(A,\zeta_{d}^s).
  \]
  It is known that the $\omega$-signature has nontrivial jumps only at
  finitely many points on~$S^1$.  Since $\{\zeta_{2^r}^s \mid r,s \in
  \Z\}$ is a dense subset of $S^1$, it follows that $\sigma_K$ is
  determined by $\{\lambda(N_n,\psi_n)\mid r,s\in \Z\}$.  In
  particular, if $BD_n(K)$ is slice, then $\sigma_K$ is trivial by
  Theorem~\ref{theorem:slice-obstruction-from-covering-tower}.
\end{proof}

\section{Intersection form defects of $(n)$-solvable manifolds and
  links}
\label{section:obstruction-to-solvability}


Following \cite{Cochran-Orr-Teichner:2002-1}, we say that a closed
3-manifold $M$ is \emph{$(n)$-solvable} if there is a spin 4-manifold
$W$ with $\pi=\pi_1(W)$ satisfying the following: $\partial W=M$, the
inclusion induces an isomorphism $H_1(M) \cong H_1(W)$, and there are
elements
\[
u_1,\ldots,u_r, v_1,\ldots,v_r\in H_2(W;\Z[\pi/\pi^{(n)}])
\]
such that $2r=b_2(W)$ and $\lambda_n(v_i, v_j)=0$, $\mu_n(v_i)=0$,
$\lambda_n(v_i,u_j)=\delta_{ij}$ for any $i, j$, where $\lambda_n$ is
the $\Z[\pi/\pi^{(n)}]$-valued intersection pairing on
$H_2(W;\Z[\pi/\pi^{(n)}])$ and $\mu_n$ is the
$(\Z[\pi/\pi^{(n)}]/\text{involution})$-valued self-interesection
quadratic form on $H_2(W;\Z[\pi/\pi^{(n)}])$.  In this case $W$ is
called an \emph{$(n)$-solution of~$M$}.  A~link $L$ is
\emph{$(n)$-solvable} if the zero-surgery manifold of $L$ is
$(n)$-solvable.

\begin{remark}
  In this paper, all results on solvability are proved without using
  that a solution $W$ is spin nor that the self-intersection $\mu_n$
  vanishes on the~$v_i$.
\end{remark}

\subsection{Obstructions to being $(n)$-solvable}

In this subsection we show that our invariants from $p$-towers of
height $<n$ vanishes for $(n)$-solvable manifolds and links:

\begin{theorem}
  \label{theorem:Witt-invariant-for-solvable-manifold}
  Suppose $W$ is an $(n)$-solution of $M$ and $H_1(M)$ is $p$-torsion
  free.  Then the following holds:
  \begin{enumerate}
  \item The inclusion $M \to W$ is a $p$-tower map of
    height~$n$, in the sense of
    Definition~\ref{definition:t-equivalence}.
  \item For any $p$-tower
    \[
    M_{n-1} \to \cdots \to M_1 \to M_0 = M
    \]
    and for any $\phi\colon \pi_1(M_{n-1}) \to \Z_d$ with $d$ a power
    of $p$, $\lambda(M_{n-1},\phi)=0$ in $L^0(\Q(\zeta_d))$.
  \end{enumerate}
\end{theorem}

Theorem~\ref{theorem:solvablity-obstruction} is an immediate
consequence of
Theorem~\ref{theorem:Witt-invariant-for-solvable-manifold}.  

\begin{proof}[Proof of
  Theorem~\ref{theorem:Witt-invariant-for-solvable-manifold}~(1)]
  We use an induction on~$n$.  Fix $n$, and if $n>0$, then suppose that
  \[
  W_n \to W_{n-1} \to \cdots \to W_0=W
  \]
  is a $p$-tower of height $n$ for $W$, and let $M_i=\partial W_i$ for
  $i\le n$.  As we did in
  Section~\ref{section:homology-cobordism-invariants}, it suffices to
  show that $H_1(M_n;\Z_r)\to H_1(W_n;\Z_r)$ is an isomorphism for any
  power $r$ of~$p$.  Since $H_1(M;\Z_p)\cong H_1(W;\Z_p)$,
  $H_1(W,M;\Z_p)=0$.  By applying Levine's
  Lemma~\ref{lemma:levine's-lemma} inductively, we have
  $H_1(W_n,M_n;\Z_p)=0$, and so $H_1(W_n,M_n;\Z_r)=0$.  It follows
  that $H_1(M_n;\Z_r)\to H_1(W_n;\Z_r)$ is surjective.  So, it
  suffices to prove:
  
  \begin{assertion}
    \label{assertion:intersection-form-over-Z_p^a}
    $H_2(W_n;\Z_r) \to H_2(W_n,M_n;\Z_r)$ is an isomorphism.
  \end{assertion}

  To prove Assertion~\ref{assertion:intersection-form-over-Z_p^a}, we
  investigate the intersection form of~$W_n$.  Let $W^{(n)}$ be the
  cover of $W$ corresponding to the $n$th derived subgroup $\pi^{(n)}$
  of~$\pi$.  Since $W$ is an $(n)$-solution, there are
  \[
  u_1,\ldots,u_r, v_1,\ldots,v_r \in
  H_2(W;\Z[\pi/\pi^{(n)}])=H_2(W^{(n)})
  \]
  such that $\lambda_n(u_i,u_j)=0$ and
  $\lambda_n(u_i,v_j)=\delta_{ij}$ where $r=\frac12 b_2(W)$.  Note
  that (the $\pi^{(n)}$-coset of) an element $g\in \pi$ acts on
  $H_2(W^{(n)})$ via the covering transformation on~$W^{(n)}$.  (As a
  convention we assume that it is a right action.)

  Fixing basepoints of the covers $W_i$, we can regard $\pi_1(W_i)$ as
  a (possibly non-normal) subgroup of~$\pi$.  Since each $W_i \to
  W_{i-1}$ is an abelian cover, we have $\pi^{(n)} \subset
  \pi_1(W_n)$, that is, $W^{(n)}$ is a cover of $W_n$.  So there is an
  induced map $H_2(W^{(n)}) \to H_2(W_n)$.  Note that the index
  $s=[\pi:\pi_1(W_n)]$ is finite.  Choosing coset representatives
  $g_k\in \pi$, we write the right cosets of the subgroup $\pi_1(W_n)
  \subset \pi$ as $\pi_1(W_n)g_k$ for $1\le k\le s$.  Let $u_{ik}$,
  $v_{ik} \in H_2(W_n)$ be the images of $u_i\cdot g_k$, $v_i\cdot g_k
  \in H_2(W^{(n)})$, respectively, where $\mathstrut\cdot g_k$ denotes
  the action of $g_k$ on $H_2(W^{(n)})$.
  
  For a 4-manifold $V$, denote the untwisted $\Z$-valued
  intersection pairing on $H_2(V)$ by $I_V(-,-)$.

  \begin{assertion}
    \label{assertion:lagrangian-and-dual}
    $I_{W_n}(u_{ik}, u_{jl})=0$ and
    $I_{W_n}(u_{ik},v_{jl})=\delta_{ij}\delta_{kl}$ for any $i, j, k,
    l$.
  \end{assertion}
  
  \begin{proof}
    Note that
    \[
    0 = \lambda_n(u_i,u_j) = \sum_{g\pi^{(n)} \in \pi/\pi^{(n)}}
    I_{W^{(n)}}(u_i,u_j\cdot g) \cdot (g\pi^{(n)}) \quad\text{in }
    \Z_p[\pi/\pi^{(n)}].
    \]
    It follows that $I_{W^{(n)}}(u_i,u_j\cdot g)=0$ for any $g\in \pi$.
    Therefore
    \begin{align*}
      I_{W_n}(u_{ik},u_{jl}) &= \sum_{g\pi^{(n)} \in
        \pi_1(W_n)/\pi^{(n)}}
      I_{W^{(n)}}(u_i\cdot g_k, u_j \cdot g_l g) \\
      &= \sum_{g\pi^{(n)}\in \pi_1(W_n)/\pi^{(n)}} I_{W^{(n)}}(u_i ,
      u_j \cdot g_l^{\vphantom{-1}} g g_k^{-1}) = 0.
    \end{align*}
    Similarly, for $g\in\pi$,
    \[
    I_{W^{(n)}}(u_i,v_j\cdot g)=
    \begin{cases}
      1 & \text{if } i = j \text{ and } g \in \pi^{(n)} \\
      0 & \text{otherwise}
    \end{cases}
    \]
    and
    \[
    I_{W_n}(u_{ik}, v_{jl}) = \sum_{g\pi^{(n)}\in\pi_1(W_n)/\pi^{(n)}}
    I_{W^{(n)}}(u_i, v_j \cdot g_l^{\vphantom{-1}}g g_k^{-1}).
    \]
    Note that $g_l^{\vphantom{-1}} g g_k^{-1} \in \pi^{(n)}$ if and
    only if $g g_k^{-1}g_l^{\vphantom{-1}} \in \pi^{(n)}$ since
    $\pi^{(n)}$ is a normal subgroup.  If
    $gg_k^{-1}g_l^{\vphantom{-1}} \in \pi^{(n)}$ for $g\in
    \pi_1(W_n)$, then $k=l$ since $\pi^{(n)} \subset \pi_1(W_n)$, and
    consequently $g\in \pi^{(n)}$.  It follows that $I_{W_k}(u_{ik},
    v_{jl})=\delta_{ij} \delta_{kl}$.
  \end{proof}  

  We denote the $i$th Betti number of a space or a pair by $b_i(-)$,
  and the $\Z_p$-coefficient Betti number by $b_i(-;\Z_p) =
  \rank_{\Z_p}H_i(-;\Z_p)$.
  
  \begin{assertion}
    \label{assertion:b_2-of-W_n}
    $b_2(W_n)=2rs=b_2(W_n;\Z_p)$.
  \end{assertion}

  \begin{proof}
    First we consider the case of $n=0$.  Recall the assumption that
    $H_1(M)$ is $p$-torsion free.  Since $H_1(M)\cong H_1(W)$,
    $b_1(W;\Z_p)=b_1(W)$.  Since $M\to W$ induces an
    $H_1$-isomorphism, $b_3(W;\Z_p)=b_1(W,M;\Z_p)=0$ and
    $b_3(W)=b_1(W,M)=0$.  Since the $\Z_p$-coefficient Euler
    characteristic of $W$ is equal to the integral coefficient Euler
    characteristic, we have $b_2(W;\Z_p)=b_2(W)=2r$.

    For the case of $n\ge 1$, we need the following lemma, whose proof
    is postponed:
  
    \begin{lemma}
      \label{lemma:b_2-of-p-covering}
      Suppose $V$ is a compact 4-manifold.  Then for any connected
      regular $p^a$-fold cover $\tilde V$ of $V$, $b_2(\tilde V;\Z_p)
      \le p^a\cdot b_2(V;\Z_p)$.
    \end{lemma}
    
    Applying this lemma to $W_0,\ldots,W_n$ inductively, we have
    \[
    b_2(W_n) \le b_2(W_n;\Z_p) \le s \cdot b_2(W;\Z_p)=s \cdot b_2(W)
    = 2rs.
    \]
    By Assertion~\ref{assertion:lagrangian-and-dual}, $b_2(W_n) \ge
    2rs$.  From this Assertion~\ref{assertion:b_2-of-W_n} follows.
  \end{proof}
  
  Now we continue the proof of
  Assertion~\ref{assertion:intersection-form-over-Z_p^a}.  From
  Assertion~\ref{assertion:b_2-of-W_n} and the universal coefficient
  theorem
  \[
  H_2(W_n;\Z_p)=(H_2(W_n)\otimes \Z_p) \oplus
  \operatorname{Tor}(H_1(W_n), \Z_p),
  \]
  it follows that $H_2(W_n)$ is $p$-torsion free and $H_1(W_n)$ is
  torsion and $p$-torsion free.  So, for any power $r$ of $p$, we have
  $H_2(W_n;\Z_r)=(H_2(W_n)/\text{torsion})\otimes \Z_r$ by the
  universal coefficient theorem.  By a similar argument, we have
  $H_2(W_n,M_n;\Z_r)=(H_2(W_n,M_n)/\text{torsion})\otimes \Z_r$.  By
  Assertions \ref{assertion:lagrangian-and-dual}
  and~\ref{assertion:b_2-of-W_n}, the map 
  \[
  H_2(W_n)/\text{torsion} \to H_2(W_n,M_n)/\text{torsion}
  \]
  is an isomorphism.  From this
  Assertion~\ref{assertion:intersection-form-over-Z_p^a} follows.
  This completes the proof of
  Theorem~\ref{theorem:Witt-invariant-for-solvable-manifold}~(1).
\end{proof}

\begin{proof}
  [Proof of
  Theorem~\ref{theorem:Witt-invariant-for-solvable-manifold} (2)]

  By Theorem~\ref{theorem:Witt-invariant-for-solvable-manifold} (1),
  there is a $p$-tower
  \[
  W_{n-1} \to \cdots \to W_1 \to W_0 = W
  \]
  such that $\partial W_i = M_i$ and $\Hom(\pi_1(W_{n-1}),\Z_d)
  \approx \Hom(\pi_1(M_{n-1}),\Z_d)$.  Thus the given $\phi\colon
  \pi_1(M_{n-1}) \to \Z_d$ extends to $\psi\colon \pi_1(W_{n-1}) \to
  \Z_d$.  From this it follows that $\lambda(M_{n-1},\phi)$ is
  well-defined as an element in $L^0(\Q(\zeta_d))$.

  We will show that $\lambda(M_{n-1},\phi)$ vanishes by investigating
  intersection form of the bounding 4-manifold~$W_{n-1}$.  By applying
  Assertions \ref{assertion:lagrangian-and-dual}
  and~\ref{assertion:b_2-of-W_n} of the proof of
  Theorem~\ref{theorem:Witt-invariant-for-solvable-manifold} (1) to
  $n-1$, it follows that the untwisted intersection form on
  $H_2(W_{n-1};\Q)$ is metabolic and so the ordinary signature
  $\sigma(W_{n-1})$ vanishes.

  It remains to show that the intersection form
  $\lambda^{\Q(\zeta_d)}_{W_{n-1}}$ on $H_2(W_{n-1};\Q(\zeta_d))$ is
  Witt trivial.  Indeed, we can construct a ``Lagrangian'' and its
  dual for $\lambda^{\Q(\zeta_d)}_{W_{n-1}}$, as we did for the
  untwisted intersection form $I_{W_n}$ in
  Assertion~\ref{assertion:lagrangian-and-dual} of the proof of
  Theorem~\ref{theorem:Witt-invariant-for-solvable-manifold}~(1).
  Details are as follows.  Let $u_i, v_i \in
  H_2(W;\Z[\pi/\pi^{(n)}])=H_2(W^{(n)})$, where $1 \le i \le r =
  \frac12 b_2(W)$, be the elements such that $\lambda_n(u_i,u_j)=0$
  and $\lambda_n(u_i,v_j)=\delta_{ij}$ as before.  Let
  $\pi_1(W_{n-1})g_k$ be the right cosets of $\pi_1(W_{n-1}) \subset
  \pi=\pi_1(W)$, where $1\le k \le s=[\pi:\pi_1(W_{n-1})]$.  Let
  $u_{ik}$, $v_{jl} \in H_2(W_n)$ be the images of $u_i\cdot g_k$,
  $v_j\cdot g_l$.  Similarly to
  Assertion~\ref{assertion:lagrangian-and-dual} of the proof of
  Theorem~\ref{theorem:Witt-invariant-for-solvable-manifold} (1), for
  any $h\in \pi_1(W_{n-1})$ we have
  \begin{align*}
    I_{W_n}(u_{ik},u_{jl}\cdot h) &= \sum_{g\pi^{(n)} \in
      \pi_1(W_n)/\pi^{(n)}} I_{W^{(n)}} (u_i, u_j\cdot
    g^{\vphantom{-1}}_\ell h g g_k^{-1}) = 0,
    \\
    I_{W_n}(u_{ik},v_{jl}\cdot h) & = \sum_{g\pi^{(n)} \in
      \pi_1(W_n)/\pi^{(n)}} I_{W^{(n)}} (u_i, v_j\cdot
    g^{\vphantom{-1}}_\ell h g g_k^{-1})
    \\
    &=
    \begin{cases}
      1 & \text{if $i=j$, $k=l$, and $h\in \pi_1(W_n)$} \\
      0 & \text{otherwise}.
    \end{cases}
  \end{align*}
  To obtain the last equality, observe that the concerned summand can
  be nonzero if and only if $hgg_k^{-1}g_l^{\vphantom{-1}}\in
  \pi^{(n)}$; if it is the case, then $k$ should be equal to $l$ since
  since $h,g\in \pi_1(W_{n-1})$, and consequently $g$ should be in the
  coset $h^{-1}\pi^{(n)}$.  Since $g\in \pi_1(W_n)$, this can occur
  only when $h\in\pi_1(W_n)$.

  Let $\lambda^{\Z[\Z_d]}_{W_{n-1}}$ be the $\Z[\Z_d]$-valued
  intersection form on $H_2(W_{n-1};\Z[\Z_d]) = H_2(W_n)$, where
  $\Z_d$ is identified with $\pi_1(W_{n-1})/\pi_1(W_n) =$ the covering
  transformation group of $W_{n} \to W_{n-1}$.  From the above
  computation, it follows that
  \begin{align*}
    \lambda^{\Z[\Z_d]}_{W_{n-1}}(u_{ik}, u_{jl}) &= 0,\\
    \lambda^{\Z[\Z_d]}_{W_{n-1}}(u_{ik}, v_{jl}) &= \sum_{h\pi_1(W_n)
      \in \frac{\pi_1(W_{n-1})}{\pi_1(W_n)}} I_{W_n}(u_{ik}, v_{ik}\cdot h)
    \cdot h\pi_1(W_n) = \delta_{ij}\delta_{kl}.
  \end{align*}
  Therefore, by naturality, the values of
  $\lambda^{\Q(\zeta_d)}_{W_{n-1}}$ evaluated at the image of
  $(u_{ik}, u_{jl})$ and $(u_{ik}, v_{jl})$ are $0$ and
  $\delta_{ij}\delta_{kl}$, respectively, for $1\le i \le r$, $1\le
  k\le s$.

  Now, in order to conclude that $\lambda^{\Q(\zeta_d)}_{W_{n-1}}$ is
  Witt trivial, it suffices to show that
  $b_2(W_{n-1};\Q(\zeta_d))=2rs$.  From the properties of the $u_{ik},
  v_{jl}$ proved above, it follows that $b_2(W_{n-1};\Q(\zeta_d))\ge
  2rs$.  For the opposite inequality, we appeal to the following
  analogue of Lemma~\ref{lemma:b_2-of-p-covering}, which will be
  proved later:

  \begin{lemma}
    \label{lemma:twisted-b_2-of-p-covering}
    Suppose $V$ is a 4-manifold and $\pi_1(V) \to \Gamma$ is a map,
    where $\Gamma$ is a $p$-group endowed with a map $\Z\Gamma\to \K$
    into a (skew-)field with characteristic zero.  If the induced map
    $\Z\pi_1(V) \to \Z\Gamma\to \K$ is nontrivial, then $b_2(V;\K) \le
    b_2(V;\Z_p)$.
  \end{lemma}

  If $\phi\colon\pi_1(M_{n-1}) \to \Z_d$ is a trivial map, then
  $\lambda(M_{n-1},\phi)=0$.  So we may assume that $\phi$ is
  nontrivial.  Then $\Z\pi_1(W_{n-1}) \to \Z[\Z_d] \to \Q(\zeta_d)$ is
  nontrivial since its composition with $\Z\pi_1(M_{n-1}) \to
  \Z\pi_1(W_{n-1})$ is nontrivial.  From this and
  Assertion~\ref{assertion:b_2-of-W_n} of the proof of
  Theorem~\ref{theorem:Witt-invariant-for-solvable-manifold} (1) for
  $n-1$, it follows that
  \[
  b_2(W_{n-1};\Q(\zeta_d))\le b_2(W_{n-1};\Z_p) = 2rs
  \]
  by Lemma~\ref{lemma:twisted-b_2-of-p-covering}.
\end{proof}

\begin{proof}[Proof of Lemma~\ref{lemma:b_2-of-p-covering}]
  This is a $p$-covering analogue of previously known results for
  quotient (skew-)field coefficient homology modules of
  poly-torsion-free-abelian covers, which were dealt with
  in~\cite{Cochran-Orr-Teichner:2002-1} for the special case that
  $H_1(\partial V) \to H_1(V)$ is an isomorphism, and in Proposition
  2.1 of \cite{Cha:2006-1} for the general case.  Since our statement
  can also be proved along the same lines, we will discuss how the
  proof of Proposition 2.1 of \cite{Cha:2006-1} is modified.  As an
  analogue of Lemma 2.5 of \cite{Cha:2006-1}, we need the following:
  suppose $(X,A)$ is a finite CW-pair with $X$ connected, $\tilde X$
  is the cover of $X$ induced by $\phi\colon \pi_1(X) \to \Gamma$
  where $\Gamma$ is a $p$-group, and $\tilde A\subset \tilde X$ is the
  pre-image of~$A$.
  \begin{enumerate}
  \item If $A$ is nonempty, then $b_1(\tilde X,\tilde A;\Z_p) \le
    |\Gamma|\cdot b_1(X,A;\Z_p)$.
  \item If $\phi$ is surjective, then $b_1(\tilde X;\Z_p) \le |\Gamma|
    \cdot (b_1(X;\Z_p)-1)+1$.
  \end{enumerate}
  The proof of (1), (2) is exactly the same as that of Lemma 2.5 of
  \cite{Cha:2006-1}, except that we should use Levine's
  Lemma~\ref{lemma:levine's-lemma} in place of Lemma 2.3 (2) of
  \cite{Cha:2006-1}.  Using the above (1), (2) in place of Lemma 2.5
  of \cite{Cha:2006-1}, the argument of the proof of \cite[Proposition
  2.1]{Cha:2006-1} proves our Lemma~\ref{lemma:b_2-of-p-covering}.
\end{proof}

\begin{proof}[Proof of Lemma~\ref{lemma:twisted-b_2-of-p-covering}]
  We proceed similarly to the proof of
  Lemma~\ref{lemma:b_2-of-p-covering}, along the same lines as the
  proof of Proposition 2.1 of \cite{Cha:2006-1}.  In this case we need
  the following analogue of Lemma 2.5 of \cite{Cha:2006-1}:
  \begin{enumerate}
  \item If $A$ is nonempty, then $b_1(X,A;\K) \le
    b_1(X,A;\Z_p)$.
  \item If $\phi$ is nontrivial, then $b_1(X;\K) \le b_1(X;\Z_p)-1$.
  \end{enumerate}
  The proof of (1) proceeds as follows.  As in the proof of Lemma 2.5
  of~\cite{Cha:2006-1}, we consider an inclusion
  \[
  (Y,B)=\bigcup^{b_1(X,A;\Z_p)} ([0,1],\{0,1\}) \hookrightarrow (X,A)
  \]
  such that $Y\cap A = B$ and $H_1(X,Y\cup A;\Z_p) = 0$.  By Levine's
  Lemma~\ref{lemma:levine's-lemma}, $H_1(X,Y\cup A;\Z_{(p)}\Gamma)=0$.
  It follows that $H_1(X,Y\cup A;\K)=0$, and so $b_1(X,A;\K) \le
  b_1(Y,B;\K) \le b_1(X,A;\Z_p)$.  (2)~is proved exactly as in the
  proof of Lemma 2.5 of~\cite{Cha:2006-1}.  Now, the arguments of the
  proof of Proposition 2.1 of \cite{Cha:2006-1} can be applied to prove
  Lemma~\ref{lemma:twisted-b_2-of-p-covering}.
\end{proof}

\subsection{Iterated Bing doubles and solvable and grope filtrations}

In this subsection we show that there are nontrivial torsion elements
(as well as infinite order elements) in an arbitrary depth of the
solvable filtration and grope filtration of link concordance.  Recall
that in Theorems~\ref{theorem:non-slice-torsion-bd_n}
and~\ref{theorem:bd_n-and-signature} we proved that the iterated Bing
doubles $BD_n(K)$ of certain knots $K$ are not slice.  Note that we
used 2-towers of height $n+1$ and $n$ to prove
Theorems~\ref{theorem:non-slice-torsion-bd_n}
and~\ref{theorem:bd_n-and-signature}, respectively.  Therefore, by
Theorem~\ref{theorem:Witt-invariant-for-solvable-manifold}, we
immediately obtain the following non-solvability results: (1)~For the
amphichiral knots $K_{a_i}$ considered in
Theorem~~\ref{theorem:non-slice-torsion-bd_n}, $BD_n(K_{a_i})$ is
2-torsion but $BD_n(K_{a_i})$ is not $(n+2)$-solvable for any~$n$.
(2)~If $\sigma_K$ is nontrivial, then $BD_n(K)$ is not
$(n+1)$-solvable for any~$n$.

Modifying the construction slightly, we can easily obtain
\emph{$(n)$-solvable} examples:

\begin{theorem}
  \label{theorem:highly-nonsolvable-bd_n}
  \mbox{}\Nopagebreak
  \begin{enumerate}
  \item There are infinitely many amphichiral knots $K$ such that
    $BD_n(K)$ is 2-torsion and $(n)$-solvable but not $(n+2)$-savable
    for any~$n$.
  \item If $\sigma_K$ is nontrivial and $\Arf(K)=0$, then $BD_n(K)$ is
    $(n)$-solvable but not $(n+1)$-solvable for any~$n$.
  \end{enumerate}
\end{theorem}

Here $\Arf(K)$ is the Arf invariant of a knot~$K$.

\begin{proof}
  (1) For the solvability part, we need the following result of
  Cochran-Orr-Teichner
  \cite[Proposition~3.1]{Cochran-Orr-Teichner:2002-1}: if $K$ is a
  knot with $\Arf(K)=0$ and $L$ is obtained by infection on an
  $(n)$-solvable link $L_0$ by $K$ along a curve $\alpha$ such that
  $[\alpha]\in\pi_1(S^3-L_0)^{(n)}$, then $L$ is $(n)$-solvable.  (In
  \cite{Cochran-Orr-Teichner:2002-1} they stated the result when $L_0$
  is a knot, but their argument works for links as well.)

  Choose any $i\ne j$ and let $K=K_{a_i}\# K_{a_j}$ where the
  $K_{a_i}$ are as in Theorem~\ref{theorem:non-slice-torsion-bd_n}.
  Since both $K_{a_i}$ and $K_{a_j}$ are amphichiral, so is $K$.  So,
  by Lemma~\ref{lemma:torsion-property-of-bd_n}, $BD_n(K)$ has the
  desired 2-torsion property.  Due to Levine~\cite{Levine:1966-1},
  $\Arf(K)=1$ if and only if $\Delta_K(-1)=\pm 3\pmod 8$.  Since
  \[
  \Delta_{K_{a_i}}(-1)=-(4a_i^2+1) \equiv 3 \pmod 8
  \]
  (recall that the $a_i$ are all odd), $\Arf(K_{a_i})=1$ for all~$i$.
  Since $\Arf$ is additive under connected sum, $\Arf(K)=1+1=0$.  From
  this it follows that $BD_n(K)$ is $(n)$-solvable.

  In order to show the non-solvability of $BD_n(K)$, recall that in
  the proof of Theorem~\ref{theorem:non-slice-torsion-bd_n} we
  considered the discriminant of the invariant
  $\lambda(N_{n+1},\psi_{n+1})$ where $N_{n+1}$ is the $(n+1)$st term
  of a height $(n+1)$ 2-tower of the zero-surgery manifold of
  $BD_n(K)$ and $\psi_{n+1}$ is a $\Z_4$-valued character of
  $\pi_1(N_{n+1})$.  Since $\dis \lambda(N_{n+1},\psi_{n+1})$ is
  determined by the Alexander polynomial of $K$ as in the proof of
  Theorem~\ref{theorem:non-slice-torsion-bd_n} (see also
  Lemma~\ref{lemma:discriminant-of-Seifert-matrix}) and the Alexander
  polynomial is multiplicative under connected sum, $\dis
  \lambda(N_{n+1},\psi_{n+1})$ for our $K$ is equal to the product of
  those for $K_{a_i}$ and~$K_{a_j}$.  So, for our $K$ we have
  \[
  \lambda(N_{n+1},\psi_n) = (2a_i^2+1)(2a_i^4+4a_i^2+1) \cdot
  (2a_j^2+1)(2a_j^4+4a_j^2+1).
  \]
  by the computation for $K_{a_i}$ done in the proof of
  Theorem~\ref{theorem:non-slice-torsion-bd_n}.  Since $i\ne j$, by
  applying the norm residue symbol $(\mbox{}\cdot\mbox{},-1)_{p_i}$ or
  $(\mbox{}\cdot\mbox{},-1)_{p_j}$ which was computed in
  Proposition~\ref{proposition:dual-primes}, it follows that
  $\lambda(N_{n+1},\psi_n)$ is nontrivial.

  (2) From the Arf invariant condition, it follows immediately that
  $BD_n(K)$ is $(n)$-solvable by
  \cite[Proposition~3.1]{Cochran-Orr-Teichner:2002-1}.  We already
  discussed the non-solvability.
\end{proof}

For the iterated Bing doubles $BD(K_{a_i})$ in
Theorem~\ref{theorem:highly-nonsolvable-bd_n}, Harvey's invariant
$\rho_k$ vanishes (for all $k$) as mentioned in
Remark~\ref{remark:vanishing-of-signatures-for-bd_n}.  Also, there are
infinitely many knots $K$ such that $\sigma_K$ is nontrivial but the
integral of $\sigma_K$ over the unit circle is zero.  So, we have the
following consequence:

\begin{corollary}
  \label{corollary:n-solvable-links-with-vanishing-harvey-invariants}
  \leavevmode\Nopagebreak
  \begin{enumerate}
  \item For any $n$, there are infinitely many $(n)$-solvable
    2-torsion links which are not $(n+2)$-solvable but have vanishing
    $\rho_k$-invariants.
  \item For any $n$, there are infinitely many $(n)$-solvable links
    which are not $(n+1)$-solvable but have vanishing
    $\rho_k$-invariants.
  \end{enumerate}
\end{corollary}

From this it follows that the invariant $\rho_n$, which is viewed as a
homomorphism of ``$(n)$-solvable boundary string links modulo
$(n+1)$-solvable boundary string links'' as in
\cite{Harvey:2006-1,Cochran-Harvey-Leidy:2006-1}, has nontrivial
kernel for any~$n$.  We remark that Cochran, Harvey, and Leidy have
announced a (different) proof of the nontriviality of the kernel
of~$\rho_n$.

In a subsequent paper \cite{Cha:2007-2}, it will be shown that the
links in
Corollary~\ref{corollary:n-solvable-links-with-vanishing-harvey-invariants}
(1) and (2) (can be chosen so that they) are independent modulo
$\F_{(n+2)}$ and $\F_{(n+1)}$, respectively, in an appropriate sense.
In fact, considering the subgroup generated by these links, it can be
proved that the kernel of $\rho_n$ contains a subgroup whose
abelianization is isomorphic to $\Z^\infty$.

In~\cite{Cochran-Orr-Teichner:1999-1, Cochran-Orr-Teichner:2002-1,
  Cochran-Teichner:2003-1, Harvey:2006-1}, another filtration
$\{\G_{n}\}$ of the set of link concordance classes is defined in
terms of \emph{gropes}, instead of the solvability; $\G_{(n)}$ is
defined to be the subset of concordance classes of links in $S^3$
which bound an embedded symmetric grope of height $n$ in the 4-ball.
(As an abuse of notation, we will write $L\in \G_{(n)}$ when the
concordance class of $L$ is in~$\G_{(n)}$.)  $\{\G_{(n)}\}$ is called
the \emph{grope filtration}.  For a precise definition of a grope and
related discussions, the reader is referred
to~\cite{Cochran-Orr-Teichner:1999-1,Cochran-Teichner:2003-1}.  The
following result on the existence of torsion elements in the grope
filtration is a consequence of
Theorem~\ref{theorem:highly-nonsolvable-bd_n}:

\begin{corollary}
  \label{corollary:nontrivial-elts-in-group-filtration}
  \leavevmode\Nopagebreak
  \begin{enumerate}
  \item There are infinitely many knots $K$ such that $BD_n(K)$ is
    2-torsion and $BD_n(K) \in \G_{(n+1)}$ but $BD_n(K) \notin
    \G_{(n+4)}$ for any~$n$.
  \item If $\sigma_K$ is nontrivial and $\Arf(K)=0$, then $BD_n(K)\in
    \G_{(n+1)}$ but $BD_n(K)\notin \G_{(n+3)}$ for any~$n$.
  \end{enumerate}
\end{corollary}

As in
Corollary~\ref{corollary:n-solvable-links-with-vanishing-harvey-invariants},
the iterated Bing doubles in
Corollary~\ref{corollary:nontrivial-elts-in-group-filtration} can be
chosen so that the invariant $\rho_k$ vanishes.

\begin{proof}
  In \cite{Cochran-Orr-Teichner:1999-1}, it was shown that $\G_{(n+2)}
  \subset \F_{(n)}$.  So, the iterated Bing doubles considered in
  Theorem~\ref{theorem:highly-nonsolvable-bd_n} (1) and (2) are not in
  $\G_{(n+4)}$ and $\G_{(n+3)}$, respectively.

  Recall that $BD_n(K)$ is obtained by performing infection on the
  trivial link $BD_n$ along a curve~$\alpha$ shown in
  Figure~\ref{figure:bd_n-infection-curve}.  It is well known that
  $\alpha$ bounds an embedded symmetric grope of height $n$ in
  $S^3-BD_n$.  So, by the argument of \cite[proof of Theorem
  6.13]{Harvey:2006-1}, it follows that $BD_n(K)\in \G_{(n+1)}$ if
  $\Arf(K)=0$.  This completes the proof.
\end{proof}

Finally we remark that $BD_n(K)$ is a boundary link for any knot $K$,
so that our results hold in the solvable and grope filtrations of
boundary links.

\appendix

\section*{Appendix: Computation for zero-surgery manifolds of knots}
\label{section:computation-for-knot-surgery-manifold}

In this appendix we compute the intersection form defect invariants of
cyclic covers of the zero-surgery manifold $M$ of a knot~$K$.  We use
the same notation as in Section~\ref{subsection:knot-infection}: let
$X_r$ be the $r$-fold cyclic cover of $M$ which is determined by the
canonical map $\pi_1(M) \to \Z_r$ sending a (positive) meridian of $K$
to $1\in \Z_r$.  The image of the natural map $\pi_1(X_r) \to \pi_1(M)
\to H_1(M) = \Z$ is~$r\Z$.  Composing it with $r\Z \to \Z_d$ sending
$r\in r\Z$ to $s\in \Z_d$, we obtain a character
$\phi_r^{s,d}\colon\pi_1(X_r) \to \Z_d$ sending the lift of the $r$th
power of a meridian of $K$ to $s\in\Z_d$.  The following result was
stated as Lemma~\ref{lemma:lambda-of-surgery-manifolds} in
Section~\ref{subsection:knot-infection}:

\theoremstyle{plain}
\newtheorem*{lemma-nonum}{Lemma}

\begin{lemma-nonum}
  $(X_r,\phi_r^{s,d})$ is null-bordant over $\Z_d$, and
  \[
  \lambda(X_r,\phi_r^{s,d})=[\lambda_{r}(A,\zeta_d^s)]-[\lambda_{r}(A,1)]
  \quad\text{in } L^0(\Q(\zeta_d))
  \]
  where $A$ is a Seifert matrix of $K$ and $[\lambda_{r}(A,\omega)]$
  is the Witt class of (the nonsingular part of) the hermitian form
  represented by the following $r\times r$ block matrix:
  \[
  \lambda_{r}(A,\omega) =
  \begin{bmatrix}
    \vphantom{\ddots} A+A^T & -A & & & -\omega^{-1} A^T\\
    \vphantom{\ddots} -A^T & A+A^T & -A \\
    \vphantom{\ddots} & -A^T & A+A^T & \ddots \\
    \vphantom{\ddots} & & \ddots & \ddots & -A \\
    \vphantom{\ddots} -\omega A & & & -A^T & A+A^T
  \end{bmatrix}_{r\times r}
  \]
  For $r=1,2$, $\lambda_{r}(A,\omega)$ should be understood as
  \[
  \begin{bmatrix}
    (1-\omega)A+(1-\omega^{-1})A^T
  \end{bmatrix}
  \quad\text{and}\quad
  \begin{bmatrix}
    A+A^T & -A-\omega^{-1}A^T \\
    -A^T-\omega A & A+A^T
  \end{bmatrix}.
  \]
\end{lemma-nonum}

\begin{proof}
  Let $V$ be the 4-manifold obtained by attaching a 2-handle to $D^4$
  along the zero-framing of~$K\subset S^3$.  Then $\partial V=M$.
  Consider a Seifert surface of $K$ from which the Seifert matrix $A$
  is defined.  By capping it off using a disk in the boundary of the
  2-handle of $V$, we obtain a closed surface in $M$, which is usually
  called a ``capped-off Seifert surface''.  Pushing it slightly into
  the interior of $V$, we obtain a surface $F$ in $V$ with trivial
  normal bundle.  In fact the trace of pushing induces a framing of
  the normal bundle of~$F$.  We identify a tubular neighborhood of $F$
  with $F\times D^2$ using this framing.  Let $W = V - \inte( F\times
  D^2)$.  Obviously $\partial W=M \cup (F\times D^2)$.  By a
  Thom-Pontryagin construction along the trace of pushing $\cong
  F\times [0,1] \subset W$, $W$ can be viewed as a space over $\Z$.
  Let $W_r$ be the $r$-fold cyclic cover associated to $\pi_1(W) \to
  \Z \to \Z_r$.

  $H_1(W)\cong \Z$ and is generated by a meridian of~$F$.  Therefore
  the canonical map $\pi_1(M) \to H_1(M)=\Z$ is the restriction of
  $\pi_1(W) \to H_1(W)=\Z$.  This enables us to define
  $\psi_r^{s,d}\colon \pi_1(W_r) \to \Z_d$ exactly in the same way as
  the definition of $\phi_r^{s,d}\colon\pi_1(X_r)\to \Z_d$, such that
  $\phi_r^{s,d}$ is identical with
  \[
  \pi_1(X_r) \to \pi_1(W_r) \xrightarrow{\psi_r^{s,d}} \Z_d.
  \]
  Obviously $\partial W_r = X_r
  \cup (F\times S^1)$ over~$\Z_d$.  Here $F\times S^1$ is endowed with
  \[
  \phi'\colon \pi_1(F\times S^1) \xrightarrow{proj.} \pi_1(S^1)=\Z
  \xrightarrow{proj.} \Z_d.
  \]

  We claim that $\lambda(F\times S^1,\phi')=0$.  For, choosing a
  handlebody $H$ bounded by $F$, it can be seen that $\partial
  (H\times S^1)=F\times S^1$ over $\Z_d$.  Since $H_2(F\times
  S^1;R)\to H_2(H\times S^1;R)$ is surjective for $R=\Q$ and
  $\Q(\zeta_d)$, both $[\lambda_{\Q(\zeta_d)}(H\times S^1)]$ and
  $\sigma(H\times S^1)$ vanish.  This proves the claim.

  From the claim it follows that
  \[
  \lambda(X_r,\phi_r^{s,d}) = [\lambda_{\Q(\zeta_d)}(W_r)] -
  [i^*_{\Q(\zeta_d)}\sigma(W_r)].
  \]
  To compute the intersection form of $W_r$ from the Seifert matrix
  $A$ of $K$, we use a known cut-paste construction of~$W_r$ (e.g.,
  similar arguments were used in \cite{Kauffman:1987-1} and
  \cite{Cha-Ko:1999-1} to compute some signature invariants).  Details
  are as follows.  Let $Y$ be the manifold obtained by cutting $W$
  along the trace of pushing of the capped-off Seifert surface $F$.
  Obviously there are inclusion maps $i_{\pm}\colon F\times[0,1] \to
  \partial Y$ corresponding to the positive and negative normal
  directions of $F$ such that $i_+(F\times[0,1])$ and
  $i_-(F\times[0,1)]$ are disjoint and
  \[
  W = Y / i_+(z) \sim i_-(z) \text{ for } z \in F\times[0,1].
  \]
  The covering $W_r$ is obtained by gluing $r$ disjoint copies $t^0 Y,
  tY, \ldots, t^{r-1}Y$ of~$Y$:
  \[
  W_r = \Big( \bigcup_{k=0}^{r-1} t^k Y \Big) \Big/ t^{k+1}i_+(z) \sim
  t^k i_-(z) \text{ for } z \in F\times[0,1],\, k=0,\ldots,r-1
  \]
  where $t^{r}$ is understood as~$t^0$.  From this we have a
  Mayer-Vieotoris long exact sequence:
  \[
  \to \bigoplus^r H_2(Y) \to H_2(W_r) \to \bigoplus^r H_1(F) \to
  \bigoplus^r H_1(Y) \to
  \]
  Since $Y$ can also be obtained by cutting $D^4$ along the trace of
  pushing of $F$, it can be seen that $Y$ is homeomorphic to $D^4$.
  It follows that
  \[
  H_2(W_r) \cong \bigoplus^r H_1(F).
  \]
  For a 1-cycle $x$ on $F$, the corresponding element in $H_2(W_r)$ is
  described as follows.  Since $Y$ is contractible, there are 2-chains
  $u_+$ and $u_-$ in $Y$ such that $\partial u_{\pm}$ is equal to
  $i_\pm(x\times *)$, which is a pushoff of $x$ along the $\pm$-normal
  direction of the Seifert surface.  Then the homology class $x^k$ of
  the 2-chain $t^{k+1}u_+ \cup t^k u_-$ in $W_r$ corresponds the class
  of $x$ in the $k$th $H_1(F)$ factor.  For another 1-cycle $y$,
  choosing $v_+$ and $v_-$ in $X$ such that $\partial v_{\pm} =
  i_\pm(y\times 1)$, the intersection number $x_k \cdot y_\ell$ in
  $W_r$ is given by
  \[
  \begin{cases}
    (u_+ \cdot v_+) + (u_- \cdot v_-) & \text{if }k=\ell,\\
    u_+ \cdot v_- & \text{if }k=\ell-1,\\
    u_- \cdot v_+ & \text{if }k=\ell+1,\\
    0 & \text{otherwise}.
  \end{cases}
  \]
  By the definition, $u_+ \cdot v_-$ is exactly the value of the
  Seifert form on $(x,y)$, and the other terms in the above formula
  can also be interpreted similarly.  (A technical issue is that in
  computing the expression for $k=\ell$, one needs to push one of
  $u_\pm$ and $v_\pm$ further along the $\pm$-direction to remove
  intersection points on the boundary of the 2-chains; Pushing $u_\pm$,
  one can see that $u_+ \cdot v_+$ and $u_- \cdot v_-$ are the Seifert
  form evaluated at $(x,y)$ and $(y,x)$.)

  From this it follows that the intersection form on $H_2(W_r)$ is
  given by the block matrix $\lambda_r(A,1)$.  So $\sigma(W_r)$ is the
  Witt class of $\lambda_r(A,1)$.

  To compute the intersection form $\lambda_{\Q(\zeta_d)}(W_r)$, we
  consider $W_{dr}$; indeed
  \[
  H_2(W_r;\Z[\Z_d])=H_2(W_{dr}).
  \]
  The above argument shows that $H_2(W_{dr})$ is the direct sum of
  $dr$ copies of $H_1(F)$.  It can be seen easily that covering
  transformation action of a generator of $\Z_d$, say $g$, is exactly
  sending the $k$th $H_1(F)$ factor of $H_2(W_{dr})$ to the $(k+r)$th
  factor.  Therefore,
  \[
  H_2(W_r;\Z[\Z_d])\cong H_2(W_r)\otimes_\Z \Z[\Z_d].
  \]
  The above computation also shows that the $\Z[\Z_d]$-intersection on
  $H_2(W_r;\Z[\Z_d])$ is represented by the matrix~$\lambda_r(A,g)$.
  Note that $\Q(\zeta_d)$ is $\Q[\Z_d]$-projective, being an
  irreducible factor of the regular representation $\Q[\Z_d]$
  of~$\Z_d$.  Therefore the universal coefficient theorem gives
  \[
  H_2(W_r;\Q(\zeta_d))=H_2(W_r;\Z[\Z_d])\otimes_{\Z[\Z_d]}
  \Q(\zeta_d).
  \]
  It follows that the intersection form on $H_2(W_r;\Q(\zeta_d))$ is
  represented by $\lambda_r(A,\zeta_d^s)$.  This completes the
  computation of $\lambda(X_r,\phi_r^{s,d})$.
\end{proof}

\bibliographystyle{amsplainabbrv}
\renewcommand{\MR}[1]{}

\bibliography{research}

\end{document}